\documentclass[11pt,english]{smfart}
 \usepackage[english,francais]{babel} 
\usepackage{amscd}
\usepackage{amssymb,url,xspace,smfthm}
\usepackage{amsbsy,bm} 
\usepackage{paralist}
\usepackage{datetime} 
\usepackage{colortbl}
\usepackage{color}
\usepackage[dvipsnames]{xcolor}
\usepackage{mathrsfs}  

\usepackage[T1]{fontenc}
\usepackage{newtxtext}
\usepackage{newtxmath}
\usepackage{textcomp}
\usepackage[colorlinks,	
linkcolor=BrickRed,
citecolor=Green,
urlcolor=Cerulean]{hyperref} 
\usepackage{mathtools}
\usepackage{euscript}
\usepackage[all]{xy}
\usepackage{smfhyperref}
\usepackage{graphicx} 
\usepackage[text={150mm,251mm},centering, marginparwidth=75pt]{geometry} 
\usepackage{comment}  
\usepackage{cite}  
\usepackage{thmtools}
\usepackage{enumitem}
\usepackage{letltxmacro} 
\usepackage{nameref}
\usepackage{cleveref}
\usepackage{calc}
\usepackage{interval}
\usepackage{manfnt}
\usepackage{tikz-cd}  
\usepackage{marginnote}
\usepackage{smfhyperref}
\usepackage[hyperpageref]{backref}
\renewcommand{\backref}[1]{$\uparrow$~#1}

\usepackage{faktor} 
 \usepackage{scalerel}

\usepackage{soul}
\usetikzlibrary{shapes,arrows,positioning}
\setlist[itemize]{wide = 0pt, labelwidth = 2em, labelsep*=0em, itemindent = 0pt, leftmargin = \dimexpr\labelwidth + \labelsep\relax, noitemsep,topsep = 1ex,}
\setlist[enumerate]{wide = 0pt, labelwidth = 2em, labelsep*=0em, itemindent = 0pt, leftmargin = \dimexpr\labelwidth + \labelsep\relax, noitemsep,topsep = 1ex}

\theoremstyle{plain}
 \newtheorem{thm}{Theorem}[section]   
\newtheorem{lem}[thm]{Lemma}
\newtheorem{claim}[thm]{Claim} 
\newtheorem{proposition}[thm]{Proposition}
\newtheorem{cor}[thm]{Corollary} 
\newtheorem{conjecture}[thm]{Conjecture} 
\newtheorem{thmx}{Theorem}
\renewcommand{\thethmx}{\Alph{thmx}} 

\newtheorem{corx}[thmx]{Corollary} 
\theoremstyle{definition}
\newtheorem{dfn}[thm]{Definition}

\theoremstyle{remark}
\newtheorem{rem}[thm]{Remark} 
\newtheorem{example}[thm]{Example}

\numberwithin{equation}{section}  

\theoremstyle{plain}
\newlist{thmlist}{enumerate}{1}
\setlist[thmlist]{wide = 0pt, labelwidth = 2em, labelsep*=0em, itemindent = 0pt, leftmargin = \dimexpr\labelwidth + \labelsep\relax, noitemsep,topsep = 1ex, font=\normalfont, label=(\roman*), ref=\thethm.(\roman{thmlisti})}

\addtotheorempostheadhook[thm]{\crefalias{thmlisti}{thm}}

\addtotheorempostheadhook[assumpsion]{\crefalias{thmlisti}{assumption}}

\addtotheorempostheadhook[cor]{\crefalias{thmlisti}{cor}}

\addtotheorempostheadhook[proposition]{\crefalias{thmlisti}{proposition}}

\addtotheorempostheadhook[dfn]{\crefalias{thmlisti}{dfn}}

\addtotheorempostheadhook[lem]{\crefalias{thmlisti}{lem}}
\addtotheorempostheadhook[main]{\crefalias{thmlisti}{main}}

\addtotheorempostheadhook[rem]{\crefalias{thmlisti}{rem}}

\newlist{thmenum}{enumerate}{1} % also creates a counter called 'propenumi'
\setlist[thmenum]{wide = 0pt, labelwidth = 2em, labelsep*=0em, itemindent = 0pt, leftmargin = \dimexpr\labelwidth + \labelsep\relax, noitemsep,topsep = 1ex, font=\normalfont, label=(\roman*), ref=\thethmx.(\roman{thmenumi})}%{label=\alph*), ref=\thethmx~(\alph*)}
\crefalias{thmenumi}{thmx} 

\newlist{corlist}{enumerate}{1} % also creates a counter called 'propenumi'
\setlist[corlist]{wide = 0pt, labelwidth = 2em, labelsep*=0em, itemindent = 0pt, leftmargin = \dimexpr\labelwidth + \labelsep\relax, noitemsep,topsep = 1ex, font=\normalfont, label=(\roman*), ref=\thecorx.(\roman{corlisti})}%{label=\alph*), ref=\thethmx~(\alph*)}
\crefalias{corlisti}{corx} 

\crefname{lem}{Lemma}{Lemmas} 
\crefname{conjecture}{Conjecture}{Conjectures}
\crefname{thm}{Theorem}{Theorems}
\crefname{proposition}{Proposition}{Propositions}
\crefname{dfn}{Definition}{Definitions}
\crefname{rem}{Remark}{Remarks}
\crefname{cor}{Corollary}{Corollaries}
\crefname{corx}{Corollary}{Corollaries}
\crefname{problem}{Problem}{Problems}
\crefname{thmx}{Theorem}{Theorems}
\crefname{claim}{Claim}{Claims}
\crefname{assumption}{Assumption}{Assumptions}
\crefname{main}{Main Theorem}{Main Theorems}



\def\O{\mathscr{O}}

\newcommand{\cS}{\mathcal{S}}

\newcommand{\Spab}{\mathrm{Sp}_{\mathrm{sab}}}
\newcommand{\Sph}{\mathrm{Sp}_{\mathrm{h}}}
\newcommand{\Spp}{\mathrm{Sp}_{\mathrm{p}}}
\newcommand{\Spalg}{\mathrm{Sp}_{\mathrm{alg}}}

\makeatletter
\newcommand*{\rom}[1]{\expandafter\@slowromancap\romannumeral #1@}
\makeatother

\makeatletter     %Replace the Section ...  by the symbol \S ...  in Cref
\newcommand{\crefnames}[3]{%
	\@for\next:=#1\do{%
		\expandafter\crefname\expandafter{\next}{#2}{#3}%
	}%
}
\makeatother

\crefnames{part,chapter,section}{\S}{\S\S}

\newcommand{\cD}{\mathcal D}

\newcommand{\bA}{\mathbb{A}}

\newcommand{\bC}{\mathbb{C}}
\newcommand{\bD}{\mathbb{D}}

\newcommand{\bF}{\mathbb{F}}
\newcommand{\bG}{\mathbb{G}}

\newcommand{\bP}{\mathbb{P}}

\newcommand{\bR}{\mathbb{R}}

\newcommand{\bZ}{\mathbb{Z}}

\newcommand{\mxp}{M_{\rm B}(X,N)_{\bF_p}}

\newcommand{\myp}{M_{\rm B}(Y,N)_{\bF_p}}

\newcommand{\xsp}{X^{\! \rm sp}}

  \def\spec{\textrm{Spec}\,}
  
\def\hess{{\rm d}{\rm d}^{\rm c}}

\def\Sym{{\text{Sym}}}

\def\alb{{\rm alb}}
\def\Alb{{\rm Alb}}

\def\btau{{\bm{\tau}}}

\newcommand{\ord}{{\rm ord}\,}

\newcommand{\ram}{{\rm ram}\,}

\newcommand{\GL}{{\rm GL}}

\begin{document} 
	\title[Linear Shafarevich conjecture, Hyperbolicity \& Applications]{Linear Shafarevich Conjecture in positive characteristic, Hyperbolicity and Applications}

	\author[Y. Deng]{Ya Deng}
	
	\email{ya.deng@math.cnrs.fr, deng@imj-prg.fr}
	\address{CNRS, Institut \'Elie Cartan de Lorraine, Universit\'e de Lorraine, 54506, Nancy \& 
		Institut de Math\'ematiques de Jussieu-Paris Rive Gauche,
		Sorbonne Universit\'e, Campus Pierre et Marie Curie,
		4 place Jussieu, 75252 Paris Cedex 05, France}
	\urladdr{https://ydeng.perso.math.cnrs.fr}

	\author[K. Yamanoi]{Katsutoshi Yamanoi}
	
	\email{yamanoi@math.sci.osaka-u.ac.jp}
	\address{Department of Mathematics, Graduate School of Science,Osaka University, Toyonaka, Osaka 560-0043, Japan} 
	\urladdr{https://sites.google.com/site/yamanoimath/}

 \begin{abstract}  
  Given a  complex quasi-projective normal variety $X$ and  a linear representation $\varrho:\pi_1(X)\to {\rm GL}_{N}(K)$ with $K$ any field of positive characteristic, we mainly establish the following results:
 	\begin{enumerate}[label*={\rm (\alph*)}]
 		\item the construction of the Shafarevich morphism ${\rm sh}_\varrho:X\to {\rm Sh}_\varrho(X)$ associated with $\varrho$. 
 			\item In cases where $X$ is projective, $\varrho$ is faithful and the $\Gamma$-dimension of $X$ is  at most two (e.g. $\dim X=2$), we prove that the Shafarevich conjecture holds for $X$: the universal covering of $X$ is holomorphically convex. 
 		\item In cases where $\varrho$ is big,  we prove that  the   Green-Griffiths-Lang conjecture  holds for   $X$: $X$ is of log general type if and only it is pseudo Picard or Brody hyperbolic. 
  	\item 	When $\varrho$ is big and the Zariski closure of $\varrho(\pi_1(X))$ is a semisimple algebraic group, we prove that  $X$ is pseudo Picard hyperbolic, and strongly of log general type.  
  	\item If $X$ is special or $h$-special, then $\varrho(\pi_1(X))$ is virtually abelian. 
 	\end{enumerate} 
   We also prove    Claudon-H\"oring-Koll\'ar's conjecture for complex projective manifolds with linear fundamental groups of any characteristic.
\end{abstract}   
\maketitle
\tableofcontents	
	\section{Main results}
\subsection{Existence of the Shafarevich morphism}
The Shafarevich conjecture stipulates that the universal covering of a complex projective variety is holomorphically convex. If this conjecture holds true, it implies the existence of the \emph{Shafarevich morphism}. Over the past three decades, this conjecture has been extensively studied when considering cases where   fundamental groups are  subgroups of  \emph{complex} general linear groups,   referred to as the \emph{linear Shafarevich conjecture}.  Drawing upon the robust techniques of non-abelian Hodge theories established  by Simpson \cite{Sim88,Sim92} and Gromov-Schoen \cite{GS92},  linear Shafarevich conjecture has been studied in \cite{Kat97,KR98,Eys04,EKPR12,CCE15,DY23},   to quote only a few. It is natural to ask whether the conjecture holds when the fundamental groups of algebraic varieties are subgroups of  general linear groups in positive characteristic.  In this paper we address this question along with the exploration of hyperbolicity and  algebro-geometric properties of these algebraic variety.     The  first result of this paper is the construction of  the Shafarevich morphism.   %As an application,  we establish the generalized Green-Griffiths-Lang conjecture for varieties admitting a big linear representation in positive characteristic.
\begin{thmx}[=\cref{thm:Sha22}]\label{main}
	Let $X$ be a  quasi-projective normal variety and $$\varrho:\pi_1(X)\to \GL_{N}(K)$$ be a linear representation, where $K$ is a field of positive characteristic. Then there exists a dominant (algebraic) morphism  ${\rm sh}_\varrho:X\to {\rm Sh}_\varrho(X)$  over a quasi-projective normal variety ${\rm Sh}_\varrho(X)$ with connected general fibers  such that for any connected Zariski closed subset  $Z\subset X$, the following properties are equivalent:
	\begin{enumerate}[label=\textup{(\alph*)}]
		\item ${\rm sh}_\varrho(Z)$ is a point;
		\item $ \varrho({\rm Im}[\pi_1(Z)\to \pi_1(X)])$ is finite;
		\item for each irreducible component $Z_o$ of $Z$, $ \varrho^{ss}({\rm Im}[\pi_1(Z_o^{\rm norm})\to \pi_1(X)])$ is finite, where $\varrho^{ss}:\pi_1(X)\to \GL_{N}(\bar{K})$ is the semisimplification of $\varrho$ and   $Z_o^{\rm norm}$ denotes the normalization of $Z_o$. 
	\end{enumerate}  
\end{thmx}   
The above morphism ${\rm sh}_\varrho:X\to {\rm Sh}_\varrho(X)$ will be called the \emph{Shafarevich morphism} associated with $\varrho$.  We remark that in our previous work \cite{DY23}, \cref{main} was proved when ${\rm char}\, K=0$ and $\varrho$ is semisimple, with a weaker statement that ${\rm sh}_\varrho:X\to {\rm Sh}_\varrho(X)$   is algebraic in the function field level.  See also the paper by Brunebarbe \cite{Bru23}, which independently obtained a similar result.

We also prove the following theorem on the Shafarevich conjecture.
\begin{thmx}[=\cref{thm:convexity}]\label{main:Sha}
	Let $X$ be a projective normal variety and let $$\varrho:\pi_1(X)\to \GL_{N}(K)$$ be a faithful representation where $K$ is a field of positive characteristic. If the $\Gamma$-dimension (see \cref{def:Gamma}) of $X$  is at most two (e.g. when $\dim X\leq 2$), then the universal covering $\widetilde{X}$  of $X$ is holomorphically convex.  
\end{thmx}
\subsection{On the Green-Griffiths-Lang conjecture}
Building on the methods utilized in establishing \cref{main},   together with the techniques  developed in \cite{CDY22}, we  prove the following theorem on the generalized Green-Griffiths-Lang conjecture.  A stronger and more refined result will be stated in \cref{corx}.
\begin{thmx}[=\cref{thm:GGL} $\subsetneqq$ \cref{corx}]\label{main:GGL}
	Let $X$ be a  complex quasi-projective normal variety. Let $\varrho:\pi_1(X)\to \GL_{N}(K)$ be a big  representation where $K$ is a field of  positive characteristic.  Then  the following properties are equivalent:
	\begin{enumerate}[label=(\alph*)]
		\item   $X$ is of log general type;
		\item   $X$ is strongly of log general type; 
		\item   $X$ is pseudo Picard hyperbolic, that is,  there exists a proper Zariski closed subset $\Xi\subsetneq X$ such that   any holomorphic map $f:\bD^*\to X$ from the punctured disk with essential singularity at the origin has image $f(\bD^*)\subset \Xi$. 
		\item  $X$ is pseudo Brody hyperbolic, that is,  there exists a proper Zariski closed subset $\Xi\subsetneq X$ such that any non-constant holomorphic map $f:\bC\to X$  has image $f(\bC)\subset \Xi$.   
	\end{enumerate} 
\end{thmx}
We say a quasi-projective variety $X$ is \emph{strongly of log general type} if there exists a proper Zariski closed subset $\Xi\subsetneq X$ such that   any positive dimensional closed subvariety $V\subset X$ is of log general type provided that $V\not\subset \Xi$. 
Recall that a representation $\varrho:\pi_1(X)\to G(K)$ is said to be \emph{big}, (or \emph{generically large} in \cite{Kol95}), if for any closed irreducible subvariety $Z\subset X$ containing a \emph{very general} point of $X$, $\varrho({\rm Im}[\pi_1(Z^{\rm norm})\to \pi_1(X)])$ is infinite.   
It is worth noticing that a stronger notion   exists: a representation $\varrho$ is called \emph{large}   if $\varrho({\rm Im}[\pi_1(Z^{\rm norm})\to \pi_1(X)])$ is infinite for any closed subvariety $Z$ of $X$. 
We remark that in \cite{CDY22} we prove \cref{main:GGL}  when ${\rm char}\, K=0$ and $\varrho$ is semisimple. It is worthwhile to mention that in the case ${\rm char}\, K>0$, $\varrho$ is not required to be semisimple.   

We would like to refine \cref{main:GGL} to compare  the \emph{non-hyperbolicity locus} of the hyperbolicity notions in \cref{main:GGL}.  We first introduce   a notion of special loci ${\rm Sp}(\varrho)$ for any  big representation $\varrho:\pi_1(X)\to \GL_{N}(K)$ which measures the "non-large locus"  of $\varrho$.
\begin{dfn}\label{dfn:special}
	Let $X$ be a smooth quasi-projective variety. Let $\varrho:\pi_1(X)\to \GL_{N}(K)$ be a    representation where $K$ is a field.	We define
	$$
	{\rm Sp}(\varrho):= \overline{\bigcup_{\iota:Z\hookrightarrow X}Z}^{\rm Zar},
	$$ 
	where $\iota:Z\hookrightarrow X$ ranges over all positive dimensional closed subvarieties of $X$ such that $\iota^*\varrho(\pi_1(Z))$ is finite. 
\end{dfn}

Additionally,  as in \cite[Definition 0.1]{CDY22}, we can introduce four \emph{special subsets} $\Spab$, $\Spalg(X)$, $\Sph(X)$ and $\Spp(X)$ of $X$ that measure the \emph{non-hyperbolicity locus} of the hyperbolicity notions in \cref{main:GGL} from different perspectives.   
\begin{dfn}\label{def:special2}
	Let $X$ be a   quasi-projective normal variety. We define
	\begin{thmlist} 
		\item $\Spab(X) := \overline{\bigcup_{f}f(A_0)}^{\mathrm{Zar}}$, where $f$ ranges over all non-constant rational maps $$f:A\dashrightarrow X$$ from all semi-abelian varieties $A$ to $X$ such that $f$ is regular on a Zariski open subset $A_0\subset A$ whose complement $A\backslash A_0$ has codimension at least two;
		\item $\Sph(X) := \overline{\bigcup_{f}f(\mathbb{C})}^{\mathrm{Zar}}$, where $f$ ranges over all non-constant holomorphic maps from $\mathbb{C}$ to $X$;
		\item $\Spalg(X) := \overline{\bigcup_{V} V}^{\mathrm{Zar}}$, where $V$ ranges over all positive-dimensional closed subvarieties of $X$ which are not of log general type;
		\item $\Spp(X) := \overline{\bigcup_{f}f(\bD^*)}^{\mathrm{Zar}}$, where $f$ ranges over all holomorphic maps from the punctured disk $\bD^*$ to $X$ with essential singularity at the origin, i.e., $f$ has no holomorphic extension $\bar{f}:\mathbb D\to\overline{X}$ to a projective compactification $\overline{X}$.
	\end{thmlist}
\end{dfn}
Subsequently, we establish a theorem concerning these special subsets, thereby  refining \cref{main:GGL}.  
\begin{thmx}[=\cref{lem:proper,cor:GGL}]\label{corx}
	Let $X$ be a   quasi-projective normal variety. Let $\varrho:\pi_1(X)\to \GL_{N}(K)$ be a big   representation where $K$ is a field of  positive characteristic.  Then  ${\rm Sp}(\varrho)$ is a proper Zariski closed subset of $X$, and  we have
	\begin{align*} 
		\Spab(X)\setminus {\rm Sp}(\varrho) =	  \Spalg(X)\setminus        {\rm Sp}(\varrho)= \Spp(X)\setminus        {\rm Sp}(\varrho)= \Sph(X)\setminus        {\rm Sp}(\varrho).
	\end{align*}
	We have ${\rm Sp}_\bullet(X)\subsetneq X$ if and only if $X$ is of log general type, 	where ${\rm Sp}_{\bullet}$  denotes any of $\Spab$, $\Spalg$,  $\Sph$ or $ \Spp$. 
\end{thmx}

\subsection{How fundamental groups determine hyperbolicity}
It is natural to explore how the fundamental groups of algebraic varieties determine their hyperbolicity properties. In our previous work  \cite{CDY22}, we provided a characterization based on representations of fundamental groups into complex general linear groups. In this paper, we    establish analogous results concerning representations in  positive characteristic fields.  
\begin{thmx}[=\cref{thm:fun}]\label{main2}
	Let $X$ be a  quasi-projective normal variety. Let $$\varrho:\pi_1(X)\to \GL_{N}(K)$$ be a big   representation where $K$ is an algebraically closed  field of  positive characteristic.  If the Zariski closure $\varrho(\pi_1(X))$ is a semisimple algebraic group over $K$, then  ${\rm Sp}_\bullet(X)\subsetneq X$, 	where ${\rm Sp}_{\bullet}$  denotes any of $\Spab$, $\Spalg$,  $\Sph$ or $ \Spp$, in particular, $X$ is of log general type. 
\end{thmx}
It is worthwhile mentioning that when ${\rm char}\, K=0$,  \cref{main2} was proved in \cite[Theorem A]{CDY22}.   We also remark that the condition in \cref{thm:fun} is sharp (cf. \cref{rem:sharp}).

\subsection{Some applications}
\cref{main2} has various applications. 
Campana's abelianity conjecture \cite{Cam04} predicts that a smooth  projective variety $X$ that is \emph{special} has a virtually abelian fundamental group.  Our first application of \cref{main2} is the proof of  this  conjecture in the context of representations in positive characteristic. 
\begin{thmx}[=\cref{thm:abelian}]\label{main:abelian}
	Let $X$ be a  smooth quasi-projective   variety, and let  $\varrho:\pi_1(X)\to \GL_{N}(K)$ be any   representation where $K$ is a  field of  positive characteristic. If  $X$ is  special or   $h$-special (cf. \cref{def:special,defn:20230407}), then $\varrho(\pi_1(X))$ is virtually abelian. 
\end{thmx}
Note that in cases when $X$ is projective and  ${\rm char}\, K=0$,  \cref{main:abelian} was proved by Campana  \cite{Cam04}   (for $X$ special) and the second author   \cite{Yam10} (for $X$ Brody special).  It is important to mention that  \cref{main:abelian} does not hold when ${\rm char}\, K=0$ as  in \cite[Example 11.26]{CDY22} we constructed a special  and Brody special smooth quasi-projective variety with  nilpotent fundamental group but not virtually abelian.    In cases where ${\rm char}\, K=0$, in \cite{CDY22} we prove that $\varrho(\pi_1(X))$ is virtually nilpotent (cf. \cref{thm:CDY}). 

As a first consequence of \cref{main:abelian}, we can address a conjecture proposed by Claudon, H\"oring, and Koll\'ar concerning algebraic varieties with compactifiable universal coverings (cf. \Cref{conj}).
\begin{corx}[=\cref{thm:univ}]\label{main:kollar}
	Let $X$ be a smooth projective variety  with an infinite fundamental group $\pi_1(X)$, such that its universal covering  $\widetilde{X}$  is a Zariski open subset  of some compact K\"ahler manifold. If  there exists a faithful   representation $\varrho:\pi_1(X)\to \GL_{N}(K),$ where $K$ is any field of any characteristic, then   the Albanese map of $X$ is (up to finite \'etale cover)   locally isotrivial with simply connected fiber $F$.  In particular we have $\widetilde{X} \simeq F \times \mathbb{C}^{q(X)}$ with $q(X)$ the irregularity of $X$.
\end{corx}
As another application of \cref{main:abelian}, we provide a characterization of semiabelian variety. 
\begin{corx} [=\cref{thm:char}]\label{corx2}
	Let $X$ be a  smooth quasi-projective   variety, and let  $\varrho:\pi_1(X)\to \GL_{N}(K)$ be a  big  representation where $K$ is a  field of  positive characteristic.
	\begin{enumerate}[label=(\alph*)]
		\item If $X$ is special or $h$-special, then after replacing $X$ by some finite \'etale cover, its quasi-Albanese map $\alpha:X\to A$ is birational and $\alpha_*:\pi_1(X)\to \pi_1(A)$ is an isomorphism.
		\item If the logarithmic Kodaira dimension vanishes, 
		then   after replacing $X$ by some finite \'etale cover,  its quasi-Albanese map $\alpha:X\to A$ is birational and   proper in codimension one,  i.e. there exists a Zariski closed subset $Z\subset A$ of codimension at least two such that $\alpha$ is proper over $A\backslash Z$.
	\end{enumerate}
\end{corx}
It is worth mentioning that in \cite{CDY22} we proved \cref{corx2} in cases where ${\rm char}\, K=0$ and $\varrho$ is big and  reductive. 

Lastly, we apply \cref{main2}  to obtain a structure theorem for quasi-projective varieties $X$ for which there exists   a big representation $\varrho:\pi_1(X)\to \GL_N(K)$ where $K$ is a field of positive characteristic. See \cref{thm:structure}. 

\medspace

\subsection{Structure of the paper}
The paper presents several results  from different perspectives. For the readers' convenience, we list in Figure \ref{fig:relationships} the  relationships between main theorems. 
\begin{figure}
	\centering
	\begin{tikzpicture}[>=stealth, node distance=1.5cm, every node/.style={rectangle, draw, align=center}]
		\node (A) {\cref{main}};
		\node (B) [below left of=A] {\cref{main:Sha}};
		\node (C) [below right of=A] {\cref{main:GGL}};
		\node (D) [below of= C, yshift=0.5cm] {\cref{corx}};
		\node (E) [below of= D, yshift=0.5cm] {\cref{main2}};
		\node (I) [below left of=E, yshift=-0.3cm] {\cref{thm:structure}};
		\node (G) [below right of=E, yshift=-0.3cm] {\cref{main:abelian}};
		\node (H) [below left of=G, yshift=-0.1cm] {\cref{corx2}}; 
		\node (F) [below right of= G, yshift=-0.1cm] {\cref{main:kollar}};
		
		\path[->] (A) edge (B);
		\path[->] (A) edge (C);
		\path[->] (C) edge (D);
		\path[->] (D) edge (E);
		\path[->] (E) edge (G);
		\path[->] (E) edge (I);
		\path[->] (G) edge (H);  
		\path[->] (G) edge (F);
	\end{tikzpicture}
	\caption{Relationships between Main Theorems}
	\label{fig:relationships}
\end{figure}
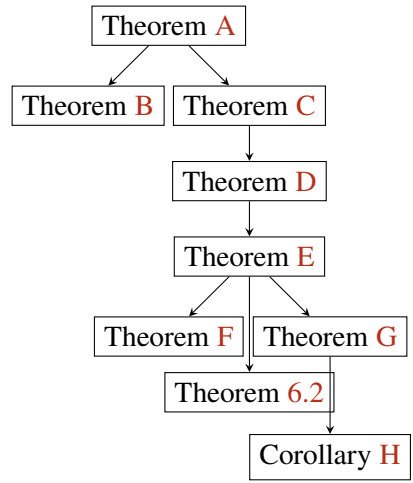

We remark that \cref{main,main:Sha,main:GGL,corx,main2} are entirely novel results, even in cases where $X$ is a projective variety. Their proofs   differ   from those used in studying complex reductive representations of fundamental groups.   In a forthcoming work, we will extend \cref{main:Sha} to arbitrary projective normal varieties. 

\subsection*{Convention and notation} In this paper, we use the following conventions and notations:
\begin{itemize}[noitemsep]
	\item Quasi-projective varieties and their closed subvarieties are assumed to be positive-dimensional and irreducible unless specifically mentioned otherwise. Zariski closed subsets, however, may be reducible.  
	\item Fundamental groups are always referred to as topological fundamental groups.
	\item If $X$ is a complex space, its normalization is denoted by $X^{\mathrm{norm}}$.  
	\item $\bD$ denotes the unit disk in $\bC$, and $\bD^*$ denotes the punctured unit disk.  
	\item For an algebraic group $G$, we denote by $\cD G$ its derived group. 
	\item For any prime number $p$, we denote by $\GL(N,\bF_p)$ the general linear group over $\bF_p$. If $K$ is a field with ${\rm char}\, K=p$, we denote by $\GL_N(K)$ its $K$-points.
	\item For a finitely generated group $\Gamma$, any field $K$ and any representation $\varrho:\Gamma\to \GL_{N}(K)$, we denote by $\varrho^{ss}:\Gamma\to \GL_{N}(\bar{K})$ the semisimplification of $\varrho$, where $\bar{K}$ denotes the algebraic closure of $K$.  
	The representation 
	$\varrho$ is \emph{reductive} if the Zariski closure of $\varrho(\Gamma)$ is a reductive group. 
	We note that $\varrho^{ss}$ is a semisimple representation, thus reductive (cf. \cite[Corollary 19.18]{Mil17}).  
\end{itemize}

\section{Technical preliminary} 

\subsection{Katzarkov-Eyssidieux reduction}
\begin{dfn}[Katzarkov-Eyssidieux reduction]
	Let $X$ be a complex smooth quasi-projective variety, and let $\varrho:\pi_1(X)\to {\rm GL}_N(K)$ be a linear representation where $K$ is a non-archimedean local field. 
	A morphism $s_\varrho:X\to S_\varrho$ to a complex normal quasi-projective variety $S_{\varrho}$ is called Katzarkov-Eyssidieux reduction map if 
	\begin{itemize}
		\item $s_\varrho$ is dominant and has connected general fibers, and
		\item for every connected Zariski closed  subset $T$ of $X$, the image $s_\varrho(T)$ is a point if and only if the image $\varrho({\rm Im}[\pi_1(T)\to \pi_1(X)])$ is a bounded subgroup of $\mathrm{GL}_N(K)$.
	\end{itemize}
\end{dfn}
Recall that a subgroup of $\mathrm{GL}_N(K)$ is said to be \emph{bounded} if it is contained in a maximal compact subgroup of $\mathrm{GL}_N(K)$. Note that a maximal compact subgroup of $\mathrm{GL}_N(K)$ is, up to conjugation, equal to $\mathrm{GL}_N(\mathcal{O}_K)$. 

When $X$ is projective, we may easily see that $s_\varrho:X\to S_\varrho$ is unique up to isomorphism, if it exists.
In our previous work \cite{CDY22} jointly with Cadorel, we establish  the existence of Katzarkov-Eyssidieux reduction map for reductive representations.
This generalized previous work by Katzarkov \cite{Kat97} and Eyssidieux \cite{Eys04} from projective varieties  to the quasi-projective cases.  Here we state a stronger result, which is implicitly contained in our   paper \cite{DY23}. 
\begin{thm} \label{thm:KZ}
	Let $X$ be a complex smooth quasi-projective variety, and let $$\varrho:\pi_1(X)\to {\rm GL}_N(K)$$ be a \emph{linear} representation where $K$ is a non-archimedean local field.  Then there exists a quasi-projective normal variety $S_\varrho$ and a dominant morphism $s_\varrho:X\to S_\varrho$ with connected general fibers, such that  for any connected Zariski closed  subset $T$ of $X$, the following properties are equivalent:
	\begin{enumerate}[label=(\alph*)]
		\item \label{item bounded} the image $\varrho({\rm Im}[\pi_1(T)\to \pi_1(X)])$ is a bounded subgroup of $\GL_{N}(K)$. 
		\item \label{item normalization} For every irreducible component $T_o$ of $T$, the image $\varrho({\rm Im}[\pi_1(T_o^{\rm norm})\to \pi_1(X)])$ is a bounded subgroup of $\GL_{N}(K)$. 
		\item \label{item contraction}The image $s_\varrho(T)$ is a point.\qed
	\end{enumerate} 
\end{thm} 
\begin{proof}
	Let $\varrho^{ss}:\pi_1(X)\to \GL_{N}(\bar{K})$ be the semisimplification of $\varrho$. We assume that $L/K$ is a finite extension such that $\varrho^{ss}:\pi_1(X)\to\GL_{N}(L)$.  It is proven in 	  \cite[Theorem H]{CDY22} that the Katzarkov-Eyssidieux reduction map $s_{\varrho^{ss}}:X\to S_{\varrho^{ss}}$ for $\varrho^{ss}$ exists and satisfies the properties in the theorem.

	On the other hand, we have the following result. 
	\begin{claim}[{\cite[Lemma 3.7]{DY23}}]\label{lem:same bound}
		Let $K$ be a non-archimedean local field and $\Gamma$ be a finitely generated group.    If $\{\varrho_i:\Gamma\to {\rm GL}_N(\bar{K})\}_{i=1,2}$ are two linear  representations such that there semisimplifications are conjugate,  then $\varrho_1$ is bounded if and only if $\varrho_2$ is bounded. \qed %In other words, for the GIT quotient  $\pi: R(X,N) \to M_B(X,N) $  where  $R(X,N)$ is  the representation variety of $\pi_1(X)$ into ${\rm GL}_N$, for any $x\in M_B(X,N)(\bar{K})$, the representations in $\pi^{-1}(x)\subset R(X,N)(\bar{K})$ are either all bounded or all unbounded.  \qed
	\end{claim} 
	Therefore, if we define $s_{\varrho}$ to be $s_{\varrho^{ss}}:X\to S_{\varrho^{ss}}$, it satisfies the properties required in the theorem. 
	Indeed, let $T\subset X$ be a connected Zariski closed subset.
	Set $\Gamma=\pi_1(T)$.
	Then $\Gamma$ is finitely generated.
	Let $\iota:\Gamma\to \pi_1(X)$ be a natural morphism.
	Note that the semisimplifications of two composite representations $\varrho\circ \iota:\Gamma\to {\rm GL}_N(\bar{K})$ and $\varrho^{ss}\circ \iota:\Gamma\to {\rm GL}_N(\bar{K})$ are conjugate.
	Hence by \Cref{lem:same bound}, $\varrho\circ \iota$ is bounded iff $\varrho^{ss}\circ \iota$ is bounded.
	Hence the image $s_{\varrho^{ss}}(T)$ is a point iff $\varrho({\rm Im}[\pi_1(T)\to \pi_1(X)])$ is a bounded subgroup of $\GL_{N}(K)$.
	Similarly $s_{\varrho^{ss}}(T)$ is a point iff $\varrho({\rm Im}[\pi_1(T_o^{\rm norm})\to \pi_1(X)])$ is a bounded subgroup of $\GL_{N}(K)$ for every irreducible component $T_o$ of $T$.
	Thus $s_{\varrho^{ss}}:X\to S_{\varrho^{ss}}$ satisfies the properties required in the theorem. 
\end{proof}

According to this theorem, the two properties \ref{item bounded} and \ref{item normalization} are equivalent for every connected Zariski closed  subset $T$ of $X$.
Hence for every Katzarkov-Eyssidieux reduction map $s_\varrho:X\to S_\varrho$, the three statements \ref{item bounded}, \ref{item normalization} and \ref{item contraction} are equivalent.

\begin{rem}\label{rem:20240306}
	The proof of \cref{thm:KZ} shows that for every linear representation $$\varrho:\pi_1(X)\to {\rm GL}_N(K),$$ where $K$ is a non-archimedean local field, the Katzarkov-Eyssidieux reduction map $$s_{\varrho^{ss}}:X\to S_{\varrho^{ss}}$$ for the semisimplification $\varrho^{ss}$ is the Katzarkov-Eyssidieux reduction map for $\varrho$.
\end{rem}

The following lemma proved in \cite{DY23} will be used throughout this paper.

%\cref{lem:same bound} allows us to define reduction map $s_\varrho:X\to S_\varrho$ for any \emph{linear} representation $\varrho:\pi_1(X)\to {\rm GL}_N(K)$  where $K$ is a non-archimedean local field.
%\footnote{added.}

\begin{lem}[{\cite[Lemma 1.28]{DY23}}]\label{lem:simultaneous}
	Let $V$ be a quasi-projective normal variety and let $(f_{\lambda}:V\to S_{\lambda})_{\lambda\in\Lambda}$ be a family of morphisms into quasi-projective varieties $S_{\lambda}$.
	Then there exist a quasi-projective normal variety $S_{\infty}$ and a morphism $f_{\infty}:V\to S_{\infty}$ such that 
	\begin{itemize}
		\item $f_{\infty}$ is dominant and has connected general fibers,
		\item for every subvariety $Z\subset V$,  $f_{\infty}(Z)$ is a point if and only if $f_{\lambda}(Z)$ is a point for every $\lambda\in \Lambda$, and
		\item
		there exist $\lambda_1,\ldots,\lambda_n\in\Lambda$ such that $f_{\infty}:V\to S_{\infty}$ is the quasi-Stein factorization of $(f_1,\ldots,f_n):V\to S_{\lambda_1}\times\cdots \times S_{\lambda_n}$. 
		\qed
	\end{itemize}
\end{lem}
Such $f_{\infty}:V\to S_{\infty}$ is  called the \emph{simultaneous Stein factorization} of $(f_{\lambda}:V\to S_{\lambda})_{\lambda\in\Lambda}$.

\begin{lem}\label{lem:20240218}
	Let $p$ be a prime number.
	Let $\Lambda$ be a non-empty set of reductive representations $\tau:\pi_1(X)\to\mathrm{GL}_{N_{\tau}}(K_{\tau})$, where $K_{\tau}$ are local fields of $\mathrm{char}\ K_{\tau}=p$.
	Then there exist
	\begin{itemize}
		\item
		a local field $K$ of $\mathrm{char}\ K=p$,
		\item
		a positive integer $N>0$, 
		\item
		a reductive representation $\varrho:\pi_1(X)\to\mathrm{GL}_{N}(K)$
	\end{itemize}
	such that the simultaneous Stein factorization of Katzarkov-Eyssidieux reduction maps $$(s_{\tau}:X\to S_{\tau})_{\tau\in\Lambda}$$ coincides with Katzarkov-Eyssidieux reduction map of $\varrho$.
	Moreover we have $$\cap_{\tau\in\Lambda}\mathrm{ker}(\tau^{ss})\subset \mathrm{ker}(\varrho).$$
\end{lem}

\begin{proof}
	Let $\sigma:X\to \Sigma$ be the simultaneous Stein factorization of $(s_{\tau}:X\to S_{\tau})_{\tau\in\Lambda}$.
	Then $\Sigma$ is normal, and $\sigma:X\to \Sigma$ is dominant and has connected general fibers. 
	By \cref{lem:simultaneous}, there exist $\tau_1,\ldots,\tau_n\in\Lambda$ such that $\sigma:X\to \Sigma$ is the quasi-Stein factorization of $(s_{\tau_1},\ldots,s_{\tau_n}):X\to S_{\tau_1}\times\cdots \times S_{\tau_n}$. 
	Note that each $K_{\tau_i}$ is a finite extension of the Laurent series $\mathbb F_p((t))$ over the finite field $\mathbb F_p$. 
	We may take a local field $K$ of $\mathrm{char}\ K=p$ such that $K_{\tau_i}\subset K$ and $\tau_i^{ss}:\pi_1(X)\to\mathrm{GL}_{N}(K)$ for all $i=1,\ldots,n$.
	Set $N=N_{\tau_1}+\cdots+N_{\tau_n}$.
	We define $\varrho:\pi_1(X)\to \mathrm{GL}_{N}(K)$ by 
	$$(\tau_1^{ss},\ldots,\tau_n^{ss}):\pi_1(X)\to \mathrm{GL}_{N_{\tau_1}}(K)\times\cdots\times \mathrm{GL}_{N_{\tau_n}}(K)\subset \mathrm{GL_{N}(K)}.$$
	Then $\varrho$ is semisimple.
	For every connected Zariski closed  subset $T$ of $X$, $\sigma(T)$ is a point iff $s_{\tau_i}(T)$ is a point for every $i=1,\ldots,n$.
	This happens iff $\tau_i^{ss}({\rm Im}[\pi_1(T)\to \pi_1(X)])$ is a bounded subgroup of $\mathrm{GL}_{N_{\tau_i}}(K)$ for every $i=1,\ldots,n$ (cf. \Cref{rem:20240306}).
	The later is equivalent to the boundedness of $\varrho({\rm Im}[\pi_1(T)\to \pi_1(X)])$.
	Hence the map $\sigma:X\to \Sigma$ is a Katzarkov-Eyssidieux reduction map for $\varrho$.
	Note that $\mathrm{ker}(\varrho)=\mathrm{ker}(\tau_1^{ss})\cap\cdots\cap\mathrm{ker}(\tau_n^{ss})$.
	Hence $\cap_{\tau\in\Lambda}\mathrm{ker}(\tau^{ss})\subset \mathrm{ker}(\varrho)$.
\end{proof}

\subsection{Some facts on algebraic groups}
Note that an algebraic group $G$ over a field of characteristic zero is reductive if and only if $G$ is \emph{linearly reductive}, i.e, for every finite dimensional representation of $G$ is semisimple (cf. \cite[Corollary 22.43]{Mil17}).  However, this fact fails for algebraic group defined over positive characteristic field.  We recall the   following example in \cite[Example 12.55]{Mil17}.
\begin{example}
	Let $k=\overline{\bF_2}$ and let  $V$ be the standard 2-dimensional representation of $\mathrm{SL}_2(k)$.  Then  $\operatorname{Sym}^2V$ is not semisimple as a representation. Indeed, let $e_1:=(1,0)$ and $e_2:=(0,1)$. Within the basis $\{e_1e_2, e_1^2, e_2^2\}$ of $\operatorname{Sym}^2k^2$,  we can express $\operatorname{Sym}^2V$	in the matrix form as 
	\begin{align*}
		\mathrm{SL}_2(k) &\rightarrow \mathrm{GL}_3(k)\\
		\left(\begin{array}{ll}
			a & b \\
			c & d
		\end{array}\right) &\mapsto\left(\begin{array}{ccc}
			1 & 0 & 0 \\
			ab & a^2 & b^2 \\
			cd & c^2 & d^2
		\end{array}\right) 
	\end{align*}
	It is easy to see  that it	is not a semisimple representation since the ${\rm SL}_2$-invariant subspace  spanned by $\{e_1^2, e_2^2\}$  have   no ${\rm SL}_2$-equivariant complement because $ab$ and $cd$  are not linear polynomials in $a^2, b^2, c^2,d^2$.
\end{example} 
It is worthwhile mentioning that an algebraic group $G$ over a field of characteristic $p \neq 0$ is linearly reductive if and only if its identity component $G^{\circ}$ is a torus and $p$ does not divide the index $\left(G: G^{\circ}\right)$ (cf. \cite[Remark 12.56]{Mil17}).  Therefore, in this paper, we must distinguish between \emph{semisimple} and \emph{reductive} representations as we are working over linear algebraic groups over fields of positive characteristic.

\section{Shafarevich morphism in positive characteristic}  
In this section we will prove \cref{main}. 
\subsection{A lemma on finite groups}\label{sec:factor}
\begin{lem}\label{lem:finite group}
	Let $K$ be an algebraically closed field of positive characteristic and let $\Gamma$ be a finitely generated group. Let $\varrho:\Gamma\to \mathrm{GL}_N(K)$ be a representation such that its semisimplification has finite image. Then $\varrho(\Gamma)$ is finite.
\end{lem}
\begin{proof}
	Since the semisimplification $\varrho^{ss}$ of $\varrho$ has finite image,  we can replace $\Gamma$ by a finite index subgroup such that $\varrho^{ss}(\Gamma)$ is trivial.  Therefore, some conjugation $\sigma$ of $\varrho$ has image in the subgroup ${\rm U}_N(K)$ consisting 
	of all upper-triangular matrices in $\GL_{N}(K)$
	with 1's on the main diagonal.  
	
	Note ${\rm U}_N(K)$  admits a central normal series whose successive
	quotients are isomorphic to $\bG_{a,K}$.  We remark that a finitely generated subgroup  of $\bG_{a,K}$ is a finite group, for $K$ has positive characteristic.
	By \cite[Proposition 4.17]{ST00},   any finite index subgroup   of a finitely generated  group  is also finitely generated.   Consequently,    $\sigma(\Gamma)$ admits a   central normal series whose successive
	quotients are   finitely generated subgroups of $\bG_{a,K}$, which are  all finite groups. 
	It follows that $\sigma(\Gamma)$  is finite. The lemma is proved. 
\end{proof}

\subsection{Consideration of character varieties} \label{sec:2.2}
We first briefly explain the character varieties for finitely generated groups in positive characteristic and refer the readers to \cite{LM85,Ses77} for more details.   Let $\Gamma$ be a finitely generated group.
The variety of $N$-dimensional linear  representations of $\Gamma$ in characteristic zero is represented by  an affine $\bZ$-scheme $R(\Gamma,N)$ of finite type.    %Namely, for  any commutative ring $A$,   the set of homomorphisms from $\pi_1(X)$ into $\GL_N(A)$ is represented by  an aﬃne $\bZ$-scheme of ﬁnite type $R$.  
Namely, given a commutative ring $A$, the set of $A$-points of $R(\Gamma,N)$ is:
$$
R(\Gamma,N)(A)={\rm Hom}(\Gamma, \GL_N(A)).
$$
%Hence the $\bK$-points of $R$ can be canonically identified with the set $\operatorname{Hom}\left(\Gamma, \GL_N\left(\bK\right)\right)$.
Let $p $ be  a prime number.  
Consider $R(\Gamma,N)_{\bF_p}:=R(\Gamma,N)  \times_{\spec \bZ}\spec \bF_p$ and note that the general linear group over $\bF_p$, denoted by $\GL(N,\bF_p)$, acts on $R(\Gamma,N)_{\bF_p}$ by conjugation.  %Note that  the $\bK$-points  $M(\bK)$ is identified with be conjugacy classes of semi-simple representations $\pi_1(X)\to \GL_N(\bK)$.    The group $\bZ$-scheme $\GL_N$ acts on $R$ by conjugation. 
Using Seshadri's extension of geometric invariant theory quotients for schemes of arbitrary field \cite[Theorem 3]{Ses77}, we can take the GIT quotient
of $R(\Gamma,N)_{\bF_p}$ by $\GL(N,\bF_p)$, denoted by $M_{\rm B}(\Gamma,N)_{\bF_p}$. 
Then $M_{\rm B}(\Gamma,N)_{\bF_p}$  is also an affine $\bF_p$-scheme of finite type.  For any algebraically closed field $K$ of characteristic $p$,   the $K$-points  $M_{\rm B}(\Gamma,N)_{\bF_p}(K)$ is identified with the conjugacy classes of semi-simple representations $\Gamma\to \GL_N(K)$.    
Namely we have the following. 

\begin{proposition}
	Let $K$ be an algebraically closed field of characteristic $p$.
	Then:
	\begin{enumerate}[label=(\alph*)]
		\item Given a linear representation $\varrho: \Gamma\to \mathrm{GL}_N(K)$, we have $[\varrho^{ss}]=[\varrho]$.
		\item For each point $x\in M_{\rm B}(\Gamma,N)_{\bF_p}(K)$, there exists a semisimple representation $$\varrho:\Gamma\to \mathrm{GL}_N(K)$$ such that $[\varrho]=x$.
		\item Let $\varrho:\Gamma\to \mathrm{GL}_N(K)$ and $\varrho':\Gamma\to \mathrm{GL}_N(K)$ be two semisimple representations such that $[\varrho]=[\varrho']$.
		Then $\varrho$ and $\varrho'$ are conjugate.
	\end{enumerate}
\end{proposition}

\begin{proof}
	These are well-known when the field is zero characteristic (cf. \cite[Thm 1.28]{LM85}).
	The proofs also work for positive characteristic case as well, as we briefly present below. 
	
	Let $\varrho\in R(\Gamma,N)_{\bF_p}(K)$ be a linear representation.
	We denote by $O(\varrho)$ 
	the orbit of $\varrho$ by the conjugacy action, which is a constructible subset of $ R(\Gamma,N)  \times_{\spec \bZ}\spec K$. 
	Let $\overline{O(\varrho)}$ be the closure of $O(\varrho)$. 
	Then we have $\varrho^{ss}$ is a closed point of $\overline{O(\varrho)}$, whose proof is the same as \cite[Lem 1.26]{LM85}. For the GIT quotient $\pi:R(\Gamma,N)_{\bF_p}\to M_{\rm B}(\Gamma,N)_{\bF_p}$, we know that  $\overline{O(\varrho)}\subset \pi^{-1}([\varrho])$. Therefore, we have $[\varrho]=[\varrho^{ss}]$.

	Next, let $x\in M_{\rm B}(\Gamma,N)_{\bF_p}(K)$.
	Since the GIT quotient is surjective, so is 
	$$\pi_K:R(\Gamma,N)_{\bF_p}(K)\to M_{\rm B}(\Gamma,N)_{\bF_p}(K).$$
	Hence we may take $\varrho_0\in R(\Gamma,N)_{\bF_p}(K)$ such that $[\varrho_0]=x$.
	Set $\varrho=(\varrho_0)^{ss}$.
	Then $\varrho$ is semisimple and $[\varrho]=[\varrho_0]=x$.
	
	%Let us prove the last assertion.  Note that $[\varrho]=[\varrho']$ if and only if $\overline{O(\varrho)}\cap \overline{O(\varrho')}\neq\varnothing$ by the main property of GIT quotient.  By the same proof of \cite[Theorem 1.27]{LM85}, $\varrho$ is semisimple if and only if $O(\varrho)$ is closed. This implies that $O(\varrho)=O(\varrho')$. Hence $\varrho$ and $\varrho'$ are conjugate.\footnote{I revised the last assertion}
	%Then $\pi^{-1}_K([\varrho])$ is a closed set which contains $O(\varrho)$.
	%Hence $[\varrho^{ss}]=[\varrho]$.
	
	Finally, for any matrix $A\in \GL_N(K)$,  we denote by $\chi(A)=T^N+\sigma_1(A)T^{N-1}+\cdots+\sigma_N(A)$ its characteristic polynomial.  
	Given $\gamma\in \Gamma$, we define a map $f_{\gamma}:R(\Gamma,N)_{\bF_p}(K)\to\mathbb A^N(K)$ by $f_{\gamma}(\tau)=(\sigma_1(\tau(\gamma)),\ldots,\sigma_N(\tau(\gamma)))$, where $\tau\in R(\Gamma,N)_{\bF_p}(K)$. Note that $f_\gamma$ is induced by a morphism of $\bF_p$-schemes, which we still denote  by $f_\gamma:R(\Gamma,N)_{\bF_p}\to \mathbb A^N_{\bF_p}$.   Note that $f_{\gamma}$ is $\GL(N,\bF_p)$ invariant.
	Hence $f_{\gamma}$ factors the GIT quotient $\pi:R(\Gamma,N)_{\bF_p}\to M_{\rm B}(\Gamma,N)_{\bF_p}$.
	Now let $\varrho, \varrho'\in R(\Gamma,N)_{\bF_p}(K)$ be two semisimple representations such that $[\varrho]=[\varrho']$.
	Then we have $f_{\gamma}(\varrho)=f_{\gamma}(\varrho')$ for all $\gamma\in \Gamma$.
	In other words, $\varrho(\gamma)$ and $\varrho'(\gamma)$ have the same characteristic polynomials for all $\gamma\in\Gamma$.
	Hence by the Brauer-Nesbitt theorem, the two semisimple representations $\varrho$ and $\varrho'$ are conjugate.
\end{proof}

\begin{lem}\label{lem:20240304}
	Let $M\subset M_{\rm B}(\Gamma,N)_{\bF_p}$ be a constructible set. 
	Assume that, for every reductive representation $\tau:\Gamma\to \GL_N(K)$ with $K$ a local field of characteristic $p$ such that $[\tau] \in M(K)$, the image $\tau(\Gamma)\subset \GL_N(K)$ is bounded.
	Then $M$ is zero dimensional.
\end{lem}

\begin{proof}
	Let $\pi:R(\Gamma,N)_{\bF_p}\to M_{\rm B}(\Gamma,N)_{\bF_p}$ be  the GIT quotient, which is a surjective $\bF_p$-morphism.  
	Let $T\subset \pi^{-1}(M)$ be any irreducible affine curve defined over $\overline{\bF_p}$. 
	We shall show that $\pi(T)$ is a point.
	This proves that $M$ is zero dimensional.
	
	To prove that $\pi(T)$ is a point, we take $\bar{C}$ as the compactification of the normalization $C$ of $T$, and let $\{P_1,\ldots,P_\ell\}= \bar{C}\setminus C$.  
	There exists $q=p^n$ for some $n\in \bZ_{>0}$ such that $\bar{C}$ is defined over $\bF_q$ and $P_i\in \bar{C}(\bF_q)$ for each $i$. 
	By the universal property of the representation scheme $R(\Gamma,N)$, $C$ gives rise to a representation $\varrho_C:\Gamma\to \GL_N(\bF_q[C])$, where $\bF_q[C]$ is the coordinate ring of $C$. 
	Consider the discrete valuation $v_i:\bF_q(C)\to \bZ$ defined by $P_i$, where $\bF_q(C)$ is the function field of $C$. 
	Let $\widehat{\bF_q(C)}_{v_i}$  be the completion of $F_{q}(C)$ with respect to $v_i$. 
	Then we have the isomorphism $\big(\widehat{\bF_q(C)}_{v_i},v_i\big)\simeq \big(\bF_q((t)),v\big)$, where $ \big(\bF_q((t)),v\big)$ is the formal Laurent field of $\bF_q$ with the valuation $v$ defined by  $v(\sum_{i=m}^{+\infty}a_it^i)=\min \{i\mid a_i\neq 0\}$.    
	Let $\varrho_i:\Gamma\to \GL_N(\bF_q((t)))$  be the extension of $\varrho_C$ with respect to $\widehat{\bF_q(C)}_{v_i}$. 
	Then we have
	\begin{equation*}\label{eqn:20240224}
		[\varrho_i]\in M( {\bF_q((t))}).
	\end{equation*}
	Hence, by our assumption, $\varrho_i^{ss}(\Gamma)$ is bounded for each $i$.
	Hence, by \cref{lem:same bound}, $\varrho_i(\Gamma)$ is bounded for each $i$.
	Thus after we replace $\varrho_i$ by some conjugation, we have $\varrho_i(\Gamma)\subset \GL_{N}(\bF_q[[t]])$, where  the $\bF_q[[t]]$ is the ring of integers of $\bF_q((t))$, i.e.
	$$
	\bF_q[[t]]:=\{\sum_{i=0}^{+\infty}a_it^i\mid a_i\in \bF_q \}.
	$$ 
	For any matrix $A\in \GL_N(K)$,  we denote by $\chi(A)=T^N+\sigma_1(A)T^{N-1}+\cdots+\sigma_N(A)$ its characteristic polynomial.    
	%Then $\sigma_j(\varrho_C(\gamma))\in \bF_q[C]$ for each $\gamma\in \pi_1(X)$.  %  induces $\sigma_{i,\gamma}\in \overline{\bF_p}[\mxp]$ defined by $[\varrho]\mapsto \sigma_{i,\gamma}(\varrho):=\sigma_i(\varrho(\gamma))$. Since $C$ is defined over $\bF_q$, it follows that $j^*\sigma_{i,\gamma}\in \bF_q[C]$, where $j:C\to \mxp$ is the natural morphism. 
	Since we have assumed that $\varrho_i(\Gamma)\subset \GL_{N}(\bF_q[[t]])$ for each $i$, it follows that $\sigma_{j}(\varrho_i(\gamma))\in \bF_q[[t]]$ for each $i$.  
	Therefore, by the definition of $\varrho_i$, $v_i\big(\sigma_j(\varrho_C(\gamma))\big)\geq 0$ for each $i$.  It follows that $\sigma_j(\varrho_C(\gamma))$  extends to a regular function on $\overline{C}$, which is thus constant.  
	This implies that for any  two representations $\eta_1:\Gamma\to \GL_{N}(K_1)$ and $\eta_2:\Gamma\to \GL_{N}(K_2)$  such that ${\rm char}\, K_1={\rm char}\, K_2=p$ and  $\eta_i\in C(K_i)$,  we have  $\chi(\eta_1(\gamma))=\chi(\eta_2(\gamma))$ for each $\gamma\in \Gamma$.  
	In other words, $\eta_1$ and $\eta_2$ has the same characteristic polynomial.  
	It follows that $[\eta_1]=[\eta_2]$ by the Brauer-Nesbitt theorem. 
	Hence $\pi(T)$ is a point.   
	Thus $M$ is zero dimensional.
\end{proof}

\begin{cor}\label{cor:20240304}
	Let us consider $M\subset M_{\rm B}(X,N)_{\bF_p}$ satisfying the same assumptions as in \cref{lem:20240304}. 
	Let $\varrho:\Gamma\to \GL_N(L)$ be a linear representation with $L$ a field of characteristic $p$ and $[\varrho]\in M(L)$.
	Then the image $\varrho(\Gamma)\subset \GL_N(L)$ is finite.
\end{cor}

\begin{proof}
	By \cref{lem:20240304}, $M$ is zero dimensional.
	Thus we can find a point $\eta:\Gamma\to \GL_{N}(\overline{\bF_p})$ such that $[\eta]\in M(\overline{\bF_p})$ and $[\eta]=[\varrho]$ in $M(L)$.
	Since $\eta(\Gamma)$ is finite, the semisimplification $\eta^{ss}$  of $\eta$ has also finite image.   
	As the semisimplification of $\varrho$ is isomorphic to $\eta^{ss}$, by virtue of \cref{lem:finite group},  we conclude that $\varrho(\Gamma)$ is finite.   
\end{proof}

\begin{lem}\label{lem:202403041}
	Let $\varphi:\Gamma'\to \Gamma$ be a group morphism from another finitely generated group $\Gamma'$. 
	Let $M\subset M_{\rm B}(\Gamma,N)_{\bF_p}$ be a constructible set.
	Then the following two statements are equivalent:
	\begin{enumerate}[label=(\alph*)]
		\item \label{itemaf1} For every reductive representation $\tau:\Gamma\to \GL_N(K)$ with $K$ a local field of characteristic $p$ such that $[\tau] \in M(K)$, the image $\tau\circ\varphi(\Gamma')\subset \GL_N(K)$ is bounded.
		\item \label{itembf1} For every linear representation $\varrho:\Gamma\to \GL_{N}(L)$ with $L$ a field of characteristic $p$ and $[\varrho]\in M(L)$, the image $\varrho\circ\varphi(\Gamma')\subset \GL_N(L)$ is finite.
	\end{enumerate}
\end{lem}

\begin{proof}
	The implication \ref{itembf1} $\implies$ \ref{itemaf1} is trivial.
	In the following we prove the implication \ref{itemaf1} $\implies$ \ref{itembf1}.
	We have a induced map $\iota:M_{\rm B}(\Gamma,N)_{\bF_p}\to M_{\rm B}(\Gamma',N)_{\bF_p}$.
	Set $M'=\iota(M)$.
	Then $M'\subset M_{\rm B}(\Gamma',N)_{\bF_p}$ is a constructible set.
	We shall show that $M'$ satisfies the assumption of \cref{lem:20240304}.
	Indeed let $\sigma:\Gamma\to \GL_N(K')$ be a reductive representation such that $[\sigma]\in M'(K')$, where $K'$ is a local field of characteristic $p$.
	Then we may take a reductive representation $\tau:\Gamma\to \GL_N(\overline{K'})$ such that $[\tau]\in M(\overline{K'})$ and $\iota([\tau])=[\sigma]$.
	By the assumption \ref{itemaf1}, $\tau\circ\varphi(\Gamma')$ is bounded.
	By $\iota([\tau])=[\tau\circ\varphi]$, we have $[\tau\circ\varphi]=[\sigma]$.
	Thus by \cref{lem:same bound}, $\sigma(\Gamma')$ is bounded.
	Thus $M'$ satisfies the assumption of \cref{lem:20240304}.
	
	Now let $\varrho:\Gamma\to \GL_{N}(L)$ be a linear representation such that $[\varrho]\in M(L)$ with $L$ a field of characteristic $p$.
	Then we have $[\varrho\circ\varphi]\in M'(L)$.
	By \cref{cor:20240304}, $\varrho\circ\varphi(\Gamma')$ is finite.
\end{proof}

\subsection{Factorization through non-rigidity}\label{sec:2.3}
Let $X$ be  a quasi-projective smooth variety.   
We write $M_{\rm B}(X,N)_{\bF_p}=M_{\rm B}(\pi_1(X),N)_{\bF_p}$.
Let  $M \subset \mxp $ be a Zariski closed subset.

\begin{dfn}\label{def:reduction ac2} 
	The \emph{reduction map} $s_{M}:X\to S_M$ is obtained through the simultaneous Stein factorization of the reductions   $\{s_{\tau}:X\to S_\tau\}_{[\tau]\in M(K)}$. 
	Here   $\tau:\pi_1(X)\to \GL_N(K)$ ranges over all reductive representations with $K$  a non-archimedean local field of characteristic $p$ such that  $[\tau] \in M(K)$ and $s_\tau:X\to S_\tau$ is the reduction map  defined in \cref{thm:KZ}.  
\end{dfn}

The reduction map  $s_M:X\to S_M$  enjoys the following crucial property.

\begin{thm}\label{thm:Sha1}
	Let $M$ be a Zariski closed subset of $\mxp$.  	
	The reduction map $s_M:X\to S_M$   is the Shafarevich morphism for $M$. That is, for any connected  Zariski closed subset $Z$ of $X$, the following properties are equivalent:
	\begin{enumerate}[label=(\alph*)]
		\item \label{itemaf} $s_M(Z)$ is a point;
		\item \label{itembf}for any linear representation $\varrho:\pi_1(X)\to \GL_{N}(K)$ with $K$ a field of characteristic $p$ and $[\varrho]\in M(K)$,   we have $\varrho({\rm Im}[\pi_1(Z)\to \pi_1(X)])$ is finite;
		\item \label{itemcf}for any semisimple representation $\varrho:\pi_1(X)\to \GL_{N}(K)$ with $K$ a field of characteristic $p$ such that $[\varrho]\in M(K)$,   we have $\varrho({\rm Im}[\pi_1(Z_o^{\rm norm})\to \pi_1(X)])$ is finite, where $Z_o$ is any irreducible component of $Z$. 
	\end{enumerate}
\end{thm}
\begin{proof}
	\noindent {\ref{itemaf} $\implies$ \ref{itembf}:}	Let $\tau:\pi_1(X)\to \GL_N(K)$ be a  reductive representation with $K$  a local field of characteristic $p$ such that  $[\tau] \in M(K)$.
	Then by the assumption \ref{itemaf}, $s_{\tau}(Z)$ is a point.
	Hence $\tau ({\rm Im}[\pi_1(Z)\to \pi_1(X)])$ is bounded.
	Hence by \cref{lem:202403041}, $\varrho ({\rm Im}[\pi_1(Z)\to \pi_1(X)])$ is finite for every linear representation $\varrho:\pi_1(X)\to \GL_{N}(K)$ with $K$ a field of characteristic $p$ and $[\varrho]\in M(K)$.

	\medspace
	
	\noindent {\ref{itembf} $\implies$ \ref{itemcf}:}	this is obvious.
	
	\medspace
	
	\noindent {\ref{itemcf} $\implies$ \ref{itemaf}:} 
	Let $\varrho:\pi_1(X)\to \GL_{N}(K)$ with $K$ any non-archimedean local field of characteristic $p$ such that $[\varrho]\in M(K)$. 
	Then, the image $\varrho({\rm Im}[\pi_1(Z_o^{\rm norm})\to \pi_1(X)])$ is finite by our assumption, and is thus bounded.  
	By the property in \cref{thm:KZ}, $s_{\varrho}(Z)$ is a point.
	By \cref{def:reduction ac2}, $s_M(Z)$ is also a point.  
\end{proof}

\begin{rem}
	We remark that in the case of  characteristic 0, i.e. meaning that $M$ is a Zariski closed subset of $M_{\rm B}(X,N)$ defined over $\bar{\mathbb{Q}}$,  the reduction map $s_M:X\to S_M$  constructed in \cite{DY23} is the same as  \cref{def:reduction ac2}:  it is obtained through the simultaneous Stein factorization of the reductions   $\{s_{\tau}:X\to S_\tau\}_{[\tau]\in M(K)}$, where   $\tau:\pi_1(X)\to \GL_N(K)$ ranges over all reductive representations with $K$  a local field of \emph{characteristic $0$} such that  $[\tau] \in M(K)$ and $s_\tau:X\to S_\tau$ is the reduction map defined in \cref{thm:KZ}. 
	%Furthermore, the property stated in \cref{lem:bounded2} holds true in the characteristic zero case as well. 
	The interested readers can refer to \cite[\S 3.1]{DY23} for further details.
\end{rem}

\subsection{Construction of the  Shafarevich morphism}
\begin{thm}\label{thm:Sha22}
	Let $X$ be a  quasi-projective normal variety and $\varrho:\pi_1(X)\to \GL_{N}(K)$ be a linear  representation, where $K$ is a field of characteristic $p>0$. Then the Shafarevich morphism ${\rm sh}_\varrho:X\to {\rm Sh}_\varrho(X)$ exists.  That is, for any connected Zariski closed subset  $Z\subset X$, the following properties are equivalent:
	\begin{enumerate}[label=\rm (\alph*)]
		\item ${\rm sh}_\varrho(Z)$ is a point;
		\item $ \varrho({\rm Im}[\pi_1(Z)\to \pi_1(X)])$ is finite;
		\item for each irreducible component $Z_o$ of $Z$, $ \varrho^{ss}({\rm Im}[\pi_1(Z_o^{\rm norm})\to \pi_1(X)])$ is finite. 
	\end{enumerate}  
\end{thm}
\begin{proof}  
	\begin{comment}
		\footnote{
			Is this theorem follows directly,
			if we apply $M=\overline{[\varrho]}$ (Zariski closure) in \cref{thm:Sha1}?
			(I think in the case char=0, the space \eqref{eq:kc2} is needed to ensure that $M$ is $\bR^*$-invariant. ) ADDED: In \cref{thm:Sha1}, we require that $M$ is defined over $\overline{\bF_p}$. But $\overline{[\varrho]}$ is not necessarily defined over $\overline{\bF_p}$. Hence we cannot apply \cref{thm:Sha1} to conclude the theorem. 
			
			ADDED: By $\overline{[\varrho]}$, I meant to take Zariski closure in $M_B(X,N)_{\bF_p}$.  
			To be precice, we have $\psi:M_B(X,N)_{K}\to M_B(X,N)_{\bF_p}$ coming from the base change.
			We have $[\varrho]\in M_B(X,N)_{K}$, which is a closed point, and $\psi([\varrho])\in M_B(X,N)_{\bF_p}$, which may be not closed point. 
			We take the Zariski closure $M=\overline{\psi([\varrho])}\subset M_B(X,N)_{\bF_p}$.
			If $\tau:\pi_1(X)\to \mathrm{GL}_N(K)$ is semisimple and $[\tau]\in M(K)$, then $\mathrm{ker}(\varrho)\subset \mathrm{ker}(\tau)$.
			This follows from the proof of Lem 3.23 in our previous paper on reductive Shafarevich conjecture for zero characteristic, since $W_{\gamma}\subset M_B(X,N)_{\bF_p}$  is defined over $\bF_p$.
			Hence the Shafarevich morphism for $M$ coincides with that of $\varrho$.
			However I think the proof of \cref{thm:Sha22} is reasonably short, so I think we do not need to change this. ADDED: I think that you are right. } 
	\end{comment}
	\noindent {\it Step 1: The smooth case.} In this step, we will assume that $X$ is smooth.  Define 
	\begin{equation}\label{eq:kc2}
		M:=\bigcap_{f:Y\to X} j_f^{-1}\{[1]\}, 
	\end{equation}
	where $1$  stands  for the trivial representation, and  $f: Y\to X$ ranges over all  proper morphisms from positive dimensional quasi-projective normal varieties such that $f^*\varrho=1$.    	
	Here $j_f: \mxp\to \myp$ is a morphism of affine $\bF_p$-scheme induced by $f$. 
	Then $M$ is a  Zariski closed   subset.    
	We apply \cref{thm:Sha1} to construct  the  Shafarevich morphism $s_M:X\to S_M$ associated with $M$.   
	It is a dominant morphism with general fibers connected. 	Let ${\rm sh}_\varrho:X\to {\rm Sh}_\varrho(X)$ be  $s_M:X\to S_M$ and we will prove that it satisfies the properties in the theorem. 
	
	\medspace
	
	\noindent {\em {\rm (a)} $\Rightarrow$ {\rm (b)}}:   this follows from the fact that $[\varrho]\in M(K) $ and \cref{thm:Sha1}.
	
	\medspace
	
	\noindent {\em {\rm (b)} $\Rightarrow$ {\rm (c)}}:  obvious.
	
	\medspace
	
	\noindent {\em {\rm (c)} $\Rightarrow$ {\rm (a)}}:   We   take a finite \'etale cover $Y\to Z^{\rm norm}_o$ such that  $f^*\varrho^{ss}( \pi_1(Y))$ is trivial, where we denote by $f:Y\to X$  the natural proper morphism.   Let $\tau:\pi_1(X)\to {\rm GL}_N(L)$ be any linear  representation such that $[\tau]\in M(L)$ where $L$ is any  field of characteristic $p$. 
	Then $[f^*\tau]=[1]$ by \eqref{eq:kc2}. Thanks to  \cref{lem:finite group},   $f^*\tau(\pi_1(Y))$ is finite, and  it follows that   $ \tau({\rm Im}[\pi_1(Z^{\rm norm}_o)\to \pi_1(X)])$ is finite as ${\rm Im}[\pi_1(Y)\to \pi_1(Z^{\rm norm}_o)]$ is a finite index subgroup of $\pi_1(Z^{\rm norm}_o)$. According to \cref{thm:Sha1}, $s_M(Z)$, and thus ${\rm sh}_\varrho(Z)$ is a point.   
	
	\medspace 
	
	\noindent {\it Step 2: Reduction to the smooth case.}  We now do not assume that $X$ is smooth. Let $\mu:X_1\to X$ be a resolution of singularities. Then the Shafarevich morphism ${\rm sh}_{\mu^*\varrho}:X_1\to {\rm Sh}_{\mu^*\varrho}(X_1)$ exists and satisfies the properties in the theorem.  Since $X$ is normal, each fiber $F$ of $\mu$ is compact and connected.  Since $\mu^*\varrho({\rm Im}[\pi_1(F)\to \pi_1(X_1)])=\{1\}$, it implies that ${\rm sh}_{\mu^*\varrho}(F)$ is a point. Hence there exists  a morphism ${\rm sh}_{\varrho}:X\to {\rm Sh}_{\mu^*\varrho}(X_1)$ such that ${\rm sh}_{\varrho}\circ\mu={\rm sh}_{\mu^*\varrho}$. 
	
	\medspace
	
	\noindent {\em {\rm (a)} $\Rightarrow$ {\rm (b)}}:    Let $W:=\mu^{-1}(Z)$, which is a connected Zariski closed subset of $X_1$.  Then ${\rm sh}_{\mu^*\varrho}(W)$ is a point, which implies that $\mu^*\varrho({\rm Im}[\pi_1(W)\to \pi_1(X_1)])$  is finite.  By \cite[Lemma 3.47]{DY23},  
	$ \pi_1(W)\to \pi_1(Z)$ is surjective. It follows that $$\varrho({\rm Im}[\pi_1(Z)\to \pi_1(X)])=\mu^*\varrho({\rm Im}[\pi_1(W)\to \pi_1(X_1)])$$  is finite.  
	
	\medspace
	
	\noindent {\em {\rm (b)} $\Rightarrow$ {\rm (c)}}:  obvious.
	
	\medspace
	
	\noindent {\em {\rm (c)} $\Rightarrow$ {\rm (a)}}:   Let $W$ be an irreducible component of $\mu^{-1}(Z_o)$ which is surjective onto $Z_o$.  %Then $ {\rm Im}[\pi_1(W^{\rm norm})\to \pi_1(Z_o^{\rm norm})]$ is a finite index subgroup of $[\pi_1(W^{\rm norm})\to \pi_1(Z_o^{\rm norm})]$.  
	This implies that $ \mu^*\varrho^{ss}({\rm Im}[\pi_1(W^{\rm norm})\to \pi_1(X_1)])$ is finite. Since $[(\mu^*\varrho)^{ss}]=[\mu^*\varrho^{ss}]$, by \cref{lem:finite group}, $ (\mu^*\varrho)^{ss}({\rm Im}[\pi_1(W^{\rm norm})\to \pi_1(X_1)])$ is also finite. Hence ${\rm sh}_{\varrho}(Z_o)={\rm sh}_{\mu^*\varrho}(W)$ is a point. Since $Z$ is connected, we conclude that ${\rm sh}_{\varrho}(Z)$ is a point. 
	
	Letting ${\rm Sh}_{\varrho}(X):={\rm Sh}_{\mu^*\varrho}(X_1)$, we prove the theorem. 
\end{proof}  
\cref{lem:20240218,thm:Sha1,thm:Sha22} yield the following result. 
\begin{cor}\label{lem:20240226}
	Let $X$ be a  smooth quasi-projective  variety and $\varrho:\pi_1(X)\to \GL_{N}(K)$ be a linear  representation, where $K$ is a field of characteristic $p>0$.  Then there exist
	\begin{itemize}
		\item
		a local field $K'$ of $\mathrm{char}\ K'=p$,
		\item
		a positive integer $N'>0$, 
		\item
		a reductive representation $\sigma:\pi_1(X)\to\mathrm{GL}_{N'}(K')$
	\end{itemize}
	such that the Katzarkov-Eyssidieux reduction map $s_{\sigma}:X\to S_{\sigma}$ of $\sigma$ is the Shafarevich morphism ${\rm sh}_\varrho:X\to {\rm Sh}_\varrho(X)$ of $\varrho$.
\end{cor} 
\begin{proof}
	By \cref{thm:Sha1,thm:Sha22},	${\rm sh}_\varrho:X\to {\rm Sh}_\varrho(X)$  is obtained through the simultaneous Stein factorization of the reduction maps $\{s_{\tau}:X\to S_{\tau}\}_{[\tau]\in M(K)}$, where $$\tau:\pi_1(X)\to \GL_N(K)$$ ranges over all reductive representations with $K$  a local field of characteristic $p$, and $M$ is defined in \eqref{eq:kc2}. 
	By \cref{lem:20240218}, there exist a local field $K$ of $\mathrm{char}\ K=p$, a positive integer $N'$ and a reductive representation $\sigma:\pi_1(X)\to\mathrm{GL}_{N'}(K)$ such that ${\rm sh}_\varrho:X\to {\rm Sh}_\varrho(X)$ coincides with $s_{\sigma}:X\to S_{\sigma}$.
\end{proof}

\section{Hyperbolicity via linear representation in positive characteristic}\label{sec:hyperbolicity}
The structure of the Shafarevich morphism, as presented in the proof of  \Cref{thm:Sha22}, is related to the hyperbolicity of algebraic varieties. We will prove \cref{main:GGL,corx,main2} in this section. 
\subsection{On the generalized Green-Griffiths-Lang conjecture}  
\begin{thm}\label{thm:GGL}
Let $X$ be a  quasi-projective normal variety. Let $\varrho:\pi_1(X)\to \GL_{N}(K)$ be a big   representation where $K$ is a field of  positive characteristic. Then the following properties are equivalent: 
\begin{enumerate}[label=(\roman*)]
	\item  \label{h1} $X$ is of log general type;
	\item   \label{h2}$X$ is pseudo Picard hyperbolic;
	\item \label{h3}$X$ is pseudo Brody hyperbolic;   
	\item \label{h4}  $X$ is strongly of log general type.
\end{enumerate} 
\end{thm}
It's worth noting that the conjectural four equivalent properties mentioned in \cref{thm:GGL} constitute the statement of the generalized Green-Griffiths-Lang conjecture, as presented in \cite{CDY22}.
\begin{proof}[Proof of \cref{thm:GGL}]
By replacing $X$ with a desingularization and $\varrho$ with the pullback on this birational model, we can   assume that $X$ is smooth.   Let $\overline{X}$ be a smooth projective compactification of $X$ such that $D:=\overline{X}\backslash X$ is a simple normal crossing divisor. 	
By \cref{thm:Sha22}, the Shafarevich morphism ${\rm sh}_\varrho:X\to {\rm Sh}_\varrho(X)$ exists.
By \cref{lem:20240226}, the Shafarevich morphism ${\rm sh}_\varrho:X\to {\rm Sh}_\varrho(X)$ coincides with a Katzarkov-Eyssidieux reduction map $s_{\tau}:X\to S_{\tau}$, where $\tau:\pi_1(X)\to\mathrm{GL}_{N'}(K')$ is a reductive representation with some local field $K'$ of $\mathrm{char}\ K'=\mathrm{char}\ K>0$. 
By the construction of $s_{\tau}$ in  \cite[Proof of Theorem H]{CDY22}, there exists a finite (ramified) Galois cover $\pi:\overline{X^{\mathrm{sp}}}\to \overline{X}$ with Galois group $H$ such that 
\begin{enumerate}[label=(\alph*)]
	\item \label{itea} there exists a set of forms $\{\eta_{j}\}_{j=1,\ldots,m}\subset H^0(\overline{X^{\mathrm{sp}}}, \pi^*\Omega_{\overline{X}}(\log D))$ which are invariant under $H$;
	\item  \label{iteb}$\pi$ is \'etale outside  
	\begin{equation}\label{eq:rami}
		R:=\{x\in \overline{X^{\mathrm{sp}}} \mid \exists \eta_{j}\neq\eta_{\ell}  \mbox{ with }(\eta_{j}-\eta_{\ell})(x)=0\}.
	\end{equation}
	\item \label{itec}There exists a morphism $a:X^{\mathrm{sp}}\to A$ to a semi-abelian variety $A$ with $H$ acting on $A$  such that $a$ is $H$-equivariant, where $\xsp:=\pi^{-1}(X)$. Here, \( \xsp \to X \) is called the \emph{spectral cover} of \( X \) associated with \( \tau \). 
	\item \label{ited}The reduction map $s_{\tau}:X\to S_{\tau}$ is the quasi-Stein factorization of the quotient $X\to A/H$ of $a$ by $H$. 
\end{enumerate}   

\begin{claim}\label{claim:finite}
	We have $\dim \xsp=\dim a(\xsp)$.
\end{claim}
\begin{proof}
	Since $\varrho$ is big, it follows that ${\rm sh}_\varrho$, hence $s_{\tau}$ is birational.
	Thus $\dim \xsp= \dim a(\xsp)$ by \Cref{ited}.
\end{proof}
Now we will apply the techniques and results in \cite{CDY22} to prove the theorem.  

\noindent \ref{h1} $\Rightarrow$ \ref{h2}:  
We first recall some notions in Nevanlinna theory used in \cite[\S 4.1]{CDY22}.   Let $Y$ be a connected Riemann surface with a proper surjective holomorphic map $p:Y\to\mathbb C_{>\delta}$, where $\mathbb C_{>\delta}:=\{z\in\bC\mid \delta<|z|\}$ with some fixed positive constant $\delta>0$.
For $r>2\delta$, define
$
Y(r)=p^{-1}\big( \mathbb C_{>2\delta}(r) \big)
$ 
where $\mathbb C_{>2\delta}(r)=\{z\in \bC\mid  2\delta<|z|<r\}$. 
In the following, we assume that $r>2\delta$.
The \emph{ramification counting function} of the covering $p:Y\to  \bC_{>\delta}$ is defined by
\begin{align}\label{eq:ram}
	N_{{\rm ram}\,  p}(r):=\frac{1}{{\rm deg} p}\int_{2\delta}^{r}\left[\sum_{y\in {Y}(t)} \ord_y \ram p \right]\frac{dt}{t},
\end{align} 
where $\ram p\subset Y$ is the ramification divisor of $p:Y\to\mathbb C_{>\delta}$. 

\medspace

For  any holomorphic map $f:\bC_{>\delta}\to X$ whose image is not contained in $\pi(R)$, there exists  a surjective finite holomorphic map   $p:Y\to \bC_{>\delta}$  from a connected Riemann surface $Y$ to  $\bC_{>\delta}$ and  a  holomorphic map $g:Y\to \xsp$  satisfying the following diagram:
\begin{equation}\label{figure:curve}
	\begin{tikzcd}
		Y\arrow[r, "g"] \arrow[d, "p"] & \xsp\arrow[d, "\pi"]\\
		\bC_{>\delta}\arrow[r, "f"] & X
	\end{tikzcd}
\end{equation}
Here $Y$ is a  connected component of the normalization of $\bC_{>\delta}\times_X \xsp$.  By \cite[Proposition 6.9]{CDY22},  there exists a proper Zariski closed subset $E\subsetneq X$ such that for  any holomorphic map $f:\bC_{>\delta}\to X$ whose image is 
not contained in $E$,  
one has  
\begin{align}\label{eq:ramification}
	N_{{\rm ram}\, p}(r) =o( T_{g}(r,L)) + O(\log r)
\end{align} 
for all $r>2\delta$ outside some exceptional interval with finite Lebesgue measure. 
Here $g:Y\to \xsp$ is the induced holomorphic map in \eqref{figure:curve}, $L$ is an ample line bundle on $\overline{\xsp}$ equipped with a smooth hermitian metric $h_L$  
and $T_g(r,L)$ is the order function  defined by 
\begin{align}\label{eq:order}
	T_g(r,L):=\frac{1}{\operatorname{deg} p} \int_{2\delta}^{r}\left[\int_{Y(t)} g^{*}c_1(L,h_L)\right] \frac{d t}{t}. 
\end{align} 
Note that $\xsp$ is of log general type as we assume that $X$ is of log general type and $\pi:\xsp\to X$ is a Galois cover.  
We apply \cite[Theorem 4.1]{CDY22} to conclude that there exists a proper Zariski closed set $\Xi\subsetneqq \xsp$ such that an extension $\overline{g}:\overline{Y}\to\overline{\xsp}$ of $g$ exists provided $g(Y)\not\subset \Xi$, where $\overline{Y}$ is a Riemann surface such that  $p:Y\to \mathbb C_{>\delta}$ extends to a proper map $\overline{p}:\overline{Y}\to \mathbb C_{>\delta}\cup\{\infty\}$. 
This induces an extension $\bar{f}:\mathbb C_{>\delta}\cup\{\infty\}\to \overline{X}$.
Indeed,  
given a sequence $(a_n)$ in $\mathbb C_{>\delta}$ satisfing $a_n\to\infty$, we may take a sequence $(b_n)$ in $Y$ such that $p(b_n)=a_n$ and $b_n\to \beta\in \overline{Y}$.
Then $\bar{p}(\beta)=\infty$ and $f(a_n)=\pi\circ g(b_n)\to \pi\circ \bar{g}(\beta)$.
Thus $(f(a_n))$ converges for every sequence $(a_n)$ satisfing $a_n\to\infty$.
Thus $f$ has continuous extension $\bar{f}:\mathbb C_{>\delta}\cup\{\infty\}\to \overline{X}$, which is holomorphic by Riemann's removable singularity theorem.
Hence, $X$ is pseudo Picard hyperbolic.

\medspace

\noindent \ref{h2} $\Rightarrow$ \ref{h3}:  It is obvious that pseudo Picard hyperbolicity implies pseudo Brody hyperbolicity (cf. \cite[Lemma 4.3]{CDY22}).

\medspace

\noindent \ref{h3} $\Rightarrow$ \ref{h4}: Since $\pi:\xsp\to X$ is a finite surjective morphism, $\xsp$ is also pseudo Brody hyperbolic.   By \cite[Corollary 4.2]{CDY22}, we conclude that there exists   a proper Zariski closed subset $Z\subsetneq \xsp$ such that    any positive dimensional closed subvariety $V$ is of log general type if it is not contained in $Z$.  Then the rest of the proof is basically the same as that of \cite[Theorem 6.3]{CDY22} and let us explain it for the sake of completeness. 

By the proof of \cite[Theorem 6.3]{CDY22}, there exists a Zariski closed subset $\Xi\subset X$ such that we have the following properties:
\begin{enumerate}[label=(\arabic*)]
	\item  \label{item1}	$\pi(Z\cup R)\subset \Xi$, where $R$ is defined in \cref{eq:rami};
	\item Let $V\subset X$ be any closed subvariety such that $V\not\subset \Xi$.
	Let $W\to V$ be a smooth modification and let $\overline{W}$ be a smooth projective compactification of $W$ such that $D_{\overline{W}}:=\overline{W}-W$ is a simple normal crossing divisor,  and $W\to X$ extends to a morphism $\overline{W}\to \overline{X}$. 
	Let $\overline{S}$ be the  normalization of an irreducible component of $\overline{W}\times_{\overline{X}}\overline{\xsp}$.  
	\begin{equation*}
		\begin{tikzcd}
			\overline{S}  \arrow[r, "g"] \arrow[d, "p"]&  \overline{\xsp}\arrow[d, "\pi"]\\
			\overline{W} \arrow[r] & \overline{X}
		\end{tikzcd}
	\end{equation*}
	Then by \cite[Claim 6.5]{CDY22} the finite morphism	$p:\overline{S}\to \overline{W}$ is a Galois morphism with Galois group $H'\subset H$. The morphism $g:\overline{S}\to \overline{\xsp}$ is $H'$-equivariant. 
	\item Define $\psi_j:=g^*\eta_j\in  H^0(\overline{S}, p^*\Omega_{\overline{W}}(\log D_{\overline{W}}))$ for $j=1,\ldots,m$. 
	Let $I$ be the set of all $(i,j)$ such that
	\begin{itemize}
		\item $\eta_i-\eta_j\neq 0$;
		\item   the image of $S\to \xsp$ intersects with $\{z\in\overline{\xsp}\mid (\eta_i-\eta_j)(z)=0\}$. 
	\end{itemize}
	Then by  by \cite[Claim 6.6]{CDY22},  
	for $(i,j)\in I$, $\psi_i-\psi_j\not=0$ in $H^0(\overline{S}, p^*\Omega_{\overline{W}}(\log D_{\overline{W}}))$. 
\end{enumerate}   
We set
\begin{align*} 
	R':=\{z\in \overline{S} \mid \exists   (i,j)\in I  \mbox{ with } (\psi_i-\psi_j)(z)=0 \}.
\end{align*}  
Then $R'$ is a proper Zariski closed subset of $\overline{S}$.  	Denote by  $R_0$ the ramification locus of $p:\overline{S}\to\overline{W}$.  By the purity of branch locus of finite morphisms,  we know that $R_0$ is a (Weil) divisor, and thus $p(R_0)$ is also a divisor.  
Moreover $R_0=p^{-1}\big(p(R_0)\big)$ since $p$ is Galois with Galois group $H'$.  Denote by $E$ the sum of prime components of $p(R_0)$ which intersect $W$.   
One observes that $S-p^{-1}(E)\to W-E$ is finite \'etale.

Since 	 $\pi:\overline{\xsp}\to \overline{X}$  is \'etale over $\overline{\xsp}-R$, it follows that 
$p$ is \'etale over $\overline{S}-g^{-1}(R)$ and thus 
$
R_0\cap S\subset R'\cap S.
$

Since  $\overline{S}\to \overline{\xsp}$ is $H'$-equivariant, it follows that any $h\in H'$ acts on $$\{\psi_i-\psi_j\}_{(i,j)\in I}\subset H^0(\overline{S}, p^*\Omega_{\overline{W}}(\log D_{\overline{W}}))$$ as a permutation by our choice of $I$. 
Define a section
$$
\sigma:=\prod_{h\in H'}\prod_{(i,j)\in I}h^*(\psi_i-\psi_j)\in H^0(\overline{S}, \Sym^{\ell}p^*\Omega_{\overline{W}}(\log D_{\overline{W}})), 
$$
which is non-zero and vanishes at  $R'$ by our choice of $I$. 
Then it is invariant under the $H'$-action and thus descends to a section
$$
\sigma^{H'}\in H^0(\overline{W}, \Sym^{\ell}\Omega_{\overline{W}}(\log D_{\overline{W}}))
$$
so that $p^*\sigma^{H'}=\sigma$.      Since $R_0\cap S\subset R'\cap S$ and $p^{-1}(p(R_0))=R_0$,     $\sigma^{H'}$ vanishes at   the divisor $E$. This implies that there is a non-trivial morphism
\begin{align}\label{eq:CP1}
	\mathcal{O}_{\overline{W}}(E)\to \Sym^{\ell}\Omega_{\overline{W}}(\log D_{\overline{W}}).
\end{align} 
Recall that  $S$ is of log general type by \Cref{item1} and the choice of $Z$.
Since $S-p^{-1}(E)\to W-E$ is finite \'etale,  $W\backslash E$ is also of log general type.  By \cite[Lemma 3]{NWY13}, $K_{\overline{W}}+E+D_{\overline{W}}$ is big.   
Together with \eqref{eq:CP1} we can apply \cite[Corollary 8.7]{CP19} to conclude that $K_{\overline{W}}+ D_{\overline{W}}$ is big.
Hence $W$, so $V$ is of log general type. 

\medspace

\noindent \ref{h4} $\Rightarrow$ \ref{h1}:
This is obvious.
Thus we have completed the proof of \cref{thm:GGL}. 
\end{proof}
\begin{rem}\label{rem}
From the proof of (i) $\Rightarrow$ (ii) in \cref{thm:GGL}, we note that the weaker  condition of $\xsp$ being of log general type is sufficient to establish the pseudo Picard hyperbolicity of $X$. 
This observation will be particularly relevant for the  proofs of \cref{prop:sab} and \cref{thm:fun} below. 
\end{rem}

\subsection{Comparison of  special subsets}
This subsection is devoted to the proof of \cref{corx}. We first prove a lemma on the special subset of the big representation in positive characteristic. 
\begin{lem}\label{lem:proper}
Let $X$ be a smooth quasi-projective variety. Let $\varrho:\pi_1(X)\to \GL_{N}(K)$ be a big   representation where $K$ is a field of  positive characteristic.    Then  the special subset  ${\rm Sp}(\varrho)$ defined in \cref{dfn:special}   is a proper Zariski closed subset of $X$.  
\end{lem}
\begin{proof}
Note that the Shafarevich morphism ${\rm sh}_\varrho:X\to {\rm Sh}_\varrho(X)$ exists, which is dominant morphism with connected general fibers 
(cf. \cref{thm:Sha22}).
Since $\varrho$ is big, it follows that ${\rm sh}_\varrho$ is birational.  Therefore, we have a Zariski dense open set $X^\circ\subset X$ such that ${\rm sh}_\varrho:X^\circ\to {\rm sh}_\varrho(X^\circ)$ is an isomorphism. 
By \cref{thm:Sha22}, for any positive dimensional  closed subvariety $Z\subset X$ not contained in $X\setminus X^\circ$, $\varrho({\rm Im}[\pi_1(Z)\to \pi_1(X)])$ is infinite.  
Hence  ${\rm Sp}(\varrho)\subset  X\setminus X^\circ$.   
This concludes that ${\rm Sp}(\varrho)$   is a proper Zariski closed subset of $X$.  
\end{proof}
Before proceeding to the proof of  \cref{corx} (see \cref{cor:GGL} below), we will first prove the following result.
\begin{proposition}\label{prop:sab}
Let $X$ be a smooth quasi-projective variety. Let $\varrho:\pi_1(X)\to \GL_{N}(K)$ be a big   representation where $K$ is a field of  positive characteristic.  
Assume that there exists a holomorphic map $f:\bD^*\to X$ such that $f$ has essential singularity at the origin and that the image $f(\bD^*)$ is Zariski dense.
Then  $\Spab(X)=X$. 
\end{proposition}
\begin{proof}
We will maintain the same notations as introduced in the proof of \cref{thm:GGL}.	Let $\pi:\xsp\to X$ be the Galois covering, as  defined therein.  
By the assumption, $X$ is not pseudo Picard hyperbolic.
Hence $\xsp$ is not of log general type by \cref{rem}. 
Recall that in the proof of \cref{thm:GGL},  we have a morphism $a:\xsp\to A$, where 
$A$ is a semiabelian variety, and $\dim \xsp=\dim a(\xsp)$. 
%By \cite[Corollary 4.2]{CDY22}, $\xsp$ is not of log general type. 
Note that $\bar{\kappa}(\xsp)\geq 0$. We denote by  $j:\xsp\dashrightarrow J(\xsp)$ the logarithmic Iitaka fibration of $\xsp$, whose general fibers are positive dimensional as $\xsp$ is not of log general type. After replacing $\xsp$ be a proper birational model, we might assume that $\xsp$ is smooth and $j$ is regular. Then for a very general fiber $F$ of $j$, we have $\bar{\kappa}(F)=0$ and $\dim F=\dim a(F)>0$. By \cite[Lemma 3.5]{CDY22}, we have $\Spab(F)=F$. Consequently, $\Spab(\xsp)=\xsp$. Since $\pi:\xsp\to X$ is surjective, it follows that $\Spab(X)=X$.  The proposition is proved. 
\end{proof}

\begin{thm}[=\cref{corx}]\label{cor:GGL}
Let $X$ be a  quasi-projective normal variety. Let $$\varrho:\pi_1(X)\to \GL_{N}(K)$$ be a big   representation where $K$ is a field of  positive characteristic.  Then  we have
\begin{align*}%\label{eq:special same}
	\Spab(X)\setminus {\rm Sp}(\varrho) =	  \Spalg(X)\setminus        {\rm Sp}(\varrho)= \Spp(X)\setminus        {\rm Sp}(\varrho)= \Sph(X)\setminus        {\rm Sp}(\varrho).
\end{align*}
We have ${\rm Sp}_\bullet(X)\subsetneq X$ if and only if $X$ is of log general type, 	where ${\rm Sp}_{\bullet}$  denotes any of $\Spab$, $\Spalg$,  $\Sph$ or $ \Spp$. 
\end{thm}
\begin{proof} 
By \cite[Lemma 4.3]{CDY22} one has  
\begin{align}\label{eq:subset}
	\Spab(X)\subseteq	 \Sph(X)\subseteq\Spp(X).   
\end{align}
Let $f:\bD^*\to X$ be a holomorphic map with essential singularity at the origin such that $f(\bD^*)\not\subset {\rm Sp}(\varrho)$. Let $Z$ be a desingularization of the Zariski closure of $f(\bD^*)$. 
By the definition of ${\rm Sp}(\varrho)$ in \cref{dfn:special}, we note that  the natural morphism $\iota:Z\to X$ induces a big   representation $\iota^*\varrho$.   
Since $\Spp(Z)=Z$,
\cref{thm:GGL} implies that $Z$ is not of log general type. 
Hence $\iota(Z)\subset \Spalg(X)$, which implies    
\begin{align}\label{eq:subset2}
	\Spp(X)\setminus        {\rm Sp}(\varrho)\subseteq  \Spalg(X)\setminus        {\rm Sp}(\varrho).
\end{align}
By \cref{prop:sab}, $\Spab(Z)=Z$. It follows that 
\begin{align*}\label{eq:subset3}
	\Spp(X)\setminus        {\rm Sp}(\varrho)\subseteq  \Spab(X)\setminus        {\rm Sp}(\varrho).
\end{align*}
Combining this with \eqref{eq:subset}, we get 
$$\Spab(X)\setminus {\rm Sp}(\varrho) = \Sph(X)\setminus        {\rm Sp}(\varrho)=\Spp(X)\setminus        {\rm Sp}(\varrho).$$
Thus it remains to prove $ \Spp(X)\setminus        {\rm Sp}(\varrho)= \Spalg(X)\setminus        {\rm Sp}(\varrho)$.

\medspace

Let $Y$ be a closed subvariety of $X$ that is not of log general type. Assume that $Y\not\subseteq {\rm Sp}(\varrho)$.     Let $\iota:Z\to Y$ be a desingularization. 
Then $\iota^*\varrho$ is a big representation. 
By \cref{thm:GGL}, we have $\Spp(Y)=Y$.
Hence $Y\subset \Spp(X)$, which implies 	\begin{align*}
	\Spalg(X)\setminus        {\rm Sp}(\varrho)\subseteq  \Spp(X)\setminus        {\rm Sp}(\varrho).
\end{align*}
Together with \eqref{eq:subset2}, we obtain $ \Spp(X)\setminus        {\rm Sp}(\varrho)= \Spalg(X)\setminus        {\rm Sp}(\varrho)$.
We have proved our theorem.
\end{proof}

\subsection{A characterization of hyperbolicity via fundamental groups} 
\begin{thm}\label{thm:fun}
Let $X$ be a quasi-projective normal variety. Let $\varrho:\pi_1(X)\to \GL_{N}(K)$ be a big   representation where $K$ is a  field of  positive characteristic.  If the Zariski closure $G$ of $\varrho(\pi_1(X))$ is a semisimple algebraic group, then  ${\rm Sp}_\bullet(X)\subsetneq X$  	where ${\rm Sp}_{\bullet}$  denotes any of $\Spab$, $\Spalg$,  $\Sph$ or $ \Spp$, in particular, $X$ is of log general type. 
\end{thm}
\begin{proof}
We may assume that $K$ is algebraically closed.  Replacing $X$ by a desingularization, we may assume that $X$ is smooth. 	We will still maintain the same notations as introduced in the proof of \cref{thm:GGL}.	Let $\pi:\xsp\to X$ be the Galois covering  defined therein.  Consider the representation $\pi^*\varrho:\pi_1(\xsp)\to G(K)$, which is Zariski dense as ${\rm Im}[\pi_1(\xsp)\to \pi_1(X)]$ is a finite index subgroup of $\pi_1(X)$. By the proof of \cref{thm:GGL}, there exists a morphism $a:\xsp\to A$ where $A$ is a semiabelian variety such that $\dim \xsp=\dim a(\xsp)$. Hence we have $\bar{\kappa}(\xsp)\geq 0$.  
\begin{claim}\label{claim:20250710}
	$\xsp$ is of log general type. 
\end{claim}
\begin{proof}
	Let $\mu:Y\to \xsp$ be a desingularization such that the logarithmic Iitaka fibration $j:Y\to J(Y)$ is regular. For a very general fiber $F$ of $j$, we have $\bar{\kappa}(F)=0$ and $\dim F=\dim a(F)$. By \cite[Lemma 3.3]{CDY22}, we have $\pi_1(F)$ is abelian. 
	
	We write $\tau=(\pi\circ\mu)^*\varrho:\pi_1(Y)\to G(K)$. Notably,  $\tau(\pi_1(Y))$ is Zariski dense in $G$.  By \cite[Corollary 2.9.3]{Kol93} (or \cite[Lemma 2.2]{CDY22}),   ${\rm Im}[\pi_1(F)\to \pi_1(Y)]$ is a normal subgroup of $\pi_1(Y)$.  Consequently,    the Zariski closure  $N$ of $\tau({\rm Im}[\pi_1(F)\to \pi_1(Y)])$ is a normal subgroup of $G$.  It's worth noting that the connected component   $N^\circ$ of $N$ is a tori since $\tau({\rm Im}[\pi_1(F)\to \pi_1(Y)])$ is commutative. Therefore, $N^\circ$  must be trivial since $G$ is assumed to be semisimple.  Consequently,   $\tau({\rm Im}[\pi_1(F)\to \pi_1(Y)])$ is finite. 
	
	Given our assumption that $\varrho$ is big, we conclude that $\tau$ is also big. Therefore,  we arrive at the conclusion that $\dim F=0$, leading us to deduce that both $Y$ and, consequently, $\xsp$ are of log general type.
\end{proof}
Now using \Cref{claim:20250710},  
we can carry out the same proof as the step (i)$\Rightarrow$(ii)  in the proof  of \cref{thm:GGL} to conclude that $X$ is pseudo Picard hyperbolic. It's essential to emphasize that in that proof, the condition of $X$ being of log general type is only used to infer 
that    $\xsp$ is  of log general type. See \cref{rem}.  We now apply \cref{cor:GGL} to conclude that ${\rm Sp}_\bullet\subsetneq X$  	where ${\rm Sp}_{\bullet}$  denotes any of $\Spab$, $\Spalg$,  $\Sph$ or $ \Spp$.  
\end{proof}
\begin{rem}\label{rem:sharp}
Note that the condition in \cref{thm:fun} is sharp. For example,   the representation 
\begin{align*}
	\bZ \simeq  \pi_1(\bC^*)&\to  \GL_1(\bF_p(t))\\
	n&\mapsto  t^n
\end{align*}
is a big and Zariski dense representation. However, $\bC^*$ is not of general type and contains a Zariski dense entire curve. This example demonstrates that the semisimplicity of $G$ is necessary for \cref{thm:fun}  to hold.

Furthermore, the condition of bigness   in \cref{thm:fun} is also indispensable. Indeed, it's worth noting that  for any complex algebraic variety $X$, we have $\pi_1(X)\simeq \pi_1(X\times \bP^1)$. The requirement of $\varrho$ being  big  effectively excludes the case of  uniruled varieties $X\times \bP^1$, which is apparently non-hyperbolic. 
\end{rem}  
In the next three sections, we will provide some applications    of    \cref{thm:fun}. 

\section{On Campana's abelianity conjecture}
As our first application of \cref{thm:fun}, in this section we prove Campana's abelianity conjecture in the context of representations in positive characteristic.   
\subsection{Special and $h$-special varieties: properties and conjectures}
We first recall the definition of special varieties by Campana \cite{Cam04,Cam11}.
\begin{dfn}[Campana's specialness]\label{def:special}
Let $X$ be a    quasi-projective normal variety.
\begin{thmlist}
	\item	$X$ is \emph{weakly special} if for any finite \'etale cover $\widehat{X}\to X$ and any proper birational modification $\widehat{X}'\to \widehat{X}$, there exists no  dominant morphism  $\widehat{X}'\to Y$  with connected general fibers such that $Y$ is a positive-dimensional  quasi-projective  variety of  log  general type.  
	\item  $X$ is \emph{special} if  for  any proper birational modification $X'\to X$ there is no dominant morphism $X'\to  Y$   with
	connected general fibers over  a positive-dimensional  quasi-projective  variety  $Y$ such that the \emph{Campana orbifold base}  (or simply orbifold base) is  of log  general type.  
	\item $X$ is \emph{Brody special} if  it contains a Zariski dense entire curve. 
\end{thmlist}
\end{dfn}
%Let us briefly recall the definitions of \emph{classical orbifold base} and  \emph{orbifold base}  of a fibration $f:X\to Y$ between smooth projective varieties.    Let $|\Delta| \subset Y$ be the union of all codimension one irreducible components of the locus over which the schemetheoretic fibre of $f$ is not smooth. For each component $\Delta_i$ of $|\Delta|$, let $D_i:=\sum_{j \in J} m_{i, j} D_{i, j}$ be the union of all components of $f^* \Delta_i$ that are mapped surjectively onto $\Delta_i$ by $f$. Then one defines the multiplicity (resp. classical multiplicity)  of $f$  along $\Delta_i$ by $m_i: =\inf \left\{m_{i, j}, j \in J\right\}$ (resp. $m'_i: ={\rm gcd} \left\{m_{i, j}, j \in J\right\}$) and the $\mathbb{Q}$-divisors
%$$
%\Delta(f):=\sum_{i \in I}\left(1-1 / m_i\right) \Delta_i, \quad  \Delta'(f):=\sum_{i \in I}\left(1-1 / m'_i\right) \Delta_i.
%$$
%The pair $(X, \Delta(f))$ (resp.  $(X,\Delta'(f))$) is called the orbifold base (resp. classical orbifold base) of the fibration $f$.  One can see that $\Delta(f)\geq \Delta'(f)$.   %The fibration is said to be of general type if its orbifold base is of general type.  

Campana defined $X$ to be \emph{$H$-special} if $X$ has vanishing Kobayashi pseudo-distance.
Motivated by \cite[11.3 (5)]{Cam11}, in \cite[Definition 1.11]{CDY22} we introduce the following definition. 
\begin{dfn}[$h$-special]\label{defn:20230407}
Let $X$ be a smooth quasi-projective variety.
We define the equivalence relation $x\sim y$ of two points $x,y\in X$ iff there exists a sequence of holomorphic maps $f_1,\ldots,f_l:\mathbb C\to X$ such that letting $Z_i\subset X$ to be the Zariski closure of $f_i(\mathbb C)$, we have 
$$x\in Z_1, Z_1\cap Z_2\not=\emptyset, \ldots, Z_{l-1}\cap Z_l\not=\emptyset, y\in Z_l.$$
We set $R=\{ (x,y)\in X\times X; x\sim y\}$.
We define $X$ to be \emph{hyperbolically special} ($h$-special for short) iff $R\subset X\times X$ is Zariski dense.
\end{dfn} 
By definition, rationally connected projective varieties are $h$-special without referring to a theorem of Campana and Winkelmann \cite{CW16}, who proved that all rationally connected projective varieties contain Zariski dense entire curves.   It also has the following properties. 
\begin{lem}[{\cite[Lemmas 10.2 \& 10.3 \& 10.4]{CDY22}}]\label{lem:preserve}
\begin{enumerate}[label=(\alph*)]
	\item 	If a smooth quasi-projective variety $X$ is Brody special, then it is $h$-special.
	\item
	Let $X$ be an $h$-special quasi-projective variety.
	Let $S$ be a quasi-projective variety and let $p:X\to S$ be a dominant morphism.
	Then $S$ is $h$-special.
	\item Let $X$ be an $h$-special smooth quasi-projective variety, and let $p:X'\to X$ be a finite \'etale morphism or proper birational morphism from a quasi-projective variety $X'$.
	Then $X'$ is $h$-special. \qed
\end{enumerate} 
\end{lem}
\begin{proposition}[{\cite[Proposition 11.14]{CDY22}}]
\label{prop:pi1}If a quasi-projective smooth variety  $X$ is special or $h$-special, the quasi-albanese map $a:X\to A$ of $X$ is $\pi_1$-exact, i.e.,   we have the following exact sequence:
$$\pi_1(F)\to \pi_1(X)\to \pi_1(A)\to 1,$$ where $F$ is a general fiber of $a$. \qed
\end{proposition}

In \cite{Cam04,Cam11}, Campana proposed the following tantalizing abelianity conjecture. 
\begin{conjecture}[Campana]\label{conj:Campana}
A special   smooth projective variety has virtually abelian fundamental group. 
\end{conjecture}
%Note that in cases where ${\rm char}\, K=0$,  \cref{conj:Campana}  was proved by Campana  \cite{Cam04}   (for $X$ special) and the second author   \cite{Yam10} (for $X$ Brody special)  

In \cite{CDY22} we observed that  \cref{conj:Campana} fails for quasi-projective varieties.  
As illustrated in \cite[Example 11.26]{CDY22}, we constructed a smooth quasi-projective variety  such that it is both special and Brody special, yet it has nilpotent fundamental group that is not virtually abelian. Consequently, within the quasi-projective context, we   revised \cref{conj:Campana} as follows.
\begin{conjecture}[{\cite[Conjecture 1.14]{CDY22}}]\label{conj:Campana2}
A special or $h$-special smooth quasi-projective variety has virtually nilpotent fundamental group. 
\end{conjecture}
In \cite{CDY22}, we confirm \cref{conj:Campana2} for quasi-projective varieties with linear fundamental groups in characteristic zero. 
\begin{thm}[{\cite[Theorem E]{CDY22}}]\label{thm:CDY}
Let $X$ be a special or $h$-special smooth quasiprojective variety. Let  $\varrho:\pi_1(X)\to \GL_{N}(\bC)$ be a linear representation. Then $\varrho(\pi_1(X))$  is virtually nilpotent. \qed
\end{thm}
By \cite[Example 11.26]{CDY22}, \cref{thm:CDY} is shown to be sharp. 
In the context of representations in positive characteristic,   we can obtain a stronger result as motivated by the following lemma.

\begin{lem}\label{lem:VR}
Let $\Gamma\subset \GL_{N}(K)$ be a finitely generated  
subgroup where $K$ is an algebraically closed field  of positive characteristic.  If $\Gamma$ is virtually nilpotent, then it is virtually abelian.
\end{lem}
\begin{proof}
After replacing $\Gamma$ by a suitable finite index subgroup, we can assume that   the Zariski closure $G$ of $\Gamma$ in $\GL_{N}$ is   connected and nilpotent. Hence we have
$
\cD G\subset R_u(G),
$ 
where $R_u(G)$ is the unipotent radical of $G$ and $\cD G$ is the  the derived group of $G$. Consequently, 
We have $$[\Gamma,\Gamma] 
\subset [G(K),G(K)]\subset R_u(G)(K).$$   Note that $R_u(G)$ is unipotent. It is thus a successive extension of $\bG_{a,K}$. 

Since $\Gamma$ is finitely generated and nilpotent, it follows that $[\Gamma,\Gamma]$ is also finitely generated by \cite[Corollary 5.45]{ST00}. 
From this fact we conclude that $[\Gamma,\Gamma]$ is thus a successive extension of  finitely    generated  subgroups of $\bG_{a,K}$. Since finitely generated $p$-groups are finite, it implies that $[\Gamma,\Gamma]$   is a finite group.   Since $\Gamma$ is residually finite by Malcev's theorem, it has a finite index normal subgroup $\Gamma_1$ such that $\Gamma_1\cap [\Gamma,\Gamma]=\{e\}$.  Consequently, $[\Gamma_1,\Gamma_1]\subset [\Gamma,\Gamma]\cap \Gamma_1=\{e\}$. Hence $\Gamma_1$ is   abelian.  We conclude that $\Gamma$ is virtually abelian.  
\end{proof}

\subsection{A factorization result (I)}
In this subsection we prove a factorization result.
\begin{proposition}\label{lem:factor} 
Let   $f:X\to Y$ be a dominant morphism from a smooth quasi-projective variety $X$ to a normal quasi-projective variety $Y$ such that general fibers of $f$ are connected. 	Let $\varrho:\pi_1(X)\to \GL_{N}(K)$ where $K$ is any field. Assume that for a general smooth fiber $F$ of $f$, $\varrho({\rm Im}[\pi_1(F)\to \pi_1(X)])$ is finite.  Then there exists 
\begin{itemize}
	\item  	a generically finite proper surjective morphism $\mu:X_1\to X$ obtained by the composition of birational modifications and  finite \'etale Galois covers;
	\item  a dominant morphism $f_1:X_1\to Y_1$ onto a smooth quasi-projective variety $Y_1$  with connected general fibers;
	\item a generically finite dominant morphism $\nu:Y_1\to Y$;
	\item a representation $\tau:\pi_1(Y_1)\to \GL_N(K)$,
\end{itemize}
such that $f_1^*\tau=\mu^*\varrho$ and we have the following commutative diagram
\begin{equation*}
	\begin{tikzcd}
		X_1\arrow[r, "\mu"] \arrow[d, "f_1"] & X\arrow[d, "f"] \\
		Y_1 \arrow[r, "\nu"] & Y
	\end{tikzcd}
\end{equation*}  
\end{proposition}
\begin{proof}
\noindent {\it Step 1: we can assume that $\varrho({\rm Im}[\pi_1(F)\to \pi_1(X)])$ is trivial.}  Let $Y^\circ$ be the Zariski open subset of $Y$ such that it is smooth and $f$ is smooth over $Y^\circ$.  Denote by $X^\circ:=f^{-1}(Y^\circ)$.
We take a fiber $F:=f(y)$ with $y\in Y^\circ$, which is smooth and connected.   Since $\Gamma:= \varrho(\pi_1(X))$ is residually finite by Malcev's theorem, we can find a finite index normal subgroup $N\triangleleft \Gamma$ such that 
$$
N\cap  \varrho({\rm Im}[\pi_1(F)\to \pi_1(X)])=\{e\}.
$$
Let $\mu:X_1\to X$ be a finite \'etale cover such that 
$$
\mu^*\varrho(\pi_1(X_1))=N. 
$$ 
Let $X_1\stackrel{f_1}{\to}Y_1\stackrel{\nu}{\to} Y$ be the quasi-Stein factorization of $f\circ\mu$. 

%Note that we have the following exact sequence
%$$\pi_1(F)\to \pi_1(X_1^\circ) \to \pi_1(Y^\circ)\to 0,$$
% where $X_1^\circ:=f^{-1}(Y^\circ)$.  
\begin{equation*}
	\begin{tikzcd}
		X_1\arrow[r, "\mu"] \arrow[d, "f_1"] & X\arrow[d, "f"] \\
		Y_1 \arrow[r, "\nu"] & Y
	\end{tikzcd}
\end{equation*} 
Then $f_1$ is smooth over $Y_1^\circ:=\nu^{-1}(Y^\circ)$. %and $\nu$ is \'etale over $Y^\circ$. 
Take any point $y_1\in \nu^{-1}(y)$. Then   the fiber $F_1:=f^{-1}(y_1)$ is smooth and 
$$\mu^*\varrho({\rm Im}[\pi_1(F_1)\to\pi_1(X_1)])\subset N\cap \varrho({\rm Im}[\pi_1(F)\to \pi_1(X)])=\{e\}$$
as $\mu(F_1)\subset F$.  

In the following, to lighten the notation we will replace \(X\), $Y$, $f$ and $\varrho$  by \(X_{1}\),    $Y_1$, $f_1$ and $\mu^*\varrho$ respectively, and assume that  $\varrho({\rm Im}[\pi_1(F)\to\pi_1(X)])$ is trivial for a general  smooth fiber $F$ of $f$.   

\medspace

\noindent{\em Step 2. Compactifications and first reduction step.} 
Let \( \overline{X'} \supset X \) and \( \overline{Y} \supset Y \) be smooth projective compactifications of \( X \) and \( Y \), respectively, such that \( f \) extends to a morphism \( \bar{f}' : \overline{X'} \to \overline{Y} \). By performing a suitable blow-up of \( \overline{X'} \), we may assume that \( \overline{X'} \) is smooth. Since \( f \) is dominant, it follows that \( \bar{f}' \) is surjective. Set \( \overline{X} := (\bar{f}')^{-1}(Y) \). Then \( \bar{f} := \bar{f}'|_{\overline{X}} \) is a surjective proper morphism whose general fibers are connected. As \( Y \) is normal, it follows from Zariski's Main Theorem that all fibers of \( \bar{f} \) are connected. 

\begin{claim} We may assume that \(\bar{f}:\overline{X} \to Y\) is equidimensional.
\end{claim}
Indeed, by the Hironaka-Gruson-Raynaud flattening theorem, there exists a birational proper morphism \( Y_1 \to Y \) from a smooth quasi-projective variety \( Y_1 \), such that for the irreducible component \( T \) of \( \overline{X} \times_Y Y_1 \) dominating \( Y_1 \), the induced morphism \( f_T := T \to Y_1 \) is surjective, proper, and flat. 

By further blowing up \( Y_1 \) and taking the base change of \( f_T \), we may assume that \( f_T \) is smooth over a Zariski open subset \( Y_1^\circ \subset Y_1 \), with \( Y_1 \backslash Y_1^\circ \) a simple normal crossing divisor. 

Let \( \nu : \overline{X}_1 \to T \) denote the normalization map. Since  \( f_T \) is flat, the induced morphism \( f_1 : \overline{X}_1 \to Y_1 \)  has equidimensional fibers. Let \( \mu : \overline{X}_1 \to \overline{X} \) be the induced proper birational morphism, and set \( X_1 := \mu^{-1}(X) \). 

Then one has a diagram
\[
\begin{tikzcd}
	X_{1} \arrow[r] \arrow[d] & X \arrow[d] \\
	Y_{1} \arrow[r] & Y
\end{tikzcd}
\]
where the horizontal maps are proper birational, and the two spaces on the left satisfy the hypotheses of the proposition if we take the representation induced on \(\pi_{1}(X_{1})\). Clearly, it suffices to show the result where \(X\) (resp. \(Y\)) is replaced by \(X_{1}\) (resp. \(Y_{1}\)).  In the following, we may also replace \(\overline{X}\)  (resp. $Y$) by \(\overline{X}_{1}\) and $Y$ (resp. $Y_1$).

\medskip

\noindent
{\em Step 3. Induced representation on an open subset of \(Y\).}  Consider a Zariski open set $Y^\circ\subset Y$ such that    $X^\circ:=f^{-1}(Y^\circ)$   is  a topologically locally trivial fibration over $Y^\circ$.  Recall that in Step 2, we arranged that   $Y\backslash Y^\circ$ is a simple normal crossing divisor.  Then  we have a short exact sequence
$$
\pi_1(F)\to \pi_1(X^\circ)\to \pi_1(Y^\circ)\to 1
$$
where $F$ is a general fiber of $f$ over $Y^\circ$.  By Step 1,  $\varrho({\rm Im}[\pi_1(F)\to \pi_1(X)])$ is trivial. Hence we can pass to the quotient, which yields a representation $\tau_1:\pi_1(Y^\circ)\to \GL_{N}(K)$ so that $\varrho|_{\pi_1(X^\circ)}= f^{\ast}\tau_1$. 

\medskip

\noindent
{\em Step 4. Reducing \(Y\), we may assume that all divisorial components of \(Y - Y^{\circ}\) intersect \(f(X)\).}  Denote by  $E$ the sum of prime divisors of $Y$ contained in the complement $Y\backslash Y^\circ$. We decompose $E=E_1+E_2$ so that $E_1$ is the sum of prime divisors of $E$ that do not intersect \(f(X)\).  We replace $Y$ by $Y\backslash E_1$.  Then for any prime divisor $P$ contained in $Y\backslash Y^\circ$, $f^{-1}(P)\cap X\neq \varnothing$.  

\medskip

\noindent {\it Step 5. Second use of Malcev's theorem.}  Recall that  $Y\backslash Y^\circ$ is a simple normal crossing divisor $D:=\sum_{i\in I}D_i$ and     \(f^{-1}(D_i) \neq \varnothing\) for each $i\in I$. Since \(\bar{f} : \overline{X }\to Y\) is equidimensional, then for any prime component \(P\) of \(f^{-1}(D_i)\), the morphism \(f|_{P} : P \to D_i\) is dominant. Also, since \(X\) is normal, \(X\) is smooth at the general points  of \(P\).

This allows to find a smooth point \(x \) in $P$ (resp. \(y \in D_i\backslash \bigcup_{j\in I; j\neq i}D_j\)) with local coordinates \((z_{1}, \dotsc z_{m})\) (resp. \((w_{1}, \dotsc, w_{n})\)) around \(x\) (resp. \(y\)),   such that  around $x$ (resp. around $y$) we have $P=(z_1=0)$ (resp. $D=(w_1=0)$) and   \(f^{\ast}(w_{1}) = z_{1}^{k}\) for some \(k \geq 1\).
Hence the small meridian loop $\gamma$ around the general point of \(P\)  is mapped to $\eta_i^k$ where $\eta_i$ is the small meridian loop around \(D_i\).  On the other hand, since \(\gamma\) is trivial in $\pi_1(X)$, it follows that
\begin{align}\label{eq:commutative}
	0=\varrho(\gamma)=\tau_1(\eta_i^k).
\end{align}	 
Hence $\tau_1(\eta_i)$ is a torsion element.

We now fix a base point $x_0\in X^\circ$ and write $y_0:=f(x_0)\in Y^\circ$.  Set $\cS:=\{J\subset I\mid D_{J}\neq \varnothing \}$, where $D_J=\cap_{j\in J}D_j\backslash \cup_{i\in I\backslash J} D_i$.  Then $\cS$ is a finite set.   
Take any $J \in \mathcal{S}$ with $J = \{j_1, \ldots, j_k\}$.  
Let $D_{J,1},\ldots,D_{J,k_J}$ be connected components of $D_J$.  
Fix any $i\in \{1,\ldots,k_J\}$ and choose a point $y \in D_{J,i}$. 
We can find a coordinate neighborhood $(U; z_1, \ldots, z_n)$ of $y$ in $Y$ such that
$$
D\cap U=(\cup_{i=1}^{k} D_{j_i})\cap U
$$ with $D_{j_i}=(z_i=0)$ for each $i$. 
Set $U^*:=U-D$ and we fix a point  $y_J\in U^*$.  
The fundamental group $\pi_1(U^*, y_J) \simeq  \bZ^{k}$   is abelian.   Let $e_i$ be  a clockwise loop with basepoint $y_J$ around the origin in the $i$-th  factor of $(\bD^*)^{k}\times \bD^{n-k}$. Pick a smooth path $h: [0,1]\to Y^\circ$ connecting $y_0$ with $y_J$, and denote by $\gamma_i\in \pi_1(Y^\circ, y_0)$ the homotopy class of the loop $h^{-1}\cdot e_i \cdot h$. Set $T_i:=\tau_1(\gamma_i)$.   Clearly, $T_1,\ldots, T_k$ commute pairwise.   By \eqref{eq:commutative}, we know that each $T_i$ is an torsion element. Therefore,  
\begin{align}\label{eq:choice}
	G_{J,i}:=	\left\{ T_1^{\ell_1}\cdots T_k^{\ell_{k }}  \mid    \ell_1,\ldots,\ell_k\in \bZ \right\} 
\end{align}
is a \emph{finite} abelian subgroup of $\tau_1(\pi_1(Y^\circ,y_0))$. Note that  $G_{J,i}$ depends on the choice of $h$, and the above coordinate system $(U;z_1,\ldots,z_k)$, but any two such groups are \emph{conjugate}.  

By Malcev's theorem,  $\tau_1(\pi_1(Y^\circ,y_0))$ is  a finitely generated, residually finite group. Hence  there exists a finite-index normal subgroup $   \Gamma$ of $\tau_1(\pi_1(Y^\circ,y_0))=\varrho(\pi_1(X,x_0))$  such that 
\begin{align}\label{eq:choice2}
	\Gamma\cap G_{J,i}=\{e\}
\end{align}
for each $J\in \cS$ and $i\in \{1,\ldots,k_J\}$. 
Let $\mu: X_1\to X$ be the finite Galois \'etale cover such that $\mu^*\varrho(\pi_1(X_1))=\Gamma$.

\medspace

\noindent {\it Step 6. Completion of the proof.}  Let  $ {X}_1\stackrel{ {f}_1}{\to} Y_1\stackrel{\nu}{\to} Y$ be the  quasi-Stein factorization of $ {f}\circ\mu$.  Then $f_1$ is smooth over $Y_1^\circ:=\nu^{-1}(Y^\circ)$.  We take blow-ups  $\mu_2:X_2\to X_1$ and $\nu_2:Y_2\to Y_1$ such that 
\begin{itemize}
	\item    both $X_2$
	and $Y_2$ are smooth and $f_1$ induces  a morphism $f_2:X_2\to Y_2$.
	\item The map $Y_2^\circ:=\nu^{-1}_2(Y_1^\circ)\to Y_1^\circ$  is an isomorphism. 
	\item    The set  $Y_2\backslash Y_2^\circ$ is a simple normal crossing divisor. 
	\item Set  $X_2^\circ:=f_2^{-1}(Y_2^\circ)$. Then $f_2^\circ:=f_2|_{X_2^\circ}$ is a topologically locally trivial fibration over $Y_2^\circ$.
\end{itemize}    
%We denote by $X_1:=$ Since $\bar{f}$ is equidimensional, then so is $\bar{f}_1$.  Let $Y_2\to Y_1$ be a desingularization, and let $X_2$ be the normalization of the main component of $\overline{X}_1\times_{Y_1}Y_2$. %which is \'etale over $Y^\circ$. 
Note that for the representation $(\nu_2\circ\nu)^*\tau_1:\pi_1(Y_2^\circ)\to \GL_{N}(K)$, its image 
$$
(\nu\circ\nu_2)^*\tau_1(\pi_1(Y_2^\circ))= (\mu\circ\mu_2)^*\varrho(\pi_1(X_2))\subset \Gamma.
$$ 
Let $a:\bD\to Y_2$ be any holomorphic disk such that $(\nu\circ\nu_2\circ a)^{-1}(D)=\{0\}$.  It follows that $\nu\circ\nu_2\circ a(0)\in D_{J,i}$ for 	some  $J\in \cS$ and $i\in \{1,\ldots,k_J\}$.  Then $(\nu\circ a)^*\tau_1(\pi_1(\bD^*))$ is contained in  some conjugacy class of   $G_{J,i}$. Since $\Gamma\cap G_{J,i}=\{e\}$  by \eqref{eq:choice2} and given that $\Gamma$ is a normal subgroup of $	\nu^*\tau_1(\pi_1(Y_1^\circ))$, it implies that $(\nu\circ a)^*\tau_1(\pi_1(\bD^*))=\{e\}$. Consequently, the pullback $(\nu_2 \circ \nu)^* \tau_1$ extends to a representation $\tau: \pi_1(Y_2) \to \GL_N(K)$. It follows that
\[
f_2^* \tau = (\mu \circ \mu_2)^* \varrho.
\]
We now change the notation by renaming $X_2$ as $X_1$, $Y_2$ as $Y_1$, $f_2$ as $f_1$, $\mu \circ \mu_2$ as $\mu$, and $\nu \circ \nu_2$ as $\nu$. This completes the proof of the proposition.
\end{proof}
\begin{rem}
\cref{lem:factor} is proven in \cite[Proposition 2.5]{CDY22} in cases where ${\rm char}\, K=0$ and the proof utilizes Selberg's lemma: a finitely generated linear group in characteristic zero is virtually torsion free.   
\end{rem}

\subsection{On the Abelianity conjecture} 
\begin{thm}\label{thm:abelian}
Let $X$ be a  smooth quasi-projective   variety, and let  $\varrho:\pi_1(X)\to \GL_{N}(K)$ be any   representation with $K$ any field of positive characteristic. If  $X$ is    special or $h$-special, then $\varrho(\pi_1(X))$ is virtually abelian.  
\end{thm}

\begin{proof} 
\noindent {\it Step 1. We prove that $\varrho(\pi_1(X))$ is  solvable.} 
We may assume that $K$ is algebraically closed.  Let $G$ be the Zariski closure of $\varrho(\pi_1(X))$.  Note that any finite \'etale cover of a   special (resp. $h$-special) variety is still  special (resp. $h$-special).   After replacing $X$ by a  finite \'etale cover,  we may assume that $G$ is connected. Let $R(G)$ be the radical of $G$. Let $H:=G/R(G)$, which is semisimple.  If $\dim H>0$,  then $\varrho$ induces a Zariski dense representation $\sigma: \pi_1(X)\to H(K)$.  Let  ${\rm sh}_\sigma:X\to {\rm Sh}_\sigma(X)$   be the Shafarevich morphism  of $\sigma$.    By the property of the Shafarevich morphism in \cref{thm:Sha22}, a general fiber $F$ of ${\rm sh}_\sigma$ is connected and  $ \sigma({\rm Im}[\pi_1(F)\to \pi_1(X)])$ is finite.     We apply \cref{lem:factor} to conclude that there exist 	
\begin{enumerate}[label*={\rm (\roman*)}]
	\item   \label{item:etale} a generically finite proper surjective morphism $\mu:X_1\to X$ from a smooth quasi-projective variety obtained by the composition of birational modifications and  finite \'etale Galois covers;
	\item \label{item:gf} a generically finite dominant morphism $\nu:Y_1\to  {\rm Sh}_\sigma(X)$;
	\item \label{item:gs}a dominant morphism $f_1:X_1\to Y_1$ with $Y_1$ a smooth quasi-projective variety with connected general fibers;
	\item a representation $\tau:\pi_1(Y_1)\to H(K)$
\end{enumerate}
such that we have following commutative diagram
\begin{equation*}
	\begin{tikzcd}
		X_1\arrow[r, "\mu"] \arrow[d, "f_1"] & X\arrow[d, "{\rm sh}_\sigma"] \\
		Y_1 \arrow[r, "\nu"] & {\rm Sh}_\sigma(X)
	\end{tikzcd}
\end{equation*}   and $\mu^*\sigma=f_1^*\tau$. 
We can show that 
$\tau$ is a big representation.  Thanks to \cref{thm:fun}, $Y_1$ is of log general type and pseudo Picard hyperbolic.  
This leads to a contradiction since $X$ is special (thus weakly special by \cite{Cam11}) or $h$-special. Hence $G=R(G)$. 

\medspace

\noindent {\it Step 2. We prove that $\varrho(\pi_1(X))$ is virtually abelian.}   
Note that any finite \'etale cover of a special (resp. $h$-special) variety is still special (resp. $h$-special) by \cite{Cam04} and \cref{lem:preserve}.  
Replacing  $X$ by a finite \'etale cover, we may assume that  $\pi_1(X)^{ab}\to \pi_1(A)$ is an isomorphism, where   $\pi_1(X)^{ab}:=\pi_1(X)/[\pi_1(X),\pi_1(X)]$ and $A$ is the quasi-albanese variety. 
Since  $X$ is special or $h$-special, by \cref{prop:pi1}, the quasi-albanese map $a:X\to A$ of $X$ is $\pi_1$-exact, i.e.,   we have the following exact sequence:
$$\pi_1(F)\to \pi_1(X)\to \pi_1(A)\to 1,$$ where $F$ is a general fiber of $a$. 
Hence $[\pi_1(X),\pi_1(X)]$ is the image of $\pi_1(F)\to \pi_1(X)$, which is thus finitely generated.  
It implies that  $[\varrho(\pi_1(X)),\varrho(\pi_1(X))]=\varrho([\pi_1(X),\pi_1(X)])$ is also finitely generated.
By Step 1,  $G$ is solvable. Hence we have
$
\cD G\subset R_u(G),
$ 
where $R_u(G)$ is the unipotent radical of $G$ and $\cD G$ is the  the derived group of $G$. Consequently, 
we have	%\footnote{
	%	Shall we add the following line?:
	%Note that every subgroup of finite index in $[\pi_1(X),\pi_1(X)]$ is finitely generated (cf. \cite[Proposition 4.17]{ST00}
	%	).} 
$$ [\varrho(\pi_1(X)),\varrho(\pi_1(X))]
\subset [G(K),G(K)]\subset R_u(G)(K).$$  
Note that every subgroup of finite index in $[\pi_1(X),\pi_1(X)]$ is also finitely generated (cf. \cite[Proposition 4.17]{ST00}. 
By the same arguments in \Cref{lem:VR}, we conclude that $[\varrho(\pi_1(X)),\varrho(\pi_1(X))]$ is finite. 
Hence $\varrho(\pi_1(X))$ is virtually abelian. 
\end{proof} 
\subsection{A characterization of semiabelian varieties.}
\cref{thm:abelian}  allows  us to give a characterization of semiabelian varieties. 
\begin{proposition}\label{thm:char}
Let $Y$ be a  smooth quasi-projective   variety, and let  $\varrho:\pi_1(Y)\to \GL_{N}(K)$ be a big representation where $K$ is a  field of  positive characteristic.
\begin{thmlist}
	\item  \label{item:character}  If $Y$ is special or $h$-special, then there exists a finite \'etale cover $X$ of $Y$, such that its quasi-Albanese map $\alpha:X\to A$ is birational and $\alpha_*:\pi_1(X)\to \pi_1(A)$ is an isomorphism.
	\item \label{item:character2} If the logarithmic Kodaira dimension $\bar{\kappa}(Y)=0$,  then there exists a finite \'etale cover $X$ of $Y$, such that  its  quasi-Albanese map $\alpha:X\to A$ is birational and   proper in codimension one,  i.e. there exists a Zariski closed subset $Z\subset A$ of codimension at least two such that $\alpha$ is proper over $A\backslash Z$.
\end{thmlist}
\end{proposition}
The proof closely follows that of \cite[Proposition 12.7]{CDY22}, and we present it here for the sake of completeness.
\begin{proof}
\noindent {\it Proof of (i)}.	
By \cref{thm:abelian},  there is a finite \'etale cover $X$ of $Y$   such that  $G:=\varrho(\pi_1(X))$ is  abelian and torsion free. 
It follows that $\varrho|_{\pi_1(X)}:\pi_1(X)\to \GL_{N}(K)$ factors through  $H_1(X,\bZ)/{\rm torsion}$. 
By \cite[Lemma 11.5]{CDY22}, $\alpha$ is dominant with connected general fibers.	
Since $\alpha_*:H_1(X,\bZ)/{\rm torsion}\to H_1(A,\bZ)$ is isomorphic,  $\varrho$ further factors through 	$H_1(A,\bZ)$.
\begin{equation}\label{eq:dia}
	\begin{tikzcd}
		&	\pi_1(X) \arrow[d]\arrow[r, "\rho|_{\pi_1(X)}"] \arrow[ldd, bend right=30]& G\\
		&	H_1(X,\bZ)/{\rm torsion} \arrow[ru]\arrow[d, "\alpha_*"']&\\
		\pi_1(A)\arrow[r, "="]	&	H_1(A,\bZ)\arrow[ruu, bend right=30,, "\beta"']&
	\end{tikzcd}
\end{equation}
From the above diagram for every fiber $F$ of $\alpha$,   $\varrho({\rm Im}[\pi_1(F)\to \pi_1(X)])$ is trivial.   Since $\varrho$ is big,  the   general fiber of $\alpha$ is thus a point. Hence $\alpha$ is birational.  Since $\alpha:X\to A$ is $\pi_1$-exact by \cref{prop:pi1}, it follows that  $\alpha_*:\pi_1(X)\to \pi_1(A)$ is an isomorphism. 

\medspace

\noindent {\it Proof of (ii)}.	If $\bar{\kappa}(Y)=0$,  then $Y$ is special by \cite[Corollary 5.6]{Cam11} and  for any finite \'etale cover $X$ of $Y$, we have $\bar{\kappa}(X)=0$ by \cite[Lemma 6.7]{CDY22}. By \cref{item:character},  there is a finite \'etale cover $X$ of $Y$ such that its quasi-Albanese map $\alpha:X\to A$  is birational.   
We apply \cite[Lemma 3.2]{CDY22} to conclude that $\alpha$ is proper in codimension one. 
\end{proof}

\section{A structure theorem: on a conjecture by Koll\'ar}
In \cite[Conjecture 4.18]{Kol95}, Koll\'ar raised the following conjecture on the structure of varieties with big fundamental group. 
\begin{conjecture}\label{conj:Kollar}
Let $X$ be a smooth projective variety with big fundamental group such that $0<\kappa(X)<\dim X$. Then $X$ has a finite \'etale cover $p:X'\to X$ such that $X'$ is birational to a smooth family of abelian varieties over a projective variety of general type $Z$ which has big fundamental group. 
\end{conjecture} 
In this section we address \cref{conj:Kollar}.  Our theorem is the following:
\begin{thm}\label{thm:structure}
Let $X$ be a quasi-projective normal variety and let $\varrho:\pi_1(X)\to {\rm GL}_N(K)$ be a   big representation where $K$ is a field of positive characteristic.    Then 
\begin{thmlist}
	\item \label{item:LKD} the logarithmic Kodaira dimension satsifies $\bar{\kappa}(X)\geq 0$. 
	\item \label{item:Mori}There is a proper Zariski closed subset $\Xi$ of $X$ such that each non-constant algebraic  
	morphism $\bA^1\to X$ has image in $\Xi$.
	\item \label{item:local trivial}If $0<\bar{\kappa}(X)<\dim X$, after replacing $X$ by a finite \'etale cover and a birational modification, there are a    semi-abelian variety $A$, a   quasi-projective manifold $V$  and a birational morphism $a:X\to V$  such that the following commutative diagram holds:
	\begin{equation*} 
		\begin{tikzcd}
			X \arrow[rr,    "a"] \arrow[dr, "j"] & & V \arrow[ld, "h"]\\
			& J(X)&
		\end{tikzcd}
	\end{equation*}
	where $j$ is the logarithmic Iitaka fibration of $X$ and $h:V\to J(X)$ is a  locally trivial fibration with fibers isomorphic to $A$.  Moreover,
	for a   general fiber $F$ of $j$, $a|_{F}:F\to A$ is proper in codimension one.  
\end{thmlist}   
\end{thm}    
In our previous work \cite[Theorem 12.5]{CDY22}, we established \cref{thm:structure} for the cases where ${\rm char}\, K=0$ and $\varrho$ is big and reductive.  
The proof of \cref{thm:structure} closely follows that of \cite[Theorem 12.5]{CDY22}.
Note that \cref{item:character2} corresponds to the case $\bar{\kappa}(X)=0$ of \cref{item:local trivial}. 
\begin{proof}
We may assume that $K$ is algebraically closed.  To prove the theorem we are free to replace $X$ by a birational modification  and by a finite \'etale cover since the logarithmic Kodaira dimension will remain unchanged.   We replace $\varrho$  by its semisimplification, which is still big by \cref{lem:finite group}. Hence we might assume that $\varrho$ is big and semisimple. Consequently, after replacing $X$ by a finite \'etale cover,  the Zariski closure $G$ of $\varrho$ is reductive and connected.  Let $\cD G$ be the derived group of $G$, which is semisimple. Let $Z\subset G$ be the maximal central torus of $G$. Then  $T:=G/\cD G$  is a torus and  $S:=G/Z$ is semisimple.  The natural morphism $G\to  S\times T$  is a central isogeny (see \cite[Corollary 5.3.3]{Con14}). 
The induced representation $\varrho': \pi_1(X)\to  S(K)\times T(K)$ by $\varrho$  is also big.   Consider the  representation $\sigma:\pi_1(X)\to S(K)$, obtained by composing $\varrho$ with the morphism $G \to S$. Then $\sigma(\pi_1(X))$ is Zariski dense. Let ${\rm sh}_\sigma:X\to {\rm Sh}_\sigma(X)$ be the Shafarevich morphism of $\sigma$.  

Like Step 1 of the proof of \cref{thm:abelian}, we apply \cref{lem:factor} to conclude that there exist 	
\begin{enumerate}[label*={\rm (\roman*)}]
	\item  a generically finite proper surjective morphism $\mu:X_1\to X$ from a smooth quasi-projective variety obtained by the composition of birational modifications and  finite \'etale Galois covers;
	\item \label{itemgf}a generically finite dominant morphism $\nu:Y_1\to  {\rm Sh}_\sigma(X)$;
	\item  a dominant morphism $f_1:X_1\to Y_1$ with $Y_1$ a smooth quasi-projective variety with connected general fibers;
	\item a big representation $\tau:\pi_1(Y_1)\to S(K)$;
\end{enumerate}
such that we have the  
following commutative diagram
\begin{equation*}
	\begin{tikzcd}
		X_1\arrow[r, "\mu"] \arrow[d, "f_1"] & X\arrow[d, "{\rm sh}_\sigma"] \\
		Y_1 \arrow[r, "\nu"] & {\rm Sh}_\sigma(X)
	\end{tikzcd}
\end{equation*}   and $\mu^*\sigma=f_1^*\tau$. 
It is straightforward to show that
$\tau$ is a big representation.  Thanks to \cref{thm:fun}, the special loci $\Spalg(Y_1)$ and $\Spp(Y_1)$ are both proper Zariski closed subset of  $Y_1$.  In particular, $Y_1$ is of log general type.   Note that $Y_1$ can be a point. 

Consider the morphism  
\begin{align*}
	g:X_1&\to A \times Y_1\\
	x&\mapsto (\alpha(x), f_1(x)).
\end{align*}
where $\alpha:X_1\to A$ is the quasi Albanese map  
of $X_1$. 
\begin{claim}\label{claim:same}
	We have	  $\dim X_1=\dim g(X_1)$.
\end{claim}
\begin{proof}
	For a general smooth fiber $F$ of $f_1$,   $\mu^*\sigma({\rm Im}[\pi_1(F)\to \pi_1(X)])$ is trivial. Since $\mu^*\varrho':\pi_1(X_1)\to S(K)\times T(K)$ is big, by the construction of $\sigma$, we conclude that the representation $\eta:\pi_1(F)\to T(K)$ obtained by 
	$$
	\pi_1(F)\to \pi_1(X_1)\stackrel{\mu^*\varrho'}{\to}  S(K)\times T(K)\to T(K)
	$$ 
	is big.  Since $T(K)$ is commutative,  similar to   \eqref{eq:dia},  $\eta$ factors through $\pi_1(F)\to \pi_1(A)\to T(K)$. This implies that $\dim F=\dim \alpha(F)$. Hence $\dim X_1=\dim g(X_1)$. 
\end{proof}

%We apply \cref{lem:factor}. 
%Then by replacing $X$ with a finite \'etale cover and a birational modification, we obtain a smooth quasi-projective variety $Y$ (might be zero-dimensional), a semiabelian variety $A$, and a morphism $g:X\to A \times Y$ that satisfy the following properties:
%	\begin{itemize}
	%	\item  $\dim X=\dim g(X)$.
	%	\item Let $p:X\to Y$ be  the composition of $g$ with the projective map $A\times Y\to Y$. 
	%	Then  $p$   is dominant.
	%	\item $\Spalg(Y)\subsetneqq Y$ and $\Spp(Y)\subsetneqq Y$.
	%	\end{itemize} 

Let us  prove \cref{item:LKD}.
Thanks to \cref{claim:same},  for a general smooth fiber $F$ of $f_1$,  we have $\dim F=\dim \alpha(F)$.  Hence $\bar{\kappa}(F)\geq 0$. Since $Y_1$ is of log general type, by the  subadditivity of the logarithmic Kodaira dimension  proven in \cite[Theorem 1.9]{Fuj17}, we obtain $$\bar{\kappa}(X_1)\geq \bar{\kappa}(Y_1)+\bar{\kappa}(F)\geq \bar{\kappa}(Y_1)=\dim Y_1\geq 0.$$  
Hence $\bar{\kappa}(X)=\bar{\kappa}(X_1)\geq 0$. 
The first claim is proved.

\medspace

Let us  prove \cref{item:Mori}.   By \cref{lem:proper}, the special subset ${\rm Sp}(\varrho)$ defined in \cref{dfn:special}  
is a proper closed subset of $X$.  Let $\gamma:\bA^1\to X$ be a non-constant algebraic morphism. Then $\gamma^*\varrho(\pi_1(\bA^1))=\{1\}$. By the definition of  ${\rm Sp}(\varrho)$, we have $\gamma(\bA^1)\subset {\rm Sp}(\varrho)$.

Finally \cref{item:local trivial} follows from \cite[Theorem 12.1]{CDY22}. 
\end{proof}

%\section{Some examples}
%Let $X$ be a smooth cubic surface in $\bP^3$ and we consider a generic projection $X\to \bP^2\subset \bP^3$ with the branched curve $C\subset \bP^2$. This $C$ is  a sextic curve with 6 cusps. By the theorem of Zariski, we know that 
%$$
%\pi_1(\bP^2\backslash S)\simeq \bZ_2\star\bZ_3. 
%$$
%Note that $\bZ_2\star\bZ_3\simeq {\rm PSL}_2(\bZ)$
%Hence we have a faithful representation $\pi_1(\bP^2\backslash S)\to {\rm PSL}_2(\bZ)$ which is Zariski dense in   ${\rm PSL}_{2}$. Here we think of ${\rm PSL}_{2}$ as  a complex algebraic group which is semisimple. 

%\begin{claim}
% $\pi_1(\bP^2\backslash S)$ is big.\footnote{We have to analyze Zariski's proof to show this.}
%\end{claim} 
%\begin{proof}
%Indeed, if it is not big, then there exists a finite \'etale cover $\pi:X'\to \bP^2\backslash S$, a proper birational modification $X\to X'$ and  a dominant morphism $f:X\to Y$ with general fiber connected such that for a general  fiber $F$ of $X$, we have ${\rm Im}[\pi_1(F)\to \pi_1(X)]=\{e\}$.  By the short exact sequence 
%$$
%\pi_1(F)\to \pi_1(X)\to \pi_1(Y)\to 0,
%$$ 
%we conclude that $\pi_1(Y)$ is isomorphic to a finite index subgroup of ${\rm PSL}_2(\bZ)$. 
%\end{proof}

%\begin{rem}
%As we remarked in the introduction, in cases where ${\rm char}\, K=0$, \cref{thm:abelian} was proved by Campana \cite{Cam04} and the second author \cite{Yam10}. When $X$ is quasi-projective, as we show in \cite{CDY22}, Campana's abelianity conjecture has to be revised as nilpotentcy conjecture: a smooth quasi-projective special variety has virtually nilpotent fundamental groups. 
%\end{rem}

\section{On the Shafarevich conjecture and a conjecture by   Claudon-H\"oring-Koll\'ar}
In this section, we will work on projective varieties.   
\subsection{Algebraic varieties with compactifiable universal cover} 
In the paper \cite{CHK13,CH13}, Claudon, H\"oring and Koll\'ar proposed the following intriguing conjecture:
\begin{conjecture}\label{conj}
Let $X$ be a complex projective manifold with infinite fundamental group $\pi_1(X)$. Suppose that the universal cover $\widetilde{X}$ is quasi-projective. Then   after replacing $X$ by a  finite \'etale cover,   there exists a locally trivial fibration $X \rightarrow A$ with simply connected fiber  $F$ onto a complex torus $A$. In particular we have $\widetilde{X}\simeq F \times \mathbb{C}^{\operatorname{dim} A}$.
\end{conjecture}
It's worth noting that assuming abundance conjecture, Claudon, H\"oring and Koll\'ar proved this conjecture in \cite{CHK13}, thereby providing an  unconditional proof for \cref{conj} in cases where $\dim X\leqslant 3$. Additionally, Claudon and H\"oring, in \cite{CH13}, proved \cref{conj} when $\pi_1(X)$ is virtually abelian, a result essential for the proof of \cref{thm:univ}.

As an application of \cref{thm:abelian},  we establish a linear version of \Cref{conj} without relying on the abundance conjecture.   
\begin{cor}\label{thm:univ}
Let $X$ be a smooth projective variety  with an infinite fundamental group $\pi_1(X)$, such that its universal covering  $\widetilde{X}$  is a Zariski open subset  of some compact K\"ahler manifold $\overline{X}$. If  there exists a faithful  representation $\varrho:\pi_1(X)\to \GL_{N}(K),$ where $K$ is any field, then   the Albanese map of $X$ is (up to finite \'etale cover)   locally isotrivial with simply connected fiber $F$.  In particular we have $\widetilde{X} \simeq F \times \mathbb{C}^{q(X)}$  with $q(X)$ the irregularity of $X$. 
\end{cor}
\begin{proof} 
The proof is suggested by a referee of our paper. Consider the core map $$c_X: X \dashrightarrow Y := C(X)$$  defined by Campana, such that  the orbifold base of $c_X$ is of orbifold general type (see \cite{Cam04}). By the orbifold version of the Kobayashi-Ochiai theorem \cite[Theorem 8.2]{Cam04}, the composed meromorphic map
\[
\widetilde{X} \xrightarrow{\pi} X \overset{c_X}{\dashrightarrow} Y
\]
extends to a meromorphic map from a K{\"a}hler compactification $\overline{\widetilde{X}}$ of $\widetilde{X}$. This implies that the general fiber $F$ of $c_X$ satisfies that $\pi^{-1}(F)$ has only finitely many connected components. In particular, the induced homomorphism $\pi_1(F) \rightarrow \pi_1(X)$ has image of finite index.

According to \cite{Cam04}, $F$ is a special manifold. Since we assume the existence of a faithful representation $\varrho: \pi_1(X) \to \GL_N(K)$,   it follows from \cref{main:abelian} when $\mathrm{char}\, K>0$ and  \cite[Theorem 7.8]{Cam04} when $\mathrm{char}\, K=0$ that the image $\varrho([\pi_1(F)\to \pi_1(X)])$ is virtually abelian. As the image of $\pi_1(F) \to \pi_1(X)$ has finite index, $\pi_1(X)$ itself is virtually abelian. The conclusion then follows from \cite[Theorem 1.5]{CH13}, completing the proof of \cref{thm:univ}.
\end{proof}
\begin{rem}\label{rem:referee}
As mentioned above, the proof of \cref{thm:univ} was suggested by a referee, following the original ideas of \cite{CHK13} and \cite[p.~7]{CH10}. Our original proof was significantly more involved and relied on the hyperbolicity arguments in \cref{thm:fun}. 
\end{rem}

\subsection{A factorization result (II)}
In this subsection, we will prove that  any linear representation factors through its Shafarevich morphism after passing to  a finite \'etale cover.  It will be used in the proof of \cref{thm:convexity}.

\begin{lem}\label{lem:20240124} 
Let $X$ be a projective normal variety. 
Let $F\subset X$ be a connected Zariski closed subset.
Then there exists a connected open neighbourhood $U\subset X$ of $F$ such that ${\rm Im}[\pi_1(F)\to\pi_1(X)]= {\rm Im}[\pi_1(U)\to \pi_1(X)]$.
\end{lem}

\begin{rem} 
A stronger result holds as a consequence of \cite[Theorem 4.5]{Hof09}.
Namely there exists a connected open neighbourhood $U\subset X$ of $F$ such that the induced map $\pi_1(F)\to \pi_1(U)$ is an isomorphism.
Here we give a direct proof of \cref{lem:20240124} for the sake of convenience.
\end{rem}

\begin{proof}[Proof of \cref{lem:20240124}]
The construction of $U$ is as follows.
For each $t\in F$, we take a connected open neighbourhood $\Omega_t\subset X$ of $t$ such that 
\begin{enumerate}[label*=(\alph*)]
	\item   any loop in $\Omega_t$ is null homotopic in $X$, 
	\item  $F\cap \Omega_t$ is connected.
\end{enumerate} 
We take a connected open neighbourhood $W_t$ of $t$ such that $\overline{W_t}\subset \Omega_t$.
Since $X$ is projective, $F$ is compact.
We may take $t_1,\ldots,t_l\in F$ such that $F\subset W_{t_1}\cup \cdots\cup W_{t_l}$.
For each $s\in F$, we take a connected open neighbourhood $U_s$ of $s$ such that for each $i=1,\ldots,l$, if $s\in W_{t_i}$, then $U_s\subset W_{t_i}$, if $s\in \overline{W_{t_i}}$ then $U_s\subset \Omega_{t_i}$, and if $s\not\in \overline{W_{t_i}}$, then $U_s\cap \overline{W_{t_i}}=\emptyset$.
We set $U=\cup_{s\in F}U_s$.
Note that $U$ is open and $F\subset U$.
Since $F$ and $U_s$ are connected, $U$ is connected.
We shall show that $U$ satisfies the property of our lemma.

Let $\pi:\widetilde{X}\to X$ be the universal covering space.
Set $\Gamma=\pi_1(X)$.
The action $\Gamma \curvearrowright \widetilde{X}$ is free.
We fix a connected component $F'\subset \widetilde{X}$ of $\pi^{-1}(F)$.
Define $H\subset \Gamma$ by $\gamma\in H$ iff $\gamma F'=F'$.
Note that $F\subset \widetilde{X}$ is an analytic set, hence locally path connected.
Hence $F'\subset \widetilde{X}$ is path connected analytic set.
Thus $H={\rm Im}[\pi_1(F)\to\pi_1(X)]$.

For each $t\in F$, we take $t'\in F'$ such that $t=\pi(t')$.
Then $F'\cap \pi^{-1}(t)=\{\gamma t';\gamma\in H\}$.
For each $i=1,\ldots,l$, we denote by $\Omega_{t_i'}'\subset \widetilde{X}$ the connected component of $\pi^{-1}(\Omega_{t_i})$ which contains $t_i'$.
Since any loop in $\Omega_{t_i}$ is null homotopic in $X$, the natural map $\Omega_{t_i'}'\to \Omega_{t_i}$ is isomorphic.
Then $\pi^{-1}(\Omega_{t_i})=\sqcup_{\gamma\in \Gamma}\gamma \Omega_{t_i'}'$ is disjoint union.
Since $F\cap \Omega_{t_i}$ is connected, $\Omega_{t_i'}'$ does not intersect with $\gamma F'$ for all $\gamma\in\Gamma-H$.
Thus 
\begin{equation}\label{eqn:202401221}
	F'\cap (\cup_{\gamma\in \Gamma-H}\gamma \Omega_{t_i'}')=\emptyset.
\end{equation}
Let $W_{t_i'}'\subset \Omega_{t_i'}'$ be the inverse image of $W_{t_i}$ under $\Omega_{t_i'}'\to\Omega_{t_i}$.
Then $\pi^{-1}(\overline{W_{t_i}})=\cup_{\gamma\in\Gamma}\gamma \overline{W_{t_i'}'}$, which is a closed subset of $\widetilde{X}$.
Set $\mathcal{C}_i=\cup_{\gamma\in\Gamma-H}\gamma \overline{W_{t_i'}'}$.
By $\mathcal{C}_i=\pi^{-1}(\overline{W_{t_i}})\cap (\cup_{\gamma\in H}\gamma \Omega_{t_i'}')^c$, we conclude that $\mathcal{C}_i$ is a closed set.
By \eqref{eqn:202401221}, we have $F'\cap \mathcal{C}_i=\emptyset$.
We set $\mathcal{C}=\mathcal{C}_{1}\cup\cdots\cup\mathcal{C}_{l}$.
Then $\mathcal{C}\subset \widetilde{X}$ is a closed subset and $F'\cap \mathcal{C}=\emptyset$.

Now for each $t\in F$, we denote by $U'_{t'}\subset \widetilde{X}$ the connected component of $\pi^{-1}(U_t)$ which contains $t'$.
We claim that if $\tau\in H$, then $\tau U'_{t'}\cap \mathcal{C}=\emptyset$, and if $\tau\in \Gamma-H$, then $\tau U'_{t'}\subset \mathrm{Int}(\mathcal{C})$, where $\mathrm{Int}(\mathcal{C})$ is the interior of $\mathcal{C}$.
To prove this, we first suppose $\tau\in H$.
To show $\tau U'_{t'}\cap \mathcal{C}=\emptyset$, it is enough to show that $\tau U'_{t'}\cap \gamma \overline{W_{t'_i}'}=\emptyset$ for each $\gamma\in \Gamma-H$ and $i=1,\ldots,l$.
Indeed, if $t\in \overline{W_{t_i}}$, then $U_t\subset \Omega_{t_i}$.
By $\gamma\not=\tau$, we have $\gamma\Omega'_{t'_i}\cap \tau\Omega'_{t'_i}=\emptyset$.
Hence $\tau U'_{t'}\cap \gamma \overline{W_{t'_i}'}=\emptyset$.
If $t\not\in \overline{W_{t_i}}$, then $U_t\cap  \overline{W_{t_i}}=\emptyset$.
Thus $\tau U'_{t'}\cap \gamma \overline{W_{t'_i}'}=\emptyset$.
We have proved that $\tau U'_{t'}\cap \mathcal{C}=\emptyset$, if $\tau\in H$.
Next suppose $\tau\in \Gamma-H$.
We may take $i$ such that $t\in W_{t_i}$.
Then by $U_t\subset W_{t_i}$, we have $\tau U'_{t'}\subset \tau W'_{t'_i}\subset \mathrm{Int}(\mathcal{C})$.

Now we set $U'=\cup_{t\in F}\cup_{\gamma\in H}\gamma U'_{t'}$.
Then we have $F'\subset U'$.
Since $F'$ and $\gamma U'_{t'}$ are connected, $U'$ is connected.
Note that $U'\cap \mathcal{C}=\emptyset$ and $\gamma U'=U'$ for all $\gamma\in H$.
On the other hand, for all $\gamma\in \Gamma-H$, we have $\gamma U'\subset \mathrm{Int}(\mathcal{C})$.
By $\pi^{-1}(U)=\cup_{\gamma\in \Gamma}\gamma U'$, we conclude that $U'$ is a connected component of $\pi^{-1}(U)$ and $\varrho({\rm Im}[\pi_1(U)\to \pi_1(X)])=H$.
The proof is completed.
\end{proof}

\begin{lem}\label{lem:20230120}
Let $f:X\to S$ be a proper surjective map between quasi-projective varieties 
%	\footnote{I omitted the assumption that $X$ and $S$ are normal.
	%	I think it is not needed.}
$X$ and $S$ such that all fibers of $f$ are connected.
Let $\varrho:\pi_1(X)\to \Gamma$ be a morphism to a group $\Gamma$.
Assume that for every fiber $F$ of $f$, the image $\varrho({\rm Im}[\pi_1(F)\to \pi_1(X)])$ is trivial.
Then there exists a morphism $\tau:\pi_1(S)\to \Gamma$ such that $\tau\circ f_*=\varrho$.
\end{lem}

\begin{proof} 
Let $p:X'\to X$ be a Galois covering corresponding to $\mathrm{ker}(\varrho)\subset\pi_1(X)$.
Replacing $\Gamma$ by $\varrho(\pi_1(X))$, we may assume that $\varrho$ is surjective.
Then $\mathrm{Gal}(X'/X)=\Gamma$.
The action $\Gamma\curvearrowright X'$ is properly discontinuous.
Let $g:X'\to S$ be the composite of $p:X'\to X$ and $f:X\to S$.
We define an equivalence relation $\sim$ on $X'$ by $x\sim y$ iff $g(x)=g(y)$ and $y$ is contained in the connected component of $g^{-1}(g(x))$ which contains $x$.
Let $\varphi:X'\to S'$ be the quotient map endowed with the quotient topology on $S'$.
Then we have the induced map $q:S'\to S$, which is continuous.
Each $\gamma\in \Gamma$ preserves the equivalence relation $\sim$ on $X'$, namely if $x\sim y$, then $\gamma x\sim \gamma y$. 
Hence the action $\Gamma\curvearrowright X'$ descends to $\Gamma\curvearrowright S'$ as a continuous action.

We shall prove that $q:S'\to S$ is a topological Galois covering with $\mathrm{Gal}(S'/S)=\Gamma$.
We take $s\in S$.
We take a connected open neighborhood $U_s\subset X$ of $f^{-1}(s)$ as in \cref{lem:20240124}.
Then 
\begin{equation}\label{eqn:20250705}
	\varrho({\rm Im}[\pi_1(U_s)\to \pi_1(X)])=\{1\}.
\end{equation}
Since $f:X\to S$ is proper, we may take an open neighborhood $V_s\subset S$ of $s\in S$ such that $f^{-1}(V_s)\subset U_s$.
We replace $U_s$ by $f^{-1}(V_s)$.
Let $W\subset X'$ be a connected component of $p^{-1}(U_s)$.
Then by \eqref{eqn:20250705}, the induced map $p|_{W}:W\to U_s$ is a homeomorphism.
We have a disjoint union
\begin{equation}\label{eqn:202507051}
	p^{-1}(U_s)=\bigcup_{\gamma\in \Gamma}\gamma W.
\end{equation}
Set $\Omega=\varphi(W)$.
Then $\varphi^{-1}(\Omega)=W$, hence $\Omega\subset S'$ is open.
Since every fiber of $f:X\to S$ is connected, the induced map $q|_{\Omega}:\Omega\to V_s$ is injective.
Since $f:X\to S$ is surjective, $q|_{\Omega}:\Omega\to V_s$ is surjective, hence bijective.
Since $f:X\to S$ is proper, $(q|_{\Omega})^{-1}:V_s\to \Omega$ is continuous, hence $q|_{\Omega}:\Omega\to V_s$ is homeomorphic.
By \eqref{eqn:202507051}, we have a disjoint union $q^{-1}(V_s)=\bigcup_{\gamma\in \Gamma}\gamma \Omega$.
Hence $q:S'\to S$ is a topological Galois covering with $\mathrm{Gal}(S'/S)=\Gamma$.
Thus we get $\tau:\pi_1(S)\to\Gamma$ such that $\tau\circ f_*=\varrho$.
\end{proof}

\begin{lem}\label{lem:factor2}
Let $f:X\to S$ be a proper surjective map between projective normal varieties $X$ and $S$ such that all fibers of $f$ are connected. 
Let $\varrho:\pi_1(X)\to \mathrm{GL}_n(K)$ be a representation, where $K$ is any field.
Assume that for every fiber $F$ of $f$, the image $\varrho({\rm Im}[\pi_1(F)\to \pi_1(X)])$ is finite.
Then there exists a finite \'etale covering $\mu:X'\to X$ with the following property:
Let $X'\to S'\to S$ be the  Stein factorization of the composite $X'\to X\to S$.
Then for every fiber $G$ of $g:X'\to S'$, the image $\mu^*\varrho({\rm Im}[\pi_1(G)\to \pi_1(X')])$ is trivial.
\end{lem}

\begin{proof} 
We take $s\in S$.
Since $\varrho(\pi_1(X))$ is residually finite and $$\varrho({\rm Im}[\pi_1(f^{-1}(s))\to \pi_1(X)])$$ is finite, we may take a finite indexed normal subgroup $N_s\subset \varrho(\pi_1(X))$ such that $$\varrho({\rm Im}[\pi_1(f^{-1}(s))\to \pi_1(X)])\cap N_s=\{1\}.$$
We take a connected open neighborhood $U_s\subset X$ of $f^{-1}(s)$ as in \cref{lem:20240124}.
Then $$\varrho({\rm Im}[\pi_1(U_s)\to \pi_1(X)])\cap N_s=\{1\}.$$
We note that $f:X\to S$ is proper.
Hence, we may take an open neighborhood $V_s\subset S$ of $s\in S$ such that $f^{-1}(V_s)\subset U_s$.
Then for $t\in V_s$, we have $f^{-1}(t)\subset U$, hence $$\varrho({\rm Im}[\pi_1(f^{-1}(t))\to \pi_1(X)])\cap N_s=\{1\}.$$
Since $S$ is compact, we may take finite points $s_1,\ldots,s_k\in S$ such that $S=V_{s_1}\cup\cdots\cup V_{s_k}$. 
Set $N=N_{s_1}\cap\cdots\cap N_{s_k}$.
Then for every fiber $F$ of $f$, we have $\varrho({\rm Im}[\pi_1(F)\to \pi_1(X)])\cap N=\{1\}$.

Let $\mu:X'\to X$ be a Galois covering corresponding to $\varrho^{-1}(N)\subset\pi_1(X)$.
Let $X'\to S'\to S$ be the  Stein factorization of the composite $X'\to X\to S$.
Then for every fiber $G$ of $g:X'\to S'$, we have ${\rm Im}[\pi_1(G)\to \pi_1(X)]\subset {\rm Im}[\pi_1(F)\to \pi_1(X)]$, where $F$ is the fiber of $f$ such that $\mu(G)=F$.
Hence $\varrho({\rm Im}[\pi_1(G)\to \pi_1(X)])\cap N=\{1\}$.
By $N=\varrho(\pi_1(X'))$, we conclude $\mu^*\varrho({\rm Im}[\pi_1(G)\to \pi_1(X')])=\{1\}$. 
\end{proof}

\begin{proposition}\label{prop:factor}
Let $X$ be a projective normal variety and let $\varrho:\pi_1(X)\to \GL_{N}(K)$ be a representation where $K$ is any field of positive characteristic. Then there exists a finite \'etale cover $\mu:\widehat{X}\to X$  and  a large representation $\tau:\pi_1({\rm Sh}_{\mu^*\varrho}(\widehat{X}))\to \GL_{N}(K)$ such that $({\rm sh}_{\mu^*\varrho})^*\tau=\mu^*\varrho$, where  ${\rm sh}_{\mu^*\varrho}:\widehat{X}\to {\rm Sh}_{\mu^*\varrho}(\widehat{X})$ is the Shafarevich morphism of $\mu^*\varrho$. 
\end{proposition}
\begin{proof}
By \cref{thm:Sha22}, the Shafarevich morphism ${\rm sh}_\varrho:X\to {\rm Sh}_\varrho(X)$ exists, and for each fiber $F$ of  ${\rm sh}_\varrho$, $\varrho({\rm Im}[\pi_1(F)\to \pi_1(X)])$ is finite. 
By virtue of \cref{lem:factor2},  there exists a finite \'etale cover $\mu:\widehat{X}\to X$ such that considering the Stein factorization  $\widehat{X} \stackrel{g}{\to} S\to {\rm Sh}_\varrho(X)$   of the composite ${\rm sh}_\varrho\circ\mu$,  
for every fiber $G$ of $g $, the image $\mu^*\varrho({\rm Im}[\pi_1(G)\to \pi_1(\widehat{X})])$ is trivial.   Applying \cref{lem:20230120}, we conclude that there exists a representation $\tau:\pi_1(S)\to \GL_{N}(K)$ such that $g^*\tau=\mu^*\varrho$. 

We note that for each fiber $G$ of $g$, $g(G)$ is a point if and only if $$\mu^*\varrho({\rm Im}[\pi_1(G^{\rm norm})\to \pi_1(\widehat{X})])$$ is finite.  By the unicity of the Shafarevich morphism,    $g:\widehat{X}\to S$ is identified with the Shafarevich morphism ${\rm sh}_{\mu^*\varrho}:\widehat{X}\to {\rm Sh}_{\mu^*\varrho}(\widehat{X})$ of $\mu^*\varrho$. 

Let us prove that $\tau$ is large. Let $Z\subset  S$ be a positive dimensional closed subvariety. Let $W\subset  g^{-1}(Z)$ be an irreducible component that is dominant over $Z$. It follows that $\mu^*\varrho({\rm Im}[\pi_1(W^{\rm norm})\to \pi_1(\widehat{X})])$ is infinite by \cref{thm:Sha22}. Since $g^*\tau=\mu^*\varrho$, it follows that
$$
\mu^*\varrho({\rm Im}[\pi_1(W^{\rm norm})\to \pi_1(\widehat{X})])\subset \tau({\rm Im}[\pi_1(Z^{\rm norm})\to \pi_1(S)])
$$ 
is infinite. 
This implies that $\tau$ is a large representation. 
\end{proof}

In the rest of this paper, we will prove \cref{main:Sha}.

\subsection{Partial Albanese morphism}\label{sec:partial}
For this subsection we refer the readers to  \cite{Eys04,Eys11}  for details. Let $X$ be a smooth projective variety.  Let $\{\eta_1,\ldots,\eta_k\}\subset H^0(X,\Omega_X^1)$ be a set of holomorphic one forms. Consider the Albanese map $\alb_X: X\to \Alb(X)$. Then there exist holomorphic one forms $\{\omega_1,\ldots,\omega_k\}\subset H^0(\Alb(X),\Omega_{\Alb(X)}^1)$  such that  $\alb_X^*\omega_i=\eta_i$ for each $i$. Let $B$ be the largest abelian subvariety in $\Alb(X)$ such that $\omega_i|_{B}\equiv 0$ for each $i$.  Denote by $A:=\Alb(X)/B$ the quotient which is also an abelian variety. Then the morphism $a:X\to A$ that is the composite  of $\alb_X$ and the quotient map $q:\Alb(X)\to A$  is called the \emph{partial Albanese morphism} induced by $\{\eta_1,\ldots,\eta_k\}$.

We remark that there exist  holomorphic 1-forms $\{\omega'_1,\ldots,\omega'_k\}\subset H^0(A,\Omega_{A}^1)$ such that $q^*\omega_i'=\omega_i$ for each $i$.  They satisfy the following property: 
\begin{lem}\label{lem:snef}
For any positive-dimensional closed subvariety $Z$ of $a(X)$, there exists some $\omega_i'$ such that $\omega_i'|_{Z}\not\equiv 0$.  
\end{lem}
\begin{proof}
Since $a:X\to a(X)$ is surjective, there exists an irreducible component $W$ in $a^{-1}(Z)$ that is dominant over $Z$.  By the property of the partial Albanese map proved in \cite[Lemma 1.1]{CDY22},   there exists some $\omega_i$ such that $\omega_i|_{W}\not\equiv 0$.   Since $\omega_i=a^*\omega'_i$, it follows that $\omega_i'|_{Z}\not\equiv 0$.  
\end{proof}
\subsection{A recollection of spectral covering and canonical currents}
Let $X$ be a smooth projective variety.  Let $\tau:\pi_1(X)\to \GL_{N}(K)$ be a reductive representation where $K$ is a non-archimedean local field.   According to \cref{thm:KZ}, the \emph{Katzarkov-Eyssidieux reduction map} $s_\tau:X\to S_\tau$ of $\tau$ exists, fulfilling the properties outlined therein.   
Since $X$ is projective, $s_\tau:X\to S_\tau$ is unique.
We will outline the construction of certain  \emph{canonical}  positive closed $(1,1)$-currents  over $S_{\tau}$.  %We first assume that the Zariski closure of  $\varrho(\pi_1(X))$ is   semisimple. 

We first recall some facts about \emph{spectral forms}  as we have already seen in \Cref{itea} in the proof of \cref{thm:GGL}. We refer the readers to \cite{Eys04,CDY22,DY23} for more details. Let   $$\tau:\pi_1(X)\to \GL_{N}(K)$$ be a reductive representation where $K$ is a non-archimedean local field.  
By the work of Gromov-Schoen \cite{GS92} (see also \cite{BDDM} in cases where $X$ is non-compact), there exists a $\tau$-equivariant harmonic mapping $u:\widetilde{X}\to \Delta(G)$ where $\Delta(G)$ is the (enlarged) Bruhat-Tits building of $G$ (see \cite[Definition 4.3.2]{KP23} for the definition). Such a harmonic map is pluriharmonic, and the $(1,0)$-part of the complexified differentials gives rise to a set of multivalued holomorphic 1-forms over a dense  open set of $X$ whose complement has Hausdorff codimension at least two. 
We can take some finite (ramified) Galois covering $\pi:\xsp\to X$ (so-called \emph{spectral covering}) with  Galois group $H$ and such that the pullback of these multivalued one forms becomes single valued ones $\{\eta_1,\ldots,\eta_k\}\subset H^0(\xsp, \pi^*\Omega_X^1)$, that are called the \emph{spectral 1-forms} of $\tau$.  
Consequently, the Stein factorization of the \emph{partial Albanese morphism} $a:\xsp\to A$ of $\{\eta_1,\ldots,\eta_k\}$  leads to the (Katzarkov-Eyssidieux) reduction map $s_{\pi^*\tau}:\xsp\to S_{\pi^*\tau}$ of $\pi^*\tau$. This map  $s_{\pi^*\tau}$ is $H$-equivariant and its quotient by $H$ gives rise to the reduction map $s_\tau:X\to S_\tau$ by $\tau$.  More precisely, we have the  following commutative diagram:
\begin{equation*}
\begin{tikzcd}
	\xsp\arrow[r, "\pi"] \arrow[d, "s_{\pi^*\tau}"]\arrow[dd,bend right=37, "a"'] &  X\arrow[d, "s_{\tau}"]\\
	S_{\pi^*\tau} \arrow[d, "b"]\arrow[r, "\sigma_\pi"] & S_{\tau}\\
	A &
\end{tikzcd}
\end{equation*}
Here  $\sigma_\pi$ is also a finite ramified Galois cover with Galois group $H$. 
Note that there are 1-forms  
$\{\eta_1',\ldots,\eta_m'\}\subset H^0(A, \Omega_A^1)$ such that $a^*\eta_i'=\eta_i$. %Consider the finite morphism $b: S_{\pi^*\tau}\to A$. Then  
We define a smooth  positive 
closed $(1,1)$-form $T_{\pi^*\tau}:=b^*\sum_{i=1}^{m}i\eta_i'\wedge\overline{\eta_i}'$ on $S_{\pi^*\tau}$. Note that $T_{\pi^*\tau}$ is invariant under the Galois action $H$. Therefore,  there is a positive  
closed $(1,1)$-current $T_{\tau}$ defined on $S_\tau$ with continuous local potential such that $\sigma_\pi^*T_{\tau}=T_{\pi^*\tau}$. 
\begin{dfn}[Canonical current]\label{def:canonical}
The closed positive $(1,1)$-current 	$T_\tau$ on $S_\tau$ is called the \emph{canonical current} of $\tau$. 
\end{dfn}
\begin{lem}\label{lem:strictly nef}
$\{T_\tau\}$ is strictly nef. Namely, for any irreducible curve $C\subset S_\tau$, we have  $\{T_\tau\}\cdot C>0$. 
\end{lem}
\begin{proof}
Let $C'\subset \sigma_\pi^{-1}(C)$ be an irreducible component which is dominant over $C$. Consider its image $b(C')$. By \cref{lem:snef},  there exists some $\eta_i'\in H^0(A,\Omega_A^1)$ such that $\eta_i'|_{b(C')}\neq 0$. Hence $i\eta_i'\wedge\overline{\eta'_i}|_{b(C')}$ is strictly positive at general  points. Consequently,  $\{T_\tau\}\cdot C>0$. 
\end{proof} 

%The partial Albanese morphism $\xsp\to A$ of $ \{\eta_1,\ldots,\eta_k\}$ is $H$-equivariant. The reduction map $s_\tau$ is the Stein factorization of the quotient $X\to A/H$.  Consider the closed positive form $\sum_{i=1}^{k}i\eta_i\wedge\bar{\eta}_i$, that is invariant under the Galois group $H$. It descends to a    positive closed $(1,1)$-current $T_\tau$ over $X$ with continuous potentials. Such a current $T_\tau$ is called the \emph{canonical current} of $\tau$.  By \cite{Eys04} (or \cite{DY23}),     the harmonic mapping $u$ induces  a continuous plurisubharmonic function $\phi_\tau:\widetilde{X}\to \bR_{\geq 0}$  such that 
%$$
%\hess \phi_\tau\geq \pi_X^*T_\tau
%$$
%where $\pi_X:\widetilde{X}\to X$ is the universal covering map. Moreover, there exists a positive closed $(1,1)$-current $T'_\tau$ on $S_\tau$ such that $s_\tau^*T'_\tau=T_\tau$.  

The canonical current $T_{\varrho}$ will serve as a lower bound for the complex hessian of plurisubharmonic functions constructed by the method of harmonic mappings. 
\begin{proposition}[{\cite[Proposition 3.3.6, Lemme 3.3.12]{Eys04}}]\label{prop:Eys}
Let   $X$ be a projective normal variety and let   $\varrho:\pi_1(X)\to G(K)$ be a Zariski dense representation where $K$ is a non archimedean local field and $G$ is a reductive group.	   
Let $x_0 \in \Delta(G)$ be an arbitrary point.  Let $u: \widetilde{X}\to \Delta(G)$ be the associated harmonic mapping, where $\widetilde{X}$ is  the universal covering of $X$. 
The function $\phi: \widetilde{X} \rightarrow \mathbb{R}_{\geq 0}$ defined by
$$
\phi(x)=2d^2\big(u(x), u(x_0)\big)
$$
satisfies the following properties:
\begin{enumerate}[label=(\alph*)]
	\item   $\phi$ descends to a function $\phi_{\varrho}$ on $\widetilde{X_{\varrho}}=\widetilde{X}  / \operatorname{ker}\left(\varrho\right)$.
	%	\item $\hess\phi_\varrho\geq (s_\varrho\circ\pi)^*T_\varrho$,  where  we denote by $\pi:\widetilde{X}_\varrho\to X$ the covering map. 
	%	\item Let $\pi: \widetilde{X}^{\rm un} \rightarrow \widetilde{X} \rightarrow \widetilde{X}_{\varrho}$ be a covering space of $X$ dominating $\widetilde{X}_{\varrho}$. By abuse of notation we denote by $\phi_{\varrho}$ the function $\phi_{\varrho}^* \circ \pi$.
	%	\item $\phi_{\varrho}$ is locally Lipschitz;
	\item  \label{item: descends} Let $\Sigma$ be a normal complex space and $r:\widetilde{X}_\varrho \to \Sigma$ a proper holomorphic fibration such that $s_{\varrho} \circ \pi: \widetilde{X}_\varrho \rightarrow S_{\varrho} $ factorizes via a morphism $\nu:\Sigma {\rightarrow} S_\varrho$. 
	The function $\phi_{\varrho}  $ is of the form $\phi_{\varrho}=\phi_\varrho^{\Sigma} \circ r$  with $\phi_\varrho^{\Sigma}$ being a continuous  plurisubharmonic function on $\Sigma$;  
	\item $\hess\phi^{\Sigma}_\varrho \geq \nu^* T_{\varrho}$.\qed
\end{enumerate}
\end{proposition}

We require the following criterion for an infinite topological Galois covering of a compact complex normal space to being Stein. 

\begin{proposition}[{\cite[Proposition 4.1.1]{Eys04}}]\label{prop:stein}
Let $S$ be a compact complex normal space and let $\nu:\Sigma\to S$ be some infinite topological Galois covering.  
Let $T$ be a closed positive $(1,1)$-current on $S$  with continuous potential such that $\{T\}$ is a K\"ahler class. 
Assume that there exists a continuous plurisubharmonic function $\phi: \Sigma\to \bR_{\geq 0}$ such that $\hess \phi\geq \nu^*T$.  Then $\Sigma$ is a Stein space.  \qed
\end{proposition}
\subsection{Some lemmas in algebraic groups}
We start with the following lemma. 
\begin{lem}\label{claim:same closure}
Let $\Gamma_1$ be a finite index subgroup of $\Gamma_2$. Let $G$ be a linear algebraic group over a field $K$.  Assume that  $\Gamma_2\subset G(K)$ is Zariski dense. 
Then for the Zariski closure $H$ of $\Gamma_1$ 
in $G$, we have $H^o=G^o$, where $H^o$ and $G^o$ are the identity components of $H$ and $G$ respectively.   
\end{lem}
\begin{proof}
We may assume that $K$ is algebraically closed.  Replacing $\Gamma_1$ and $\Gamma_2$ by some finite index subgroups, we may assume that their Zariski closures are $H^o$ and $G^o$ respectively.  Let $x_1\Gamma_1,\ldots,x_k\Gamma_1$ be left cosets of $\Gamma_1$ in $\Gamma_2$.  It follows that $x_i\Gamma_1$ is contained in the Zariski closed subset $x_iH$.  Hence $G^o/H^o$ is finite. Since $G^o$ and $H^o$ are both connected, it implies that $G^o=H^o$. The lemma is proved. 
\end{proof}  
The following lemma plays a crucial role in the proof of  \cref{thm:convexity}.  It is a variant of \cite[Lemma 4.8]{DY23} and its proof is based on \cite[Lemma 5.3]{CDY22}. 
\begin{lem}\label{lem:BT}
Let $G$ be a semisimple algebraic group over a  
non-archimedean local field $K$.  Let $\Gamma\subset G(K)$ be a finitely generated subgroup  which is   Zariski dense  in $G$. 
If  its derived group $\cD \Gamma$ is  \emph{bounded}, then  $\Gamma$ is also  bounded. 
\end{lem}
\begin{proof}
We can replace $\Gamma$ by a finite index subgroup   and assume that $G$ is connected.  Since $G$ is semisimple,  according to the decomposition theorem \cite[Theorem 21.51]{Mil17} there are finitely many  almost $K$-simple normal subgroups $H_1,\ldots, H_k$ of $G$, such that $H_1\times\cdots\times H_k\to G$ is a central isogeny. Hence each quotient $G_i:=G/H_1\cdots H_{i-1}H_{i+1}\cdots H_k$ is  an  almost $K$-simple algebraic group.   Let $\Gamma_i$    be the image  of $\Gamma$ under the homomorphism $q_i:G(K)\to G_i(K)$. Then  $\Gamma_i$ is Zariski dense in $G_i$, and  we have $\cD \Gamma_i= q_i(\cD \Gamma)$.  Since $\cD \Gamma$ is bounded, by \cite[Fact 2.2.4]{KP23},  $\cD \Gamma_i$ is also bounded.    

To show that $\Gamma_i$ is bounded, we assume contrary that $\Gamma_i$ is unbounded.
Since   $\cD\Gamma_i$ is bounded and $G_i$ is almost $K$-simple, we apply  \cite[Lemma 5.3]{CDY22} to conclude that $\cD\Gamma_i$ is finite. 
Let $\Gamma'_i$ by a finite index subgroup of $\Gamma_i$ such that $\cD \Gamma'_i$ is trivial. 
Then $\Gamma'_i$ is commutative.
Hence the Zariski closure of $\Gamma_i'\subset G_i(K)$ is commutative.
Since  $\Gamma_i$ is Zariski dense in $G_i$, by \cref{claim:same closure} we conclude that the identity component $G^\circ_i$ of $G_i$ is commutative. 
This contradicts with that $G_i$ is almost $K$-simple. 
Hence $\Gamma_i$ is bounded for each $i$.  

Note that the natural morphism $G\to G_1\times\cdots\times G_k$ is also an isogeny. As a result, $\Gamma$ is bounded. 
\end{proof}

\subsection{Holomorphic convexity of universal covering}

We recall the following criterion for an infinite topological Galois covering of a compact complex normal space to be Stein. As it is well known, we omit the proof. 

\begin{lem}\label{lem:Stein}
Let $X$ be a  projective normal variety and let  $a:X\to A$ be a finite morphism to an abelian variety $A$. Then the universal covering  $\widetilde{X}$ of $X$ is a Stein space.\qed
\end{lem}
%\begin{proof}
%	Let $\pi_A:\widetilde{A}\to A$ be the universal covering map and let $X'$ be a connected component of $X\times_A\widetilde{A}$.  Then we have the following commutative diagram
%	\begin{equation*}
%	\begin{tikzcd}
	%	X'\arrow[r, "\mu"] \arrow[d, "f"] & X\arrow[d, "a"] \\
	%		\widetilde{A}\arrow[r, "\pi_A"] & A
	%	\end{tikzcd}
%	\end{equation*} 
%where $\mu:X'\to X$ and $f:X'\to \widetilde{A}$ are the induced maps. Then $\mu$ is \'etale and $f$ is finite. Hence the image $f(X')$ is a closed subvariety of $\widetilde{A}\simeq \bC^{\dim A}$, which is thus a Stein space. Since $f:X'\to f(X')$ is finite, it follows that $X'$ is also Stein. Note that any unramified covering of  a Stein space is Stein. Therefore, $\widetilde{X}$ is a Stein space. 
%\end{proof} 

We recall the notion of character varieties from sections \cref{sec:2.2} and \cref{sec:2.3}.
In particular, we set $M_{\rm B}(X,N)_{\bF_p}=M_{\rm B}(\pi_1(X),N)_{\bF_p}$.

\begin{thm}\label{thm:Stein}
Let $X_0$ be a  complex projective  normal surface and let $p$ be a fixed prime and $N$ be a fixed positive integer. We assume that  $M_{\rm B}(X_0,N)_{\bF_p}$ is large, meaning that for any positive dimensional closed subvariety $Z$ of $X_0$, there exists a linear representation $\varrho:\pi_1(X_0)\to \GL_{N}(K)$ with ${\rm char}\, K=p$ such that $\varrho({\rm Im}[\pi_1(Z^{\rm norm})\to \pi_1(X_0)])$ is infinite.  Then the universal covering of $X_0$ is Stein.  
\end{thm}
\begin{proof} 
Let $\mu:X\to X_0$ be a desingularization.  %Let $\varrho:=\mu^*\varrho_0$. Since $\varrho_0:\pi_1(X)\to \GL_{N}(K)$ is large, it follows that $ \mu$ coincides with the Shafarevich morphism ${\rm sh}_\varrho:X\to {\rm Sh}_\varrho(X)$.  
%	 Then 
It induces a morphism $$\iota:M_{\rm B}(X_0,N)_{\bF_p}\hookrightarrow \mxp$$ between affine $\bF_p$-schemes of finite type which is a closed immersion. Let $M:=\iota(M_{\rm B}(X_0,N)_{\bF_p})$ which is a Zariski closed subset of $\mxp$.
%Then $[\varrho]\in M(K)$.
%As we assume that $\varrho_0$ is faithful, it implies that $\mu^*\sigma'(\Gamma)$ is finite if $\mu^*\varrho_0(\Gamma)$ is finite. Therefore, $ \sigma(\Gamma)$ is finite  if $\varrho(\Gamma)$ is finite.  
%Write $M_0:=M_{\rm B}(X_0,N)_{\bF_p}$. By 
%	Thanks to \cref{prop:factor}, we can replace $X$ by a finite \'etale cover such that $\varrho$ factors through its Shafarevich morphism ${\rm sh}_\varrho:X\to {\rm Sh}_\varrho(X)$. Namely, there exists a representation $\varrho':\pi_1({\rm Sh}_\varrho(X))\to \GL_{N}(K)$ such that ${\rm sh}_\varrho^*\varrho'=\varrho$.    
By \cref{thm:Sha1}, the Shafarevich morphism $s_M:X\to S_M$  
of $M$ exists. 

\begin{claim}
	The Shafarevich morphism $s_M:X\to S_M$  
	coincides with  $\mu:X\to X_0$. 
\end{claim}
\begin{proof}
	Let $Z$ be any closed subvariety of $X$ and let $W:=\mu(Z)$.  
	If $s_M(Z)$ 
	is a point, then for any representation \( \sigma : \pi_1(X_0) \to \GL_N(L) \), where \( L \) is any field of characteristic \( p \) and \( [\sigma] \in M_{\mathrm{B}}(X_0, N)_{\mathbb{F}_p}(L) \), the image \( \mu^* \sigma(\mathrm{Im}[\pi_1(Z^{\mathrm{norm}}) \to \pi_1(X)]) \) is finite. 
	Therefore,   $\sigma({\rm Im}[\pi_1(W^{\rm norm})\to \pi_1(X_0)])$ is finite. Hence $W$ is a point since $M_{\rm B}(X_0,N)_{\bF_p}$  is assumed to be large. 
	
	On the other hand, assume that $\mu(Z)$ is a point.   
	Let  $\sigma:\pi_1(X)\to \GL_{N}(L)$ be any representation where $L$ is a field of characteristic $p$ such that $[\sigma]\in M(L)$.  
	Then there exists a representation $\sigma':\pi_1(X_0)\to \GL_{N}(\bar{L})$  such that $[\mu^*\sigma']=[\sigma]$. 
	By \cref{lem:finite group}, for any subgroup $\Gamma\subset \pi_1(X)$, $\mu^*\sigma'(\Gamma)$ is finite if and only if $\sigma(\Gamma)$ is finite. Since $\mu^*\sigma'({\rm Im}[\pi_1(Z^{\rm norm})\to \pi_1(X)])$   is trivial, it follows that $\sigma({\rm Im}[\pi_1(Z^{\rm norm})\to \pi_1(X)])$ is finite. 		
	By the properties of $s_M$ proven in \cref{thm:Sha1}, $s_M(Z)$ is a point.  
	This proves that $s_M=\mu$. 
\end{proof} 

Let $H_0:=\cap_{\varrho}\ker\varrho$ where $\varrho:\pi_1(X_0)\to \GL_{N}(L)$  ranges over all linear representations 
and $L$ is any field  with ${\rm char}\, L=p$.	
Denote by $H:=\mu_*^{-1}H_0$, where $\mu_*:\pi_1(X)\to\pi_1(X_0)$ is the induced morphism. 
Let $\pi_H:\widetilde{X}_H\to X$ be the Galois covering with the Galois group $\pi_1(X)/H$ and $\pi_0:\widetilde{X}_0\to X_0$ be that with $\pi_1(X_0)/H_0$.
Consequently, we have the following commutative diagram 
\begin{equation}\label{eq:Cartan}
	\begin{tikzcd}
		\widetilde{X}_H\arrow[r, "\pi_H"] \arrow[d, "f"] & X\arrow[d, "\mu"] \arrow[dr, "s_M"] &\\
		\widetilde{X}_0 \arrow[r, "\pi_0"]& X_0 \arrow[r,equal] &S_M
	\end{tikzcd}
\end{equation}
where $f:\widetilde{X}_H\to \widetilde{X}_0$ is a holomorphic proper fibration.

\begin{claim}\label{claim:20240229}
	Let $\sigma:\pi_1(X)\to \GL_{N}(L)$ be a linear representation such that $[\sigma]\in M(L)$, where $L$ is a field of characteristic $p$.
	Then $H\subset \mathrm{ker}(\sigma^{ss})$.
\end{claim}
\begin{proof}
	There exists a linear representation $\sigma':\pi_1(X_0)\to \GL_{N}(\bar{L})$  such that $[\mu^*\sigma']=[\sigma]$. 
	Then $(\mu^*\sigma')^{ss}$ is conjugate to $\sigma^{ss}$.
	Hence $\mathrm{ker}(\mu^*\sigma')\subset \mathrm{ker}((\mu^*\sigma')^{ss})=\mathrm{ker}(\sigma^{ss})$.
	We have $H\subset \mu^*\mathrm{ker}(\sigma')=\mathrm{ker}(\mu^*\sigma')$.
	Thus $H\subset \mathrm{ker}(\sigma^{ss})$.
\end{proof}

Consider  the affine $\bF_p$-scheme  $R_{\bF_p}$ and the GIT quotient $\pi:R_{\bF_p}\to \mxp$, as defined in \cref{sec:factor}. Define $R:=\pi^{-1}(M)$, which is an affine $\overline{\bF_p}$-scheme of finite type.  
According to the structure of the Shafarevich morphism described in \cref{thm:Sha1} and \cref{def:reduction ac2}, the Shafarevich morphism $s_M:X\to S_M$  
of $M$  is   obtained through the simultaneous Stein factorization of the reductions   $\{s_{\tau}:X\to S_\tau\}_{[\tau]\in M(K)}$. 
Here   $\tau:\pi_1(X)\to \GL_N(K)$ ranges over all reductive representations with $K$  a local field of characteristic $p$ such that  $[\tau] \in M(K)$ and $s_\tau:X\to S_\tau$ is the reduction map defined in \cref{thm:KZ}.   %Thanks to \cref{lem:simultaneous}, there exists finitely many   representations  $ \{\tau_i:\pi_1(X) \to \GL_N\big( \bF_{q_i}((t))\big)\}_{i=1,\ldots,m}$ such that ${\rm Sh}_\varrho$ is the Stein factorization of 
%$$
%(s_{\tau_1},\ldots,s_{\tau_m}):X\to S_{\tau_1}\times\cdots\times S_{\tau_m}.$$

%	Consider any irreducible curve $T\subset  R$ defined over $\overline{\bF_p}$. According to  \cref{sec:factor}, its reduction map  $s_T:X\to S_T$ is induced by a family of representations  $\btau:= \{\tau_i:\pi_1(X) \to \GL_N\big( \bF_{q_i}((t))\big)\}_{i=1,\ldots,\ell}$. Consequently, $s_\btau:X\to S_\btau$ coincides with $s_T$.  

%	By applying \cref{lem:same bound}, we can replace each $\tau_i$ by its semisimplification without changing its reduction map.   Then the Zariski closure of $\tau_i$ is reductive.    It follows that   $s_T:X\to S_T$   coincides with the reduction map $s_\btau:X\to S_\btau$. 
%	\begin{claim}\label{claim:small}
	%	We have $H\subset \ker\tau_i $.
	%	\end{claim}
%	\begin{proof}
	%	By the construction of $\tau_i$, we know that $[\tau_i]\in M(L)$, where $L:=\overline{\bF_{q_i}((t))}$. By the definition of $M$, there exists $\sigma:\pi_1(X_0)\to \GL_{N}(K)$ such that $K$ is some field with ${\rm char}\, K=p$ and $[\mu^*\sigma]=[\tau_i]$. Since $\tau_i$ is semisimple, it follows that $H\subset \ker \mu^*\sigma\subset \ker\tau_i$. 
	%	\end{proof}

By \cref{lem:20240218}, there exists a reductive representation $\btau:\pi_1(X)\to \GL_{N'}(K')$ with $K'$ a local field of characteristic $p$ such that the Katzarkov-Eyssidieux reduction map $s_{\btau}:X\to S_{\btau}$ of $\btau$ coincides with $s_M:X\to S_M$.   
Moreover by \cref{lem:20240218} and \cref{claim:20240229}, we have
\begin{equation}\label{eqn:20240229}
	H\subset \mathrm{ker}(\btau).
\end{equation}

%We then further use the above isogeny between reductive groups with product of semisimple one and tori. Therefore, we may assume that  $s_{T_1},\ldots,s_{T_k}$ are induced by a family of representations  $ \{\tau_i:\pi_1(X) \to \GL_N\big( \overline{\bF_{q_i}((t))}\big)\}_{i=1,\ldots,\ell}$ and $ \{\sigma_j:\pi_1(X) \to \GL_N\big( \overline{\bF_{q'_j}((t))}\big)\}_{j=1,\ldots,m}$ such that the Zariski closure of $\tau_i$ is semisimple and the Zariski closure of $\sigma_j$ is tori.  
%Let $\pi:\xsp\to X$ be the spectral covering of  $ \{\tau_i\}_{i=1,\ldots,\ell}$  and  be their  spectral forms.  They will be called the spectral coverings and spectral forms with respect to $T_1,\ldots,T_k\subset  R_{\bF_{p}}$. 

For	the spectral one forms $\{\eta_1,\ldots,\eta_k\}\subset H^0(\xsp, \pi^*\Omega_X^1)$ associated with the representation \( \btau : \pi_1(X) \to \GL_{N'}(K') \), we define the \emph{rank} as follows: if there exist   \( i \) and \( j \) such that \( \eta_i \wedge \eta_j \neq 0 \), then we say that the rank is \( 2 \); otherwise, the rank is \( 1 \). 

\medspace

\noindent {\it Case 1: The spectral 1-forms have rank 2}.    Assume that the spectral one forms $$\{\eta_1,\ldots,\eta_k\}\subset H^0(\xsp, \pi^*\Omega_X^1)$$ with respect to $\btau:\pi_1(X)\to \GL_{N'}(K')$ have rank 2.

Let $s_\btau:X\to S_\btau$ be the reduction map of $\btau$, which coincides with $s_M:X\to S_M$ 
(and thus $\mu:X\to X_0$).  
Let $T_\btau$ be the canonical current on $X_0=S_\btau$ defined in \cref{def:canonical}.  %Let us denote by $\cT_{T}:=e_T^*T_\btau$, which is a positive closed (1,1)-current on ${\rm Sh}_M(X)$ with continuous potential.  
%	We note that $\mu:X\to X_0$ is the Stein factorization of $(s_\tau_1,\ldotsm)$   Denote by $X_0\to S_{\tau_1}\times\cdots\times S_{\tau_k}$  the resulting finite morphism.  
%	By the construction of the canonical current $\cT_{T_i}$, we conclude that  for the positive closed (1,1)-current $\sum_{i=1}^{k}\cT_{T_i}$ on ${\rm Sh}_M(X)=X_0$, 
Then $T_\btau$ is strictly positive at general points since the spectral forms associated with $\btau$ has rank 2. Since $T_\btau$ has continuous local potentials, it follows that $T_\btau$  is a  big and nef class.   

On the other hand, by  \cref{lem:strictly nef}, we conclude that $T_\btau$ is strictly nef.  	We now apply a theorem by Demailly-P\u{a}un \cite{DP04} (or \cite{Lam99}) to  conclude that $\{T_\btau\}$ is a K\"ahler class.

According to \cref{prop:Eys} and \eqref{eqn:20240229}, there exists a continuous plurisubharmonic function  $\phi:\widetilde{X}_0\to \bR_{\geq 0}$ such that $\hess \phi\geq \pi_0^*T_{\btau}$.
We can apply \cref{prop:stein} to conclude that $\widetilde{X}_0$ is a Stein space. Note that any unramified covering of  a Stein space is Stein. Therefore, the universal covering of $X_0$ is a Stein space.   %Note that there exists a proper holomorphic fibration $\widetilde{X}_\varrho\to \widetilde{X}_0$, where $\widetilde{X}_0$ is the universal cover of $X_0$. Then  $f:\widetilde{X}_\varrho\to S$ factors through a proper holomorphic fibration $f_0:\widetilde{X}_0\to S$.  Consequently, $\widetilde{X}_0$ is holomorphically convex by the Cartan-Remmert theorem.  

\medspace

%By the construction of $s_{T_{k+1}}$, there exists a   reductive representation $\tau:\pi_1(X)\to \GL_{N}(\overline{\bF_q((t))})$ such that $s_\tau(T)$ is not a point. 

\noindent {\it Case 2: The spectral 1-forms have rank 1}:   
Assume that the spectral covering $\xsp\to X$  with respect to $\btau:\pi_1(X)\to \GL_{N'}(K')$ has rank 1; i.e., $\eta_i\wedge\eta_j\equiv 0$ for every $\eta_i$ and $\eta_j$, where $\{\eta_1,\ldots,\eta_\ell\}\subset H^0(\xsp,\pi^*\Omega_X^1)$  is the spectral forms associated with $\btau:\pi_1(X)\to \GL_{N'}(K')$, and $\pi:\xsp\to X$ is  the spectral covering  associated with $\btau:\pi_1(X)\to \GL_{N'}(K')$.

\medspace

\noindent {\it Case  2.1: The dimension of spectral 1-forms is at least two}:  
Suppose  
$
{\rm dim}_\bC{\rm Span}\{\eta_1,\ldots,\eta_\ell\}\geq 2.
$ 

\medspace

Without loss of generality, we may assume that $\eta_1\wedge\eta_2\equiv 0$ and $\eta_1\not\in \{\bC\eta_2\}$. 
According to the Castelnuovo-De Franchis theorem (cf. \cite[Theorem 2.7]{ABCK96}), there exists a proper fibration $h:\xsp\to C$ over a smooth projective curve $C$ such that $\{\eta_1,\eta_2\}\subset h^*H^0(C,\Omega_C^1)$.  
Since $s_M=\mu$ 
is birational,  we can choose a general fiber $F$ of $h$, which is irreducible and  such that $s_{M}\circ\pi(F)=s_\btau\circ\pi(F)$  
is not a point. 
%Here $e_{T}:X_0\to S_{T}$ is the natural morphism.  %Thanks to \cref{lem:strictly nef}, 
%	\begin{align} \label{eq:nef}
	%	\{\cT_{T}\}\cdot {\rm sh}_M\circ\pi(F)>0
	%\end{align}  
	There exists some $i$ such that $\eta_i|_{F}\neq 0$ (cf. \Cref{lem:snef}).
	Given that  $\eta_1|_{F}\equiv 0$, this implies that $\eta_i\wedge\eta_1\neq 0$. It contradicts with our assumption that the spectral $1$-forms have rank 1.  
	Therefore, this case cannot occur. 
	
	\medspace
	
	\noindent {\it Case  2.2: The dimension of spectral 1-forms is 1}:  
	We have $
	{\rm dim}_\bC{\rm Span}\{\eta_1,\ldots,\eta_\ell\}= 1.
	$ 
	
	%By the construction of ${\rm sh}_M$ in \cref{thm:Sha1}, we can find irreducible curves  $T_1,\ldots,T_k\subset R$ defined over $\overline{\bF_p}$ such that ${\rm sh}_M$ is the Stein factorization of 
	%	$$
	%	(s_{T_1},\ldots,s_{T_k}):X\to S_{T_1}\times\cdots\times S_{T_k}.
	%$$  Let  $\btau:= \{\tau_i:\pi_1(X) \to \GL_N\big( \bF_{q_i}((t))\big)\}_{i=1,\ldots,\ell}$ be a family of representations that generate  the reduction maps $s_{T_1},\ldots,s_{T_k}$ as defined in \cref{sec:factor}. Then its reduction map $s_\btau:X\to S_\btau$ is the Shafarevich morphism ${\rm sh}_M$.  By \cref{lem:same bound} we can replace each $\tau_i$  by its semisimplification such that the reduction map $s_\btau$ remains unchanged.  
	
	Let  $G$ be the Zariski closure of $\btau(\pi_1(X))$, which is reductive. Consider the isogeny
	$g:G\to G/Z\times G/\cD G$ 
	where $Z$ is the central torus of $G$ and $\cD G$ is the derived group of $G$. 
	As a result, $G':=G/Z$  is semisimple and $G'':=G/\cD G$ is a torus. 
	Let
	$
	\tau':\pi_1(X)\to G'(\overline{K'}) 
	$ 
	be the composite of $\btau$ with the projection $G\to G'$, and  
	$
	\tau'':\pi_1(X)\to   G''(\overline{K'}) 
	$ 
	be the composite of $\btau$ with the projection $G\to G''$.  Then $\tau'$ and $\tau''$ are both Zariski dense representations.

	%	Next, we consider the partial Albanese morphism $a:\xsp\to A$ induced by $\eta_1$.  If $\dim a(Y)=1$, then the Stein factorization $h:\xsp\to C$ of $a$ is a proper holomorphic fibration   over a smooth projective curve $C$ such that $\eta_1\in h^*H^0(C,\Omega_C^1)$.  We are now in  a situation akin to Case 2.1, and we can apply the same arguments to reach a contradiction. Hence $\dim a(Y)=2$. 
	
	Let $\nu:Y\to \xsp$ be a desingularization and denote $\eta:=\nu^*\eta_1$. Consider the partial Albanese morphism $a:Y\to A$ induced by $\eta$.  Then there exists a one form $\eta'\in H^0(A,\Omega_A^1)$ such that $a^*\eta'=\eta$.  If $\dim a(Y)=1$, then the Stein factorization $h:Y\to C$ of $a$ is a proper holomorphic fibration   over a smooth projective curve $C$ such that $\eta_1\in h^*H^0(C,\Omega_C^1)$.  We are now in  a situation akin to Case 2.1, and we can apply the same arguments to reach a contradiction.  Hence $\dim a(Y)=2$.   Let $\pi_A:\widetilde{A}\to A$ denote the universal covering map. We denote by $Y':=Y\times_{\widetilde{A}}A$  a connected component of the fiber product and let $\pi':Y'\to Y$ be the induced   \'etale cover.   
	It's worth noting that $\pi'^*\eta$ is  exact. Consequently, we can define the following holomorphic map: 
	\begin{align*}
		h: Y'&\to \bC\\
		y&\mapsto \int_{y_0}^{y}\pi'^*\eta.
	\end{align*} We then have the following commutative diagram:
	\begin{equation*}
		\begin{tikzcd}
			\widetilde{Y}\arrow[r, "p"]\arrow[rr, bend left=30, "\pi_Y"]&	Y'\arrow[dd, bend right=30, "h"']\arrow[r, "\pi'"] \arrow[d] & Y\arrow[d, "a"]\\
			&	\widetilde{A}\arrow[d] \arrow[r, "\pi_A"]& A\\
			&		\bC
		\end{tikzcd}
	\end{equation*} 
	The   holomorphic map $\widetilde{A}\to \bC$ in the above diagram is  defined by the linear 1-form $\pi_A^*\eta'$ on $\widetilde{A}$.   
	By Simpson's Lefschetz theorem \cite{Sim93}, for any $t\in \bC$, $h^{-1}(t)$ is connected and $\pi_1(h^{-1}(t))\to \pi_1(Y')$ is surjective.  
	By  
	definition of $h$,  $\pi_Y^*\eta|_{Z}\equiv 0$ where $Z$ is any connected component of $p^{-1}(h^{-1}(t))$. 
	Here $p:\widetilde{Y}\to Y'$ is the natural covering map. 
	
	Consider the  Zariski dense  representation $\tau':\pi_1(X) \to G'\big( \overline{K'}\big)$ as defined previously. 
	Let  $L$ be a finite extension of $K'$ such that $G'$ is defined on $L$ and  $\tau':\pi_1(X)\to G'(L)$. 
	We denote by $\sigma:\pi_1(Y)\to G'\big(L\big)$ the pullback of $\tau'$ via the morphism $Y\to X$.   
	The existence of a $\sigma$-equivariant harmonic mapping $u:\widetilde{Y}\to \Delta(G')$ is guaranteed by \cite{GS92}, where $\Delta(G')$ is the Bruhat-Tits building of $G'$.
	
	We note that $\pi_Y^*\eta$ is the (1,0)-part of the complexified differential of the harmonic mapping $u$ at general points of $\widetilde{Y}$, with $\pi_Y:\widetilde{Y}\to Y$ denoting the universal covering. 
	For  any connected component $Z$ of $p^{-1}(h^{-1}(t))$ for a general $t\in \bC$, since $\pi_Y^*\eta|_{Z}\equiv 0$,   and  all the spectral forms are assumed to be $\bC$-linearly equivalent, it follows that $u(Z)$ is constant. 
	Since $u$ is $\sigma$-equivariant, it follows that $\pi'^*\sigma({\rm Im}[\pi_1(h^{-1}(t))\to \pi_1(Y')])$ is contained in the subgroup of $G'(L)$  fixing the point $u(Z)$. 
	Recall that $\pi_1(h^{-1}(t))\to \pi_1(Y')$  is surjective. 
	Hence $\pi'^*\sigma(\pi_1(Y'))$ is a bounded subgroup of $G'(L)$. 	
	Additionally, note that $$\cD \pi_1(Y)\subset {\rm Im}[\pi_1(Y')\to \pi_1(Y)],$$ and it follows that $\sigma(\cD \pi_1(Y))$ is bounded.  	
	Since $\tau'$ is Zariski dense, and ${\rm Im}[\pi_1(Y)\to \pi_1(X)]$ is a finite index subgroup of $\pi_1(X)$,  according to \cref{claim:same closure}   the Zariski closure  of $\sigma(\pi_1(Y))$ contains the identity component of $G'$, and it is also semisimple.     
	We apply \cref{lem:BT}  to conclude  that $\sigma( \pi_1(Y))$ is bounded.

	Since $\sigma( \pi_1(Y))$  is a finite index subgroup of $\tau'(\pi_1(X))$, it follows that $\tau'(\pi_1(X))$ is also bounded. 
	Then the reduction map $s_{\tau'}$ is the constant map.  
	This implies that the reduction map $s_{\tau}$ is identified with $s_{\tau''}$.
	Recall that $G''$ is a tori.  
	By \cite[Step 6 in Proof of Theorem H]{CDY22}, we know that there exists a morphism $a_X:X\to A_X$ with $A_X$ an abelian variety such that $s_{\tau''}$ is the Stein factorization of $a_X$.

	\begin{claim}
		There exists a morphism $b:X_0\to A_X$ such that $b\circ\mu=a_X$. 
	\end{claim}
	\begin{proof} 
		Recall that  $[\tau]\in M(K')$. 
		Let $F$ be  any fiber of the birational morphism $\mu:X\to X_0$.   Then  $\tau({\rm Im}[\pi_1(F)\to \pi_1(X)])$ is finite since $s_M:X\to S_M$ 
		coincides with $\mu$. 
		It follows that  
		$\tau({\rm Im}[\pi_1(F)\to \pi_1(X)])$ is a finite group. %By our assumption for $\varrho$, %it follows that $\tau_i({\rm Im}[\pi_1(\widetilde{X}_\varrho)\to \pi_1(X)])$ is finite.
		By the definition of $\tau''$, we conclude that $\tau''({\rm Im}[\pi_1(F)\to \pi_1(X)])$ is also finite, and is thus bounded.  According to the property of the reduction map \cref{thm:KZ}, $s_{\tau''}(F)$ is a point. Hence there exists a morphism $b:X_0\to A$ such that $b\circ\mu=a_X$.  
	\end{proof}
	Consider the map $b:X_0\to A_X$. Note that this map is finite. We apply \cref{lem:Stein}  to conclude that the universal covering of $X_0$ is a Stein space.  
\end{proof}

\begin{rem}
	The proof of  \cref{thm:Stein} utilizes techniques similar to those used in the proof of \cite[Theorem C]{DY23}. In order to extend \cref{thm:Stein} to any projective normal variety, we have to establish  Simpson's theory \cite{Sim93b} on \emph{absolutely constructible subsets} for character varieties of representations  in  positive characteristic. We plan to explore this problem in our future work.
\end{rem} 
We recall the following definition by Campana \cite{Cam94}.
\begin{dfn}[$\Gamma$-dimension]\label{def:Gamma}
	Let $X$ be a projective normal variety. The $\Gamma$-dimension of $X$ is  defined to be $\dim {\rm Sh}(X)$, where   ${\rm sh}_X:X\dashrightarrow {\rm Sh}(X)$ is the Shafarevich map constructed by Campana \cite{Cam94} and Koll\'ar \cite{Kol93}.
\end{dfn}

\begin{thm}\label{thm:convexity}
	Let $X$ be a projective normal variety and let $\varrho:\pi_1(X)\to \GL_{N}(K)$ be a faithful representation where $K$ is a field of positive characteristic. If the $\Gamma$-dimension of $X$  is at most two (e.g. when $\dim X\leq 2$), then the universal covering $\widetilde{X}$  of $X$ is holomorphically convex.  
\end{thm}

\begin{proof}
	By \cref{prop:factor},  after we replace $X$ by a suitable finite \'etale cover, there exists  a large representation $\tau:\pi_1({\rm Sh}_{\varrho}( {X}))\to \GL_{N}(K)$ such that $({\rm sh}_{ \varrho})^*\tau= \varrho$, where  ${\rm sh}_{ \varrho}: {X}\to {\rm Sh}_{ \varrho}( {X})$ is the Shafarevich morphism of $ \varrho$. 
	Since $\varrho$ is faithful, it follows that $\tau$ is also faithful since the homomorphism $({\rm sh}_\varrho)_*:\pi_1(X)\to \pi_1({\rm sh}_\varrho(X))$ is surjective.  
	As we assume that the $\Gamma$-dimension of $X$  is at most two,  therefore, $\dim {\rm Sh}_{ \varrho}( {X})\leq 2$. 
	We apply \cref{thm:Stein} to conclude that  the universal covering $S$ of  ${\rm Sh}_{ \varrho}( {X})$ 
	is Stein.  
	Note that $({\rm sh}_\varrho)_*:\pi_1(X)\to \pi_1({\rm sh}_\varrho(X))$  is an isomorphism. 
	Hence there exists a proper holomorphic fibration $\widetilde{X}\to S$ between the universal coverings of $X$ and ${\rm Sh}_\varrho(X)$ that lifts ${\rm sh}_\varrho$.  
	It follows that $\widetilde{X}$ is holomorphically convex. 
\end{proof}

\subsection*{Acknowledgment.}
We would like to thank Michel Brion, Beno\^it Claudon,  Philippe Eyssidieux and Andreas H\"oring  for very helpful discussions. We also thank Beno\^it Cadorel and Yuan Liu for  reading    the paper and their  helpful remarks.  Finally, we sincerely thank the referees for  very careful readings and helpful remarks, in particular for pointing out a simpler proof of \cref{main:kollar} (see \cref{rem:referee}). Y.D. acknowledges support from the ANR grant Karmapolis (ANR-21-CE40-0010).  
K.Y. acknowledges support from JSPS Grant-in-Aid for Scientific Research (C) 22K03286.

 % \bibliography{biblio}
% \bibliographystyle{ssmfalpha}
  
 \newcommand{\etalchar}[1]{$^{#1}$}

\end{document}